 \documentclass[aap]{imsart}
\RequirePackage[OT1]{fontenc}
\RequirePackage{amsthm,amsmath}
\RequirePackage[numbers]{natbib}
\RequirePackage[colorlinks,citecolor=blue,urlcolor=blue]{hyperref}

\arxiv{arXiv:1910.12715}

\startlocaldefs
\usepackage[latin1]{inputenc}
\usepackage[english]{babel}
\usepackage[T1]{fontenc}
\usepackage[mathlines]{lineno}
\usepackage{lmodern}
\usepackage{graphicx, mathrsfs}
\usepackage[monochrome]{xcolor}
\usepackage{changes}
\usepackage{amsfonts}
\usepackage{stmaryrd}
\usepackage[nice]{nicefrac}
\usepackage{amsmath,amsthm,amssymb}
\usepackage{enumerate}
\usepackage[top=2.54cm, bottom=2.54cm, left=2.54cm , right=2.54cm]{geometry}
\usepackage{caption}
\usepackage{xifthen} 
\usepackage{wrapfig}
\usepackage{comment}
\usepackage{float}
\usepackage{ifthen}
\usepackage{soul}
\usepackage{tikz}
\usepackage{caption}
\usetikzlibrary{arrows}
\usetikzlibrary{graphs,graphs.standard}
\usetikzlibrary{positioning,arrows.meta}
\usetikzlibrary{shapes.multipart}
\usetikzlibrary{arrows}
\usetikzlibrary{automata}
\usetikzlibrary{calc}
\definecolor{arrowblue}{RGB}{0,0,0}  
\newcommand\ImageNode[3][]{
  \node[draw=arrowblue!80!black,line width=1pt,#1] (#2) {\includegraphics[width=3.1cm,height=3.1cm]{#3}};
}

\usepackage{caption}
\newtheorem{thm}{Theorem}[section]
\newtheorem{df}[thm]{Definition}
\newtheorem{lem}[thm]{Lemma}

\newtheorem{prop}[thm]{Proposition}
\newtheorem{cor}[thm]{Corollary}

\newtheorem{rmq}{Remark}

\newcommand{\bs}[1]{\ensuremath{\boldsymbol{#1}}}
\newcommand{\Supp}{\mathtt{Supp}}

\newcommand{\R}{\mathbb R}
\newcommand{\N}{\mathbb N}
\newcommand{\E}[2][]{\ensuremath{\mathbb{E}_{#1} \left[#2 \right]}}
\newcommand{\Ex}[2]{\ensuremath{\mathbb{E} \left[#2 \right]}}
\newcommand{\Prob}[2][]{\ensuremath{\mathbb{P}_{#1} \left(#2 \right)}}
\newcommand{\PIdx}[2]{\ensuremath{\mathbb{P} \left(#2 \right)}}

\newcommand{\eps}{\varepsilon}

\newcommand{\D}[2]{N_{#1} (#2)}

\newcommand{\dgl}[2]{\hat{d}_{#2}(#1)}
\newcommand{\Ind}{\mathcal{I}}
\newcommand{\Jind}{\mathcal{J}}
\newcommand{\St}[2]{\mathrm{st}_{#1}(#2)}

\newcommand{\cI}{\mathcal {I}}
\newcommand{\cE}{\mathcal {E}}
\newcommand{\cD}{\mathcal {D}}

\newcommand{\bS}{\mathbb {S}}
\newcommand{\cS}{\mathcal {S}}

\newcommand{\Dtwo}{\mathcal{B}}
\newcommand{\Dtwoind}{B}
\newcommand{\Xtwo}{\mathcal X}

\newcommand{\w}{\omega}
\newcommand{\lf}{\leftarrow}

\newcommand{\pb}{\mathbb{P}}
\newcommand{\Pxx}[2]{\ensuremath{\mathbb{P}_{#1}\left(#2 \right)}}

\newcommand{\m}{\setminus}
\newcommand{\tns}{\varotimes}

\newcommand{\cC}{\mathcal {C}}
\newcommand{\cCd}{\mathcal {C}_{d-1}}

\newcommand{\cCdd}{\mathcal {C}_{d-2}}

\newcommand{\dd}{\mathrm d}

\newcommand{\sss}{\scriptscriptstyle}

\newcommand{\Gamm}[1]{\ensuremath{\Gamma\left(#1 \right)}}
\newcommand{\K}{\mathcal{K}}
\newcommand{\p}{p}
\newcommand{\wmax}{w^*}
\newcommand{\J}{J}
\newcommand{\Jo}{J'}
\newcommand{\s}{\varsigma}
\newcommand{\me}{m}

\newcommand{\C}[3]{
\ifthenelse{\isempty{#3}}
{\ifthenelse{\isempty{#2}}
{{\mathcal K}_{#1}} 
{{\mathcal K}_{#1} \left(#2 \right)}
} 
{\ifthenelse{\isempty{#2}}
{{\mathcal K}_{#1}^{\sss (#3)}} 
{{\mathcal K}^{(\sss #3)} \left(#2 \right)}
}
}

\newboolean{second_moment_details} 
\setboolean{second_moment_details}{true}

\newboolean{finite_case} 
\setboolean{finite_case}{true}   

\newboolean{graphs} 
\setboolean{graphs}{false}   

\newboolean{wrrt} 
\setboolean{wrrt}{false}   

\newcommand{\ti}[1]{{\color{blue} TI Comments: #1}}

\newcommand{\meaneree}[1]{{\color{black} #1}}

\newcommand{\miscchange}[1]{{\color{black} #1}}

\endlocaldefs
\begin{document}
\renewcommand{\abstractname}{} 
\begin{frontmatter}
\title{Dynamical Models for Random Simplicial Complexes}
\runtitle{Dynamical Models for Random Simplicial Complexes}

\begin{aug}
 \author{\fnms{Nikolaos} \snm{Fountoulakis},
 \thanksref{t1, m1}}
 \author{\fnms{Tejas} \snm{Iyer,}
\thanksref{m1}}
\author{\fnms{C\'{e}cile} \snm{Mailler}
\thanksref{t2, m2}}
\\
\and 
\author{\fnms{Henning} \snm{Sulzbach}. \thanksref{m1}}
\thankstext{t1}{Research supported by the EPSRC, grant 
EP/P026729/1, and the Alan Turing Institute, grant EP/N510129/1} 
\thankstext{t2}{Research supported by the EPSRC fellowship EP/R022186/1.}
 \affiliation{University of Birmingham \thanksmark{m1} and University of Bath. \thanksmark{m2}}
\end{aug}
\runauthor{Fountoulakis, Iyer, Mailler and Sulzbach.}

\begin{abstract} We study a general model of random dynamical simplicial complexes and derive a formula for the asymptotic degree distribution. This asymptotic formula generalises results for a number of existing models, including random Apollonian networks and the weighted random recursive tree. It also confirms results on the scale-free nature of Complex Quantum Network Manifolds in dimensions $d > 2$, and special types of Network Geometry with Flavour models studied in the physics literature by Bianconi and Rahmede [\textit{Sci. Rep.} \textbf{5}, 13979 (2015) and \textit{Phys. Rev. E} \textbf{93}, 032315 (2016)].
\end{abstract}



\end{frontmatter}

\section{Introduction}
\normalem
Complex networks are well known for their non-trivial features, such as being scale-free, (having degree distribution whose tail follows a power law), and forming \emph{small} 
or~\emph{ultra-small} worlds  (meaning that the diameter or typical distances between two random 
vertices is logarithmic or doubly logarithmic, respectively). As a result, numerous models have been developed to describe these networks, including the \emph{preferential attachment model} introduced in this context
 by Barab\'asi and Albert~\cite{Barabasi509} and defined and studied rigorously by Bollob\'as, Riordan, Spencer and 
Tusn\'ady~\cite{bollobaspreferential}. \miscchange{This model describes a mechanism for the growth of a complex network which realises the rich-get-richer postulate: when a new vertex joins the network it is more likely to attach to vertices that are popular, that is, having high degree.}

\miscchange{Preferential attachment models had also been considered earlier within the context of \emph{random evolving recursive trees}; which may be described as growing labelled trees where vertices arrive one at a time and connect to an existing vertex chosen randomly according to a certain probability distribution. 
In the \emph{ordered recursive tree}, introduced by
Prodinger and Urbanek in \cite{prodingerurbanek} and studied and rediscovered under various guises (under the name \emph{nonuniform recursive trees} by Szyma\'{n}ski in \cite{szymanski}, random \emph{plane oriented recursive trees} in \cite{mahmoud92,mahmoudetal93}, random \emph{heap ordered recursive trees} \cite{chen1994} and \emph{scale-free trees} \cite{vanderhofstad2016}), existing vertices are chosen with probability proportional to their degree, and thus according to a the preferential attachment mechanism. Another type of randomly evolving recursive tree is the \emph{uniform recursive tree}, introduced by Na and Rapoport in \cite{narapoport}; here existing vertices are chosen uniformly at random. 
In \cite{AllPan2007}, Kuba and Panholzer derive the degree distribution in both these trees and another type of recursive tree known as a \emph{binary increasing tree}.

These models were extended by Bianconi and Barab\'asi~\cite{bianconibarabasi2001} who proposed an inhomogeneous recursive tree model in which each vertex has its own \emph{fitness}. In their model, a newly arrived vertex attaches to an existing vertex selected with probability proportional to the product of its fitness and its degree (so that the popularity of a vertex is moderated by its fitness). 
The most significant difference between the Bianconi-Barab\'asi model and the Barab\'asi-Albert model 
is the emergence of \emph{condensation} (observed in~\cite{bianconibarabasi2001} and proved rigorously by 
Borgs et al.~\cite{Borgs2007} and also later, in a more general context, by Dereich and Ortgiese~\cite{dereich_ortgiese_2014}). 
This means, under certain conditions on the distribution of fitnesses, a small (sub-linear) number 
of vertices with `high' fitness accumulate a positive fraction of the total number of edges in the graph. A number of other variations of the preferential attachment model have been proposed and studied, and for a more comprehensive overview see~\cite{vanderhofstad2016} and~\cite{bhamidi}.

The above models create random trees. However, they may  be extended so that newly arriving vertices make $m\geq 1$ new connections. One way of doing this is to consider $m$ copies of the new vertex each throwing one new connection to the existing network and then identifying them as one vertex (hence forming a multigraph). See Chapter 8 in~\cite{vanderhofstad2016} for a detailed description. 

\subsubsection*{Higher dimensional preferential attachment mechanisms}
\miscchange{All these models are 1-dimensional in the sense that newly arriving vertices are attached to single vertices. Our motivation is to consider attachment mechanisms in which newly arriving vertices join \emph{groups} of vertices, where the attachment takes into account intrinsic features of a group of vertices, and thus encodes more complexity.}
\emph{Simplicial complexes} are a natural choice for incorporating this \emph{higher dimensional} complexity at a local level. Furthermore, complex networks appearing in applications are typically \emph{locally dense}: that is, although they form sparse graphs, the neighbourhood of a typical vertex is dense. This is usually measured by the \emph{clustering coefficient}. The classic preferential attachment models do not satisfy this, as the graph that is formed is tree-like within a short distance from a randomly chosen vertex. However, this `local density' arises naturally from the fact that simplicial complexes are \emph{downwards closed}.
Hence, a preferential attachment model which involves higher order interactions encapsulates these features naturally.
Additionally, (random) simplicial complexes have already been used in applications such as topological data analysis (see, for example, \cite{gunnarcarlsson}), and recent theories of quantum gravity (see, for example,~\cite{agistein}).}

\begin{df} \label{df:simp-comp}
An \emph{(abstract) simplicial complex} $\K$ is a family of sets that is \emph{downwards
 closed}: for any set $\sigma \in \K$, if $\sigma' \subseteq \sigma$, then $\sigma' \in \K$. Any family of sets may be turned into a simplicial complex in the natural way by taking the \emph{downwards closure}, that is, by adding the minimum number of subsets to make the family downwards closed. 
\end{df}
An element $\sigma \in \K$ is called a \emph{face}, and we say that $\sigma$ has \emph{dimension} $s$ if it has cardinality $s+1$ (we also call it an $s$-\emph{face} or an $s$-\emph{simplex}). For $s \in \mathbb{N} \cup \{0, -1\}$, we denote by $\K^{(s)}$ the subset of $\K$ consisting of all its $s$-faces. The \emph{dimension} of $\K$ is defined to be the maximum $s$ such that $\K^{(s)}$ is non-empty (if $\K = \varnothing$ we say it has dimension $-1$). We call the $0$-faces of $\K$ its \emph{vertices}, and $\K^{(0)}$ its \emph{vertex set}. Finally, for a vertex $v \in \K^{(0)}$ we define its \emph{degree} by $\deg{(v)} : = \left|\left\{\sigma \in \K^{(1)}: v \in \sigma\right\} \right|$ (the degree in the usual sense with regards to the simple graph underlying the complex).  

One model that realises higher order interactions is the Random Apollonian Network. It was first introduced
in \cite{apollonianinitial} and independently in \cite{DoyeMassen} as a model for complex networks and was subsequently extended by Zhang et al.~\cite{ar:ZRC2006, ar:ZRZ2006}.
Here, in dimension $d$, we begin with a $d$-simplex, 
all of whose $(d-1)$-dimensional faces are \emph{active}.
In each step, an active $(d-1)$-dimensional face is selected 
uniformly at random and $d$ new $(d-1)$-faces are formed by the union of a newcoming vertex and each subset of the selected face of size $d-1$. Subsequently, the selected $(d-1)$-dimensional face is \emph{deactivated}, so that the number of active $(d-1)$-faces in the complex increases by $d-1$ at each step.
As each of the $d$ new $(d-1)$-faces, together with the selected face~$\sigma$ form a $d$-face, we can interpret this step geometrically as a $d$-face being `glued' onto the face~$\sigma$, with the set of active faces being the boundary of the complex (see Figure~\ref{fig:modelbdim3} below).
Note that, when a node $v$ enters the network, its degree is equal to $d$ and the number of active faces containing it is equal to $d$. Moreover, every time an active face containing $v$ is selected, the degree of $v$ increases by one and the number of active faces containing $v$ increases by $d-2$.
Therefore, the number of active faces containing a given vertex $v$ is $(d - 2) \deg(v)-d(d-3)$.  Thus, if $d > 2$ the number of active faces containing a vertex is proportional to its degree, and hence this model gives rise to a \emph{preferential attachment mechanism}. In~\cite{apollonian_hungarians_1} and independently in \cite{apollonian_frieze_2}, the authors determined that the degree distribution of this model for $d > 2$, gives rise to a power law with exponent $\tau = \frac{2d-3}{d-2} = 2 + \frac{1}{d-2}$.\footnote{Note that often in the literature surrounding Apollonian networks, rather than using the dimension of the initial simplex, authors use the number of vertices in an `active' face as the parameter of the model. Thus the Apollonian network with parameter $d$ is the same as the Apollonian network in dimension $d - 1$.}
For $d = 3$ the same model has been studied under the name \emph{random stack-triangulations} by Albenque and Marckert in \cite{albenque2008}, where they proved that the sequence of complexes with graph distance metric rescaled by $\sqrt{n}$ considered as a compact metric space converges in the Gromov-Hausdorff topology to the continuum  random tree of Aldous \cite{aldous1993}. 
\\\\
\miscchange{In the Apollonian network the choice among the active $(d-1)$-faces is uniform. In particular, there is no preferential attachment mechanism directly associated with the evolution of the vertices. This motivates us to define and study mechanisms in which these high-dimensional sub-structures are \emph{inhomogeneous} and have some intrinsic fitness which is a function of the fitness of their members.} 

\miscchange{Specific implementations of this idea were introduced by Bianconi, Rahmede, and other co-authors motivated by applications in physics (\cite{bianconi2015ComplexQN_1,bianconirahmedezhihao_2, bianconicourtney_3,bianconi2016NetworkGW_4, bianconirahmede2017_5,bianconi_selection_mechanisms}). 
(For example, random triangulations have been considered in the context of quantum gravity~\cite{agistein}}.)
The model of \emph{Complex Quantum Network Manifolds (CQNMs)} described in \cite{bianconi2015ComplexQN_1} in dimension $d > 1$ can be viewed as a generalisation of the Random Apollonian Network, where vertices are equipped with independent, identically distributed (i.i.d.) weights (called \emph{energies} in this context) and each $(d-1)$-face $\sigma$ of the evolving $d$-dimensional 
simplicial complex has energy $\epsilon_{\sigma}$ given by the sum of the energies of its vertices. The simplicial complex evolves in the same way as the Random Apollonian network, with the only difference being that at each time-step, a new vertex selects an active $(d-1)$-face $\sigma$ with probability proportional to $e^{-\beta \epsilon_\sigma}$ (where $\beta \geq 0$ is a fixed constant, usually interpreted as the ``inverse temperature'') instead of uniformly at random. 
In \cite{bianconi2015ComplexQN_1}, the authors argue that when $d=2$ the underlying graph
has degree distribution with exponential tail whilst, when $d\geq 3$ the degree distribution follows a power law with exponent that depends on $d, \beta$ and the 
distribution of the weights. In this paper, we verify a rigorous version of this result when the energies are bounded (see Subsection \ref{sec:tails}).

\miscchange{In \cite{bianconi2016NetworkGW_4}, Bianconi and Rahmede introduce a more general model called the \emph{network geometry with flavour (NGFs)}. The network geometry with flavour, in dimension $d$ and \emph{flavour} $s \in \{-1, 0, 1\}$ proceeds as follows. As before, vertices are equipped with i.i.d. \emph{energies} and each $(d-1)$-face $\sigma$ of the evolving $d$-dimensional 
simplicial complex has energy $\epsilon_{\sigma}$ which is equal to the sum of the energies of its vertices. At each time-step, a new vertex selects a $(d-1)$-face $\sigma$ with probability proportional to $e^{-\beta \epsilon_\sigma} \left(1 + s\deg_{d}{(\sigma)} - s\right)$, where $\beta \geq 0$ is a fixed constant. 
In the case $s = -1$, Bianconi and Rahmede~\cite{bianconi2015ComplexQN_1} argue that when $d=2$ the underlying skeleton graph
has degree distribution with exponential tail, whilst when $d\geq 3$ the degree distribution obeys a power law, with an exponent that depends on $d$ as well as on $\beta$ and the 
distribution of the weights. 
Moreover, in \cite{bianconirahmedezhihao_2}, Bianconi, Rahmede and Wu argue that for $d=2$, if $s = -1$ the underlying skeleton graph 
has degree distribution with exponential tail, whilst if $s = 0$, the underlying skeleton graph has power law tails. We will prove weaker versions of both these results rigorously in this paper, in the sense that the degree distribution has a tail bounded from above and below by a power law. 
See Subsection \ref{sec:tails} for more details. }

There are many other models of random simplicial complexes, and for more details see the review articles by Kahle \cite{kahle2013} and Bobrowski and Kahle \cite{bobrowskikahle}.

\subsection{Definition of the model: the inhomogeneous dynamic simplicial complex}
In this paper, we consider a sequence of simplicial complexes $\left(\C{n}{}{}\right)_{n \geq 0}$ of fixed dimension $d \geq 0$.
The distribution of $\left(\C{n}{}{}\right)_{n \geq 0}$ depends on two parameters:
a symmetric \emph{fitness function} $f: [0,1]^d \to \R$, and a probability measure $\mu$ whose support is a subset of $[0,1]$ (in fact, we only require that $\mu$ only takes positive values and has bounded support; the assumption that the essential supremum is equal to $1$ can be made without loss of generality).

For all $n\geq 0$, $\mathcal K_{n+1}$ is obtained by adding one vertex labelled $n+1$ to $\mathcal K_n$ and assigned random weight sampled independently according to $\mu$. Using the weights of the vertices, we define the {\it fitness} of a face $\sigma$ as the image by $f$ of the vector $\omega(\sigma)$ of the weights of the vertices that belong to that face.
Abusing notation slightly, we sometimes write $f(\sigma)$ instead of $f(\omega(\sigma))$.
Since $f$ is assumed to be symmetric, the order of the coordinates of $\omega(\sigma)$ is not relevant. 

Motivated by this symmetry, for all $s\geq 0$, we view the {\it type} $\omega(\sigma)$ of an $s$-dimensional face $\sigma$ as an element of  $\cC_{s}:=[0,1]^{s+1}/\sim$, where $\sim$ denotes the equivalence relation where vectors are the same under permutation of their entries. 
Unless otherwise stated, we identify entries of $\cC_{s}$ with
the set $\{(x_0, \ldots, x_s) \in [0,1]^{s+1}: x_0 \leq \ldots \leq x_s\}$ and equip $\cC_s$ with the max-norm inherited from $[0,1]^{s+1}$.

We consider two versions of the model: Model \textbf{A} and Model \textbf{B}. These models are defined as follows: 
 first, let $\C{0}{}{}$ be an arbitrary $(d-1)$-dimensional simplicial complex, with finite vertex set $V_0 \subseteq -\mathbb N_0$ and each vertex assigned a fixed weight chosen from $\Supp(\mu)$ (in fact, we show that our limiting results do not depend on this choice of weights).
 Then, recursively for all $n \geq 0$:
 \begin{itemize}
 
 \item [(i)] Define the random empirical measure 
 \[
 \Pi_{n} = \sum_{\sigma \in \C{n}{}{d-1}}\delta_{\omega (\sigma)} 
 \] 
 on $\cCd$ and the associated probability measure on the set $\C{n}{}{d-1}$ of $(d-1)$-dimensional faces:  
 \begin{equation} \label{eq:partition-def}
 \hat{\Pi}_n = 
 \frac{1}{Z_n}\sum_{\sigma \in \C{n}{}{d-1}} f(\sigma) \delta_{\sigma}, 
 \quad\text{ where }
 Z_n:=\int_{\cCd} f (x) \dd \Pi_n (x).
 \end{equation}
 We call $Z_n$ the \emph{partition function} associated with the process $(\K_{n})_{n \geq 0}$ at time~$n$. 
\item [(ii)] Select a face $\sigma' = (\sigma'_{0}, \ldots, \sigma'_{d-1}) \in \C{n}{}{d-1}$ according to the measure $\hat{\Pi}_n$. 
 \item [(iii)] In both Models \textbf{A} and \textbf{B}, for each $\sigma'' \in \C{n}{}{d-2}$ such that $\sigma'' \subset \sigma'$, add the face $\sigma'' \cup \{n+1\}$ to $\C{n}{}{}$ (recall that $\C{n}{}{-1} = \varnothing$). Moreover, in Model \textbf{B} remove the set $\sigma'$ from $\C{n}{}{}$.  Then, take the downwards closure (recall Definition~\ref{df:simp-comp}) to form $\C{n+1}{}{}$. 
 \end{itemize}
 Note that, in Model \textbf{A} the existing faces always remain in the complex, whilst in Model \textbf{B} the selected face is removed at every step. We call step (iii) applied to a chosen face $\sigma'$ a \emph{subdivision} of $\sigma'$ by vertex $n + 1$ (equivalently we say $\sigma'$ has been \emph{subdivided} by vertex $n+1$).
\begin{figure}[H]
\captionsetup{width=.8\linewidth}
\begin{center} 
 \begin{tikzpicture}[scale=0.9, node distance=1cm]
\ImageNode[label={0:}]{A}{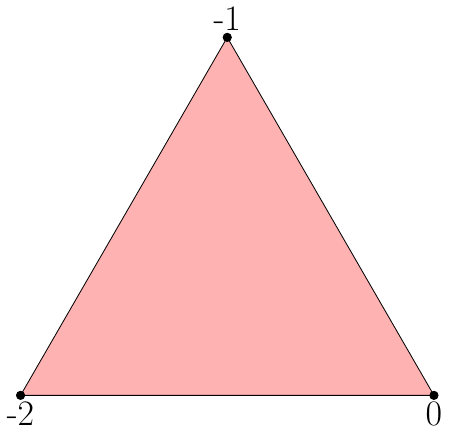}
\ImageNode[label={180:},right=of A]{B}{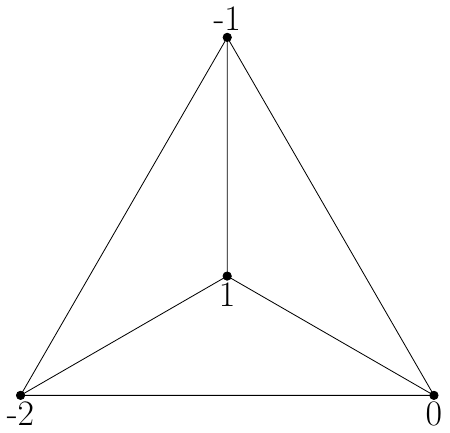}
\ImageNode[label={180:},right=of B]{C}{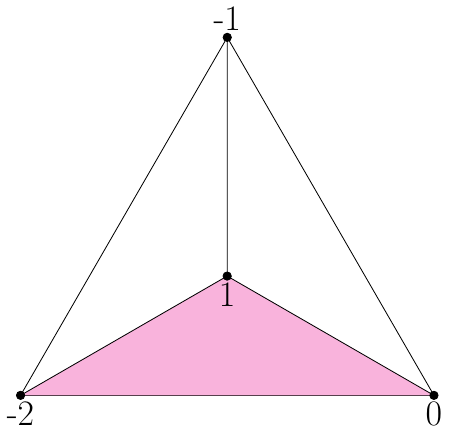}
\ImageNode[right=of C]{D}{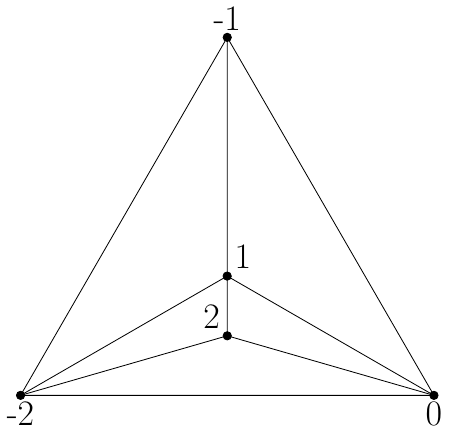}
\ImageNode[below=of D]{E}{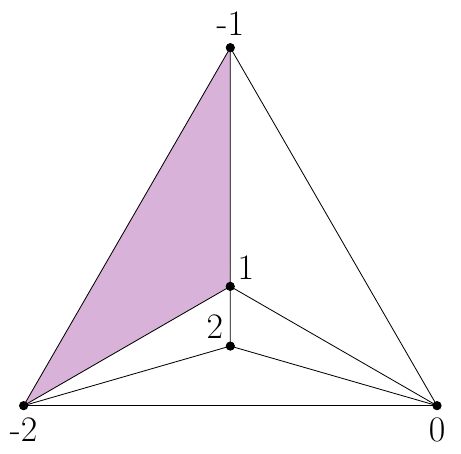}
\ImageNode[left=of E]{F}{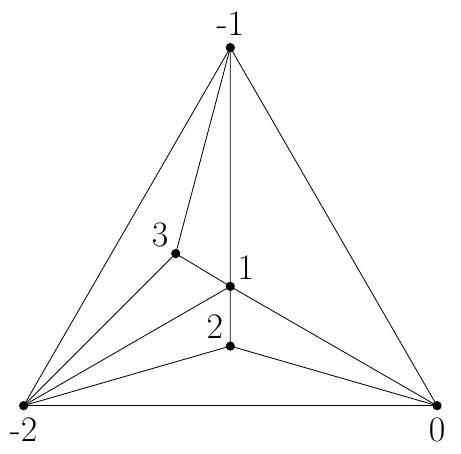}
\ImageNode[left=of F]{G}{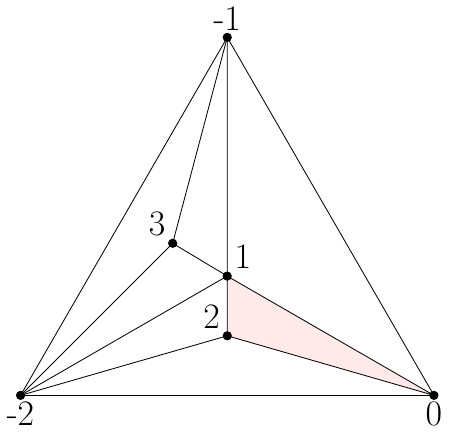}
\ImageNode[left=of G]{H}{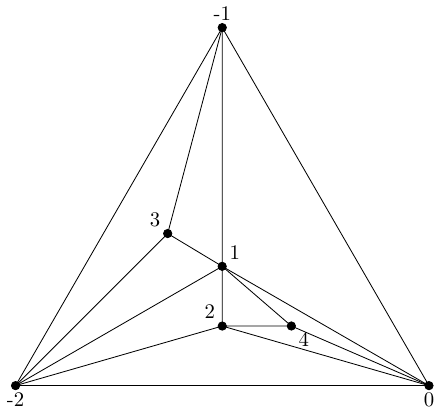}
\draw[-{Latex[length=5mm, width = 5mm]},
    arrowblue,
    line width=5pt
  ]
  (A) -- (B);
  \draw[-{Latex[length=5mm, width = 5mm]},
    arrowblue,
    line width=5pt
  ]
  (B) -- (C);
  \draw[-{Latex[length=5mm, width = 5mm]},
    arrowblue,
    line width=5pt
  ]
  (C) -- (D);
  \draw[-{Latex[length=5mm, width = 5mm]},
    arrowblue,
    line width=5pt
  ]
  (D) -- (E);
  \draw[-{Latex[length=5mm, width = 5mm]},
    arrowblue,
    line width=5pt
  ]
  (E) -- (F);
    \draw[-{Latex[length=5mm, width = 5mm]},
    arrowblue,
    line width=5pt
  ]
  (F) -- (G);
  \draw[-{Latex[length=5mm, width = 5mm]},
    arrowblue,
    line width=5pt
  ]
  (G) -- (H);
 \end{tikzpicture}
\end{center}
\centering
\caption{The evolution of the model in dimension 3. At each step, a $2$-face (triangle) is chosen randomly according to step (i), and subdivided. In Model \textbf{B}, the chosen face is then removed from the complex}
\label{fig:modelbdim3}
\end{figure}

\begin{rmq}
 For general $d$, Model \textbf{A} may be considered as a generalisation of the aforementioned NGF (see \cite{bianconi2016NetworkGW_4})
 with flavour $s = 0$, and bounded energies. We recall that when $s=0$, each face $\sigma$ is selected with probability proportional to $\mathrm e^{-\beta\epsilon_\sigma}$, where $\epsilon_\sigma$ is the (random) energy of face $\sigma$. Model \textbf{B} may be considered as a generalisation of CQNMs with bounded energies. However, note that for brevity, rather than `deactivating' selected faces, we simply remove them from the complex as this does not affect any of the results regarding degree distributions.
\end{rmq}
\begin{rmq}
The methods in this paper also allow us to study the case where the fitnesses associated with a $(d-1)$-face do not depend on the type, but are chosen independently from an underlying distribution. For brevity, we omit formulating explicit results for this model. 
\end{rmq}
\begin{rmq}
The models we introduced can be further generalised. For example, instead of selecting a $(d-1)$-face to subdivide, one may consider a setting where a face of dimension $s$ may be selected and subsequently subdivided, with the addition of an $(s+1)$-dimensional face. 
\end{rmq}
\subsubsection*{Some more notation}

Recall that for all $s\geq 0$, $\mathcal C_s=\{(x_0, \ldots, x_s) \in [0,1]^{s+1}: x_0 \leq \ldots \leq x_s\}$.
For all $x = (x_0,\ldots, x_s) \in \cC_s$ and $i \in \{0, \ldots, s\}$, we set 
$\tilde{x}_i := (x_0, \ldots, x_{i-1}, x_{i+1}, \ldots, x_{s}) \in {\cC_{s-1}}$ and define the empirical measure $\nu_x = \sum_{i=0}^s \delta_{\tilde{x}_i}$ on $\cC_{s-1}$.
For $w \geq 0$ and $y \in \cC_s$, let $y \cup w \in \cC_{s+1}$ denote the vector obtained by adding a coordinate equal to $w$ to the vector $y$ and reordering the coordinates of this $(s+1)$-dimensional vector in non-decreasing order. For $i\in \{0, \ldots, s\}$, we write  $x_{i\lf w} := \tilde x_i \cup w$. 
With this notation, when a face of type $x$ is subdivided by a vertex of weight $w$, we add to the complex $d$ new $(d-1)$-faces of respective types $x_{i \lf w}$ for $i \in \{0, \ldots, d-1\}$. In addition, for a vector $x= (x_0,\ldots,x_{j}, w,x_{j+1} \ldots, x_s) \in \cC_s$, we denote by $x\setminus \{w\}$ the element $(x_0, \ldots,x_{j},x_{j+1}, \ldots, x_s) \in \mathcal{C}_{s-1}$. 
For a vertex $v$ in a $d$-dimensional simplicial complex $\C{}{}{}$, we define the \emph{star} of $v$ in $\C{}{}{}$, which we denote by $\St{v}{\C{}{}{}}$, to 
be the subset of $\C{}{}{d-1}$ consisting of those $(d-1)$-faces which contain $v$. 
Finally, we write $\mathbf{0}$ and $\mathbf{1}$ for the vectors $(0, \ldots, 0)$ and $(1, \ldots, 1)$ respectively, in any dimension. 
\ifthenelse{\boolean{graphs}}{ 
\begin{figure}[H]
\begin{center}
\includegraphics[width=8cm]{Model2}
\end{center}
 \caption{A realisation of $\C{n}{}{}$ in dimension $d=2$, until time $n=3$. The initial complex is a triangle. The bold edges are the faces that are randomly selected at each step.}
\label{fig:model}
\end{figure}
}

\subsection{Main results, Part I: Convergence of the Partition Function} 
We will refer to the following hypotheses throughout the text: 
\begin{itemize}
    \item [\textbf{H1.}] The measure $\mu$ is finitely supported, the fitness function $f$ is positive and $\left|\C{n}{}{d-1}\right| \to \infty$ as $n \to \infty$ (where we recall that $\mathcal K^{\sss (d-1)}_n$ is the set of all $(d-1)$-faces in the random simplicial complex $\mathcal K_n$ at time $n$). 
    \item [\textbf{H2.}] The process $(\K_n)_{n \geq 0}$ evolves according to Model $\textbf{A}$ and $\mu(\{1\})=0$. 
    Moreover, the fitness function $f$ is continuous, monotonically increasing in each argument, positive and such that, for a random variable $W$ with distribution $\mu$, 
    \begin{equation}\label{eq:extra_cond}
    \mathbb E[f(\bs 1_{0\leftarrow W})]< (1+\nicefrac1d)\mathbb E[f(\bs 0_{0\leftarrow W})].
    \end{equation}
\end{itemize}

\begin{rmq}\label{rk:H2}
We do not believe that Assumption {\bf H2}, and in particular Equation~\eqref{eq:extra_cond} which ensures that the function $f$ is not ``too steep'' on its domain of definition, is necessary for our 
results to hold true. 
Our main result on the asymptotic degree distribution holds under Assumptions (a-d) of Remark~\ref{rk:(a-d)} below.
We use Assumption {\bf H2} to show that Assumptions (c-d ) hold: this done in Propositions~\ref{prop:MC} and~\ref{prop:partition}. Their proofs, in the case of $\mu$ having infinite support, rely on recent results of~\cite{maivil} on the convergence of infinitely-many colour P\'olya urns; more precisely, Assumption {\bf H2} ensures that the assumptions of~\cite[Theorem~1]{maivil} hold.

The case when $\mu$ has continuous support is expected to be more difficult to treat; as  
illustrated, for example, in~\cite{Borgs2007} where the Bianconi and Barab\'asi preferential attachment tree with fitness is studied in both the finite support and continuous support case. Borgs et al.~\cite{Borgs2007} treat the continuous support case by coupling it with a finitely-many colour P\'olya urn, but this method does not seem to work in {\color{red}this} case because of  the added complexity introduced by the dependencies in the model (in particular because several vertices belong to one face).
\end{rmq}

Note that $\left| \C{n}{}{d-1} \right| \to \infty$ as long as $d > 1$ in Model \textbf{B}, and for all $d \geq 1$ in Model \textbf{A}. 

\begin{prop} \label{prop:MC}
Assume {\rm\textbf{H1}} or {\rm\textbf{H2}}, and let $Y_n, n \geq 1$ be the $\cCd$-valued random variable that equals the type of the face chosen to be subdivided in the $n$-th step.  
Then, $Y_n$ converges to a $\cCd$-valued random variable~$Y_\infty$ in distribution when $n$ tends to infinity.
\end{prop}

Given any sub-complex $\tilde{\C{}{}{}} \subseteq \C{n}{}{}$ define \begin{equation}\label{eq:def_F}
    F(\tilde{\C{}{}{}}) : = \sum_{\sigma \in \tilde{\K}^{(d-1)} } f(\sigma).
\end{equation}
and note that $F(\C{n}{}{}) = Z_n$ (the partition function defined in \eqref{eq:partition-def}). 

\begin{prop} \label{prop:partition}
Assume {\rm\textbf{H1}} or {\rm\textbf{H2}}. Then, there exists $\lambda > 0$ such that, almost surely,
\[\frac{Z_n}{n} = \frac{F(\C{n}{}{})}{n} \longrightarrow \lambda, \quad \text{as $n \to \infty$}.\]
\end{prop}

\begin{rmq}
The distribution of the limiting random variable $Y_\infty$ and the value of $\lambda$ do not depend on the choice of the initial complex $\C{0}{}{}$. 
\end{rmq}

\begin{rmq}
  Because under {\rm\bf H1} or {\rm\bf H2}, the function $f$ is bounded, we have trivial deterministic bounds on $Z_n = F(\C{n}{}{})$, and therefore on $\lambda$: If we let \begin{equation}\label{eq:def_fmin_fmax}
  f_{\min} = \min \{f(x) : x \in \cCd\}\quad\text{ and }\quad 
  f_{\max} = \max \{f(x) : x \in \cCd\}
  \end{equation}be the minimum and the maximum respectively of the fitness function on its domain of definition, then $\lambda \in [df_{\min},df_{\max}]$ in Model \textbf{A}, whereas $\lambda \in [(d-1)f_{\min},(d-1)f_{\max}]$ in Model \textbf{B}.
\end{rmq}

\begin{rmq}
The monotonicity requirement and \eqref{eq:extra_cond}  in {\rm\textbf{H2}} may be used to cover a particular case of the NGF in \cite{bianconi2016NetworkGW_4} (namely the case with `flavour' $s=0$, in which each face $\sigma$ is selected with probability proportional to $\mathrm e^{-\beta \epsilon_{\sigma}}$, where $\epsilon_\sigma$ is the energy of face $\sigma$, and the selected faces remain in the complex) by setting the weights $w_i = (1-\epsilon_i)$ where $\epsilon_i$ are the energies assigned to the vertices. We therefore assume that the distribution of $\epsilon_i$ does not have an atom at $0$, the energies are bounded, and \eqref{eq:extra_cond} is satisfied, that is, the ``inverse temperature'' $\beta$ satisfies $\beta < \frac{1}{d-1} \log \left(1 + \frac{1}{d} \right)$.  
\end{rmq}

Both propositions are corollaries of a more general almost sure limit theorem for the empirical measure $\Pi_n, n \geq 0$ proved in Section \ref{sec:emp}. While this result (and therefore the two propositions) follows from standard P\'olya urn theory
under \textbf{H1}, for \textbf{H2} we need to make use of general results for measure-valued P\'olya urn processes  recently established in~\cite{maivil} to cover the general case. See, in particular, Section \ref{sec:emp} in this work. 

 \subsection{The companion star process}\label{sub:star_proc_def}
 We state our other main results in terms of a companion 
 process $(S^*_{n})_{n \geq 0}$. Informally, this process approximates the evolution of the star of a fixed vertex $i$ in $(\C{n}{}{})_{n \geq 0}$, assuming that $i$ is sufficiently large
 (namely large enough for the distribution of $Y_i$, the type of the face selected by node $i$ when it enters the network, to be close enough to the distribution of $Y_\infty$ - see Proposition~\ref{prop:MC}).
 Let $\pi_{\infty}$ denote the distribution of the random variable $Y_{\infty}$ from Proposition \ref{prop:MC}. Sample a face type from a measure $\pi_{\infty}$, and form a $(d-1)$-simplex (on vertex set $\{1-d, \ldots, 0\}$) with weights corresponding to this type. Subdivide this face (using the mechanisms of Model \textbf{A} or \textbf{B}) by a new vertex labelled $r$ with weight $W$ sampled from $\mu$, and form the simplicial complex $S^*_0$ consisting of the $(d-1)$-faces containing $r$. 
 We call $r$ the \emph{centre} of  $S^*_0$.
 Then, recursively: 
\begin{itemize}
  \item[(i)] Select a face $\sigma$ from $(S^{*}_{n})^{(d-1)}$ with probability proportional to its fitness, and subdivide it by a new vertex $n+1$ obeying the subdivision rules of Model \textbf{A} or Model \textbf{B} respectively. 
  \item[(ii)] Form the simplicial complex $S^*_{n+1}$ consisting only of the $(d-1)$-faces containing $r$ (essentially this means removing all the $(d-1)$-faces formed during the subdivision step not containing~$r$).
  \end{itemize}
  
\begin{figure}
\captionsetup{width=.8\linewidth}
\begin{tikzpicture}[scale=0.9]
\ImageNode[label={0:}]{A}{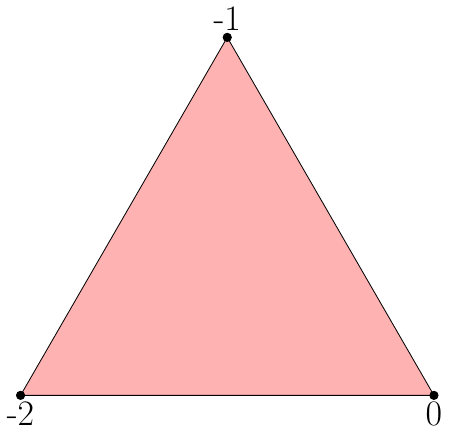}
\ImageNode[label={180:},right=of A]{B}{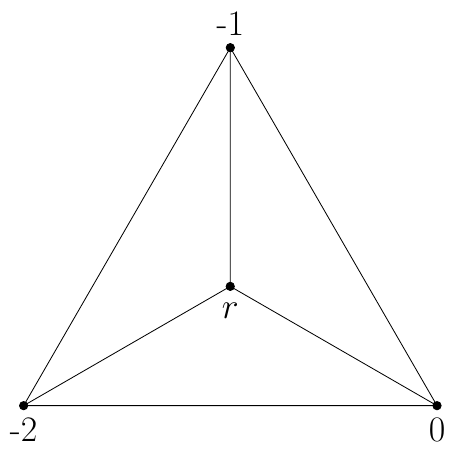}
\ImageNode[label={180:},right=of B]{C}{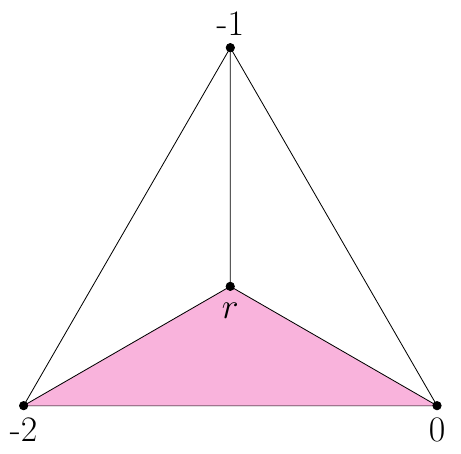}
\ImageNode[right=of C]{D}{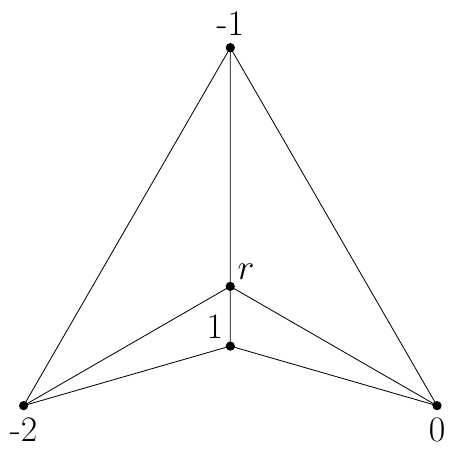}
\ImageNode[below=of D]{E}{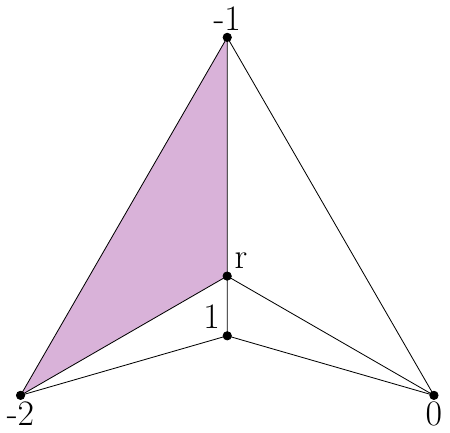}
\ImageNode[left=of E]{F}{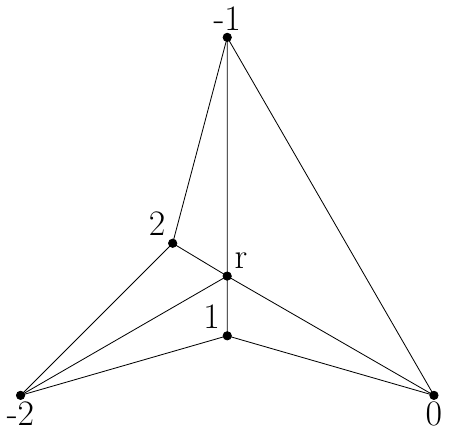}
\ImageNode[left=of F]{G}{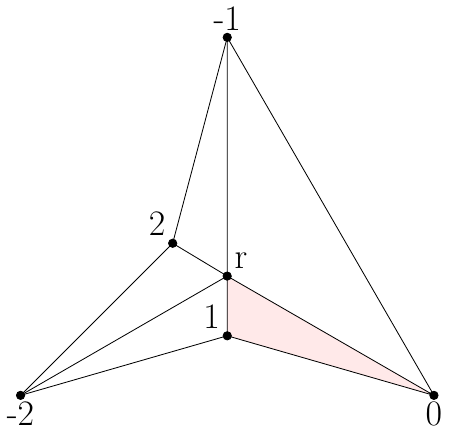}
\ImageNode[left=of G]{H}{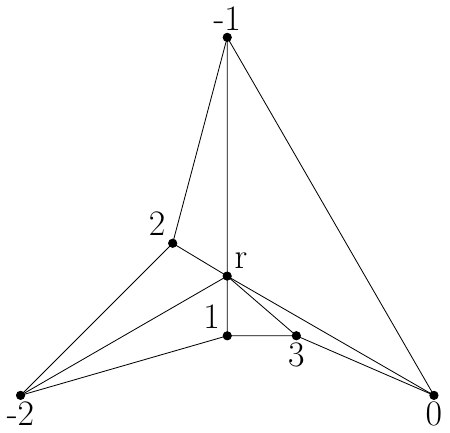}
\draw[-{Latex[length=5mm, width = 5mm]},
    arrowblue,
    line width=5pt
  ]
  (A) -- (B);
  \draw[-{Latex[length=5mm, width = 5mm]},
    arrowblue,
    line width=5pt
  ]
  (B) -- (C);
    \draw[-{Latex[length=5mm, width = 5mm]},
    arrowblue,
    line width=5pt
  ]
  (C) -- (D);
  \draw[-{Latex[length=5mm, width = 5mm]},
    arrowblue,
    line width=5pt
  ]
  (D) -- (E);
    \draw[-{Latex[length=5mm, width = 5mm]},
    arrowblue,
    line width=5pt
  ]
  (E) -- (F);
    \draw[-{Latex[length=5mm, width = 5mm]},
    arrowblue,
    line width=5pt
  ]
  (F) -- (G);
    \draw[-{Latex[length=5mm, width = 5mm]},
    arrowblue,
    line width=5pt
  ]
  (G) -- (H);
 \end{tikzpicture}
\centering
\caption{The evolution of the companion process in Model \textbf{B} and dimension 3. A face with type selected from $\pi_{\infty}$ is formed on vertices $\{-2,-1,0\}$, and subdivided with a vertex labelled $r$. Subsequently, a face is chosen randomly and subdivided according to step (i), and then faces not containing $r$ are deleted.  Since this is Model \textbf{B}, the chosen face is also removed from the complex.}
\label{fig:modelbcompdim3}
\end{figure}
  
A more formal construction of this process is provided in Subsection \ref{subsec:starprocess}. 
We set
\begin{equation}\label{eq:defF*}
    F(S^*_n) := \sum_{\sigma \in (S^{*}_{n})^{(d-1)}} f(\sigma).
\end{equation}
 \subsection{Main results, Part II: Convergence of the Degree Distribution}
\begin{thm} \label{thm:small_degrees_vague}
Assume {\rm\textbf{H1}} or {\rm\textbf{H2}}
and for all $n\geq 1$, $k\geq 0$, let $N_k(n)$ denote the number of nodes of degree~$k+d$ in the random simplicial complex $\mathcal K_n$ at time~$n$. Then, for all $k\geq 0$, we have, with convergence in probability,
\[\lim_{n \to \infty} \frac{1}{n} \D{k}{n} = \E{\frac{\lambda}{F(S^*_k)+\lambda}\prod_{j=0}^{k-1} \frac{F(S^*_j)}{F(S^*_j)+\lambda}} =: \p_k,\]
where the star process $S^*$ and its fitness function $F$ are defined respectively in Subsection~\ref{sub:star_proc_def} and Equation~\eqref{eq:defF*}.
\end{thm}

In fact, we have the more general result: suppose that  $N^{(s)}_{k}(n)$ denotes the number of vertices of $s$-degree ${d \choose s} + {d-1 \choose s-1}k$, for $1 \leq s < d$ (the $s$-degree of a face is the number of distinct $s$-faces that contain it).

\begin{cor} \label{thm:higher_dimensional_degrees}
Assume {\rm\textbf{H1}} or {\rm\textbf{H2}}. For all $k\geq 0$, we have, independent of the initial complex $\C{0}{}{}$, with convergence in probability, \[\lim_{n\to \infty} \frac{1}{n} N^{(s)}_{k}(n) = \p_k.\]
\end{cor}

\begin{rmq}\label{rk:(a-d)}
In fact, in {\color{red}the} proof of Theorem~\ref{thm:small_degrees_vague}, we show that the conclusion of the theorem holds if one assumes the following instead of {\bf H1} or {\bf H2}:
\begin{enumerate}[{\rm (a)}]
\item The measure $\mu$ is an arbitrary probability measure on $[0, \infty)$.
\item The fitness function $f$ is non-negative, symmetric, bounded and continuous.
\item If for all $n\geq 1$, $Y_n$ is the type of face that is subdivided at time $n$, then $(Y_n)_{n\geq 1}$ converges in distribution when $n\to+\infty$.
\item There exists $\lambda>0$ such that, almost surely when $n\to+\infty$, $F(\mathcal K_n)/n \to \lambda$.
\end{enumerate}
Note that {\rm (a-b)} above is much weaker than assuming {\bf H1} or {\bf H2} as we do in Theorem~\ref{thm:small_degrees_vague}, but {\bf H1} or {\bf H2} gives a sufficient condition for {\rm (c-d)} above to hold (as seen in Propositions~\ref{prop:MC} and~\ref{prop:partition}).
\end{rmq}

\begin{rmq}
Note that the boundedness of $f$ implies that
\begin{equation} \label{eq:trivboundsS}
\begin{cases}
(d + (d-1)n)f_{\min}  \leq F(S^*_n) \leq (d + (d-1)n)f_{\max}, & \text{in Model \textbf{A}}; \\
(d + (d-2)n)f_{\min}  \leq F(S^*_n) \leq (d + (d-2)n)f_{\max}, & \text{in Model \textbf{B}},
\end{cases}
\end{equation}
where we recall that $f_{\min}$ and $f_{\max}$ are the minimum and the maximum of the fitness function $f$ (see Equation~\eqref{eq:def_fmin_fmax}).
\end{rmq}

\begin{rmq}
Although Theorem~\ref{thm:small_degrees_vague} is about degrees of vertices, 
our approach is not restricted only to the graph that underlies the simplicial complex but it can be used in order to study degrees of higher order faces or degrees defined in terms of lower order faces. In the latter 
direction is Corollary~\ref{thm:higher_dimensional_degrees}.

For an $r$-face $\sigma$ with $r <d-1$, the \emph{degree} of $\sigma$ is the number of $(d-1)$-faces which contain $\sigma$. 
One can derive the analogue of Theorem~\ref{thm:small_degrees_vague}
for the degree distribution of $r$-faces by considering a star companion process for an $r$-face. 
Here, the star of an $r$-face will simply consist of the $(d-1)$-faces that contain it. 
As long as the process is such that a.s. the total
weight of the star tends to infinity, then one could
derive a formula as in Theorem~\ref{thm:small_degrees_vague}.
\end{rmq}

\subsubsection*{Outline of the paper}
In Section \ref{sec:disc} we discuss the connection of our main results to existing models. This will include classifying the values of $d$ that ensure that the degree distributions follows a power law, which are consistent with analysis from \cite{bianconi2015ComplexQN_1} and \cite{bianconi2016NetworkGW_4}.

Section \ref{sec:emp} is dedicated to the study of the empirical measure $\Pi_{n}$, $n \geq 0$, and in particular, to the proofs of Propositions \ref{prop:MC} and \ref{prop:partition}. 
\miscchange{As we remarked earlier (see Remark~\ref{rk:H2}), 
these propositions make use of the recent theory of measure-valued P\'olya processes. 
To our knowledge this is the first application of this theory in the context of evolving networks.}

In Section \ref{sec:deg-prof} we apply the results of Section \ref{sec:emp} to prove Theorem~\ref{thm:small_degrees_vague}. 
\miscchange{The proof relies on the simple idea of keeping track of the degree of a single typical vertex. 
Suppose (informally) that the partition function of the process as well as that of the star companion process were to both evolve deterministically,  
and were equal to $F(n) = \lambda (n+1)$ and $F^*(n) = \lambda^* (n+1)$, respectively. 
Then, the probability that the star of vertex $i$ is subdivided precisely at times $i<i_1 < \cdots < i_k<n$ would be
\begin{eqnarray*}
&&\prod_{j=1}^{i_1-i-1} \left(1- \frac{\lambda^*}{\lambda(i+j)}\right) \frac{\lambda^*}{\lambda i_1} \cdot \prod_{j=1}^{i_2-i_1-1} \left(1- \frac{2 \lambda^*}{\lambda(i_1+j)}\right) \frac{2\lambda^*}{\lambda i_2} \cdots \\
&&\hspace{2cm}\cdots\prod_{j=1}^{i_k-i_{k-1}-1} \left(1- \frac{\lambda^* (k-1)}{\lambda(i_{k-1}+j)}\right) \frac{\lambda^* (k-1)}{\lambda i_k} \cdot
\prod_{j=1}^{n-i_{k}} \left(1- \frac{\lambda^* k}{\lambda(i_{k}+j)}\right).
\end{eqnarray*}
If $i > \eps n$ and the $i_j$s are well-spaced (as most such $k$-tuples are), then the above products 
can be written as ratios of factorials (or Gamma functions) which, in turn, can be approximated with the use of Stirling's formula. Then the argument 
could be completed by computing the sum over the choices of $k$-tuples (by applying Lemma~\ref{lem:prob-sum}).

However, the difficulty is that 
the partition functions are not exactly of this form but only in the limit (by Propositions~\ref{prop:partition} and~\ref{prop:fitness-star-process}). Nevertheless, the almost sure convergence of $Z_n = F(\mathcal{K}_n)$ implies (by Egorov's theorem) that when $n$ is large, for `most' evolution paths of the process, the partition function $F(\K_j)$ is about $\lambda j$ for all $\eps n < j \leq n$. The crux of our analysis is to replace the linear functions in the above expression by the `almost' linear functions which occur on a typical evolution path. This is done in Subsections~\ref{subsec:upper}
and~\ref{subsec:lower}, where upper and lower bounds are obtained. 
We believe that the conceptual simplicity of this approach makes it applicable to other evolving random systems. 
}

We defer the proofs of some technical probabilistic lemmas to the appendix, so as to not interrupt the general flow of the paper. 

\section{Discussion and Examples}
\label{sec:disc}
\subsection{Constant fitness function}
In the case that the fitness functions are constant, so that $f(x) = f_0$, we have deterministic formulas for $F(S^*_n)$ and $\lambda$. These cases correspond to models where the face chosen to be subdivided at time $n+1$ is chosen uniformly at random from the set $\K_n^{(d-1)}$. Here we use the asymptotic approximation of the ratio of two gamma functions: for fixed $a \in \mathbb{R}$ as $t\to \infty$
\begin{equation} \label{eq:stirling_gamma_approx} \frac{\Gamm{t+a}}{{\Gamm{t}}} = (1+ O(1/t)) t^{a}.\end{equation} This is a straightforward result of Stirling's formula and will be used often throughout this paper.

\begin{enumerate}
\item In Model \textbf{A} we have $F(S^{*}_n) = ((d-1)n + d)f_0$, and $\lambda = df_0$. Theorem \ref{thm:small_degrees_vague} implies that \[\p_k = \frac{d}{(d-1)k+2d}\prod_{j=0}^{k-1} \frac{(d-1)j + d}{(d-1)j+2d}.\]
If $d > 1$, using~\eqref{eq:stirling_gamma_approx}
\[\p_k = \left(1 + \frac{1}{d-1}\right)\frac{\Gamm{k + \frac{d}{d-1}}\Gamm{\frac{2d}{d-1}}}{\Gamm{k + 1 + \frac{2d}{d-1}}\Gamm{\frac{d}{d-1}}} \sim k^{- \frac{2d-1}{d-1}}.\] 
This is a new result. For $d=1$ we obtain $\p_{k} = 2^{-k}$, which is an old result of Na and Rapoport for the random recursive tree \cite{na_rapo_70}. 
\item Model \textbf{B} with constant fitness function (with $\C{0}{}{}$ given by a $d$-simplex) is the same as the Random Apollonian Network. In this case, if $d \geq 2$,  $F(S^{*}_n) = ((d-2)n + d)f_{0}$ and $\lambda = (d-1)f_0$. Applying Theorem \ref{thm:small_degrees_vague} we get, \[\p_k =  \frac{d-1}{(d-2)k+2d-1}\prod_{j=0}^{k-1} \frac{(d-2)j + d}{(d-2)j+2d-1}.\]
Note that if $d = 1$, $\Pi_{n}(\cCd) = |V_0|$ (where $V_0$ is the set of vertices of the initial complex $\K_0$), so Theorem \ref{thm:small_degrees_vague} does not apply. However, in this case it is easy to see that $\p_1 = 1$. In the case $d = 2$, we have $\p_{k} = \frac{2^{k-1}}{3^{k}}$.  For $d \geq 3$, using~\eqref{eq:stirling_gamma_approx}, we get
\[\p_k = \left(1 + \frac{1}{d-2}\right)\frac{\Gamm{k + \frac{d}{d-2}}\Gamm{\frac{2d-1}{d-2}}}{\Gamm{k + 1 + \frac{2d-1}{d-2}}\Gamm{\frac{d}{d-2}}} \sim k^{- \frac{2d-3}{d-2}}.\] 
 This is the same exponent proved in \cite{apollonian_hungarians_1} and \cite{apollonian_frieze_2}. 
\end{enumerate}

\subsection{Weighted Recursive Trees}
The one-dimensional case in Model \textbf{A} and initial simplicial complex given by a node, is a type of the 
\emph{weighted recursive tree}, introduced in \cite{wrt1} (see also \cite{wrt2delphin} for some more general results).\footnote{Note that Model \textbf{B} is trivial for $d=1$ as the tree is a single path.} In this case, the fitness of the new vertex arriving at each time is independent of the rest of the complex, so the strong law of large numbers implies that $\lambda$ in Proposition \ref{prop:partition} is given by $\E{f(W)}$. Moreover, the simplicial complex $(S^{*}_{j})_{j \geq 0}$ is a fixed vertex, so that
$F(S^*_j) = f(W)$ for all $j \geq 0$, where $W$ is the weight of the vertex. Thus, Theorem \ref{thm:small_degrees_vague} implies that 
\begin{prop} \label{prop:WRRTweak}
As $n\to+\infty$, 
 we have
\[\frac{N_k(n)}{n} \to \E{\frac{ \lambda f(W)^k}{(f(W) + \lambda)^{k+1}}}, \quad \text{in probability}.\]
\end{prop}
This result can be improved significantly: the convergence holds in an almost sure sense under the much weaker assumptions that $\mu$ is a probability measure on $[0, \infty)$ and $f: \R \to \R$ is measurable such that $0 < \E{f(W)} < \infty$.
This strengthening uses the theory of \emph{Crump-Mode-Jagers} (C-M-J) processes introduced by Crump and Mode \cite{crump_mode_68} and studied by, among others, Jagers \cite{jagers_74}, Nerman \cite{nerman_81} and Jagers and Nerman \cite{jag_ner_84}. Here, $\lambda$ plays the role of the so-called \emph{Malthusian parameter} crucial to the study of C-M-J processes. We omit the details of this proof, as they detract from the main ideas in this paper. 

\subsection{Tails of the Distribution}
\label{sec:tails}
In this subsection, we will require the additional assumption that
\begin{equation}
\label{eq:tailsequation}
\left|\K_{n}^{(d-2)} \right| \stackrel{n \to \infty}{\longrightarrow} \infty.
\end{equation}  
Note that this assumption is satisfied as long as $d > 1$ in Model \textbf{A} and $d > 2$ in Model \textbf{B}. It is this assumption that leads to the emergence of scale-free behaviour for $d > 2$ in CQNMs observed by Bianconi and Rahmede in \cite{bianconi2015ComplexQN_1}, and the scale-free behaviour for all $d > 1$ in NGFs in \cite{bianconi2016NetworkGW_4}. In the case $\mu$ is not finitely supported, we will require an analogue of \eqref{eq:extra_cond}. For brevity, we define the following additional hypotheses:
\begin{itemize}
    \item [\textbf{H1*.}] Assume \textbf{H1} and \eqref{eq:tailsequation} holds.
    \item [\textbf{H2*.}] Assume \textbf{H2} and \eqref{eq:tailsequation} holds. 
    Moreover, for all $w \in \Supp{(\mu)}$, the function $\tilde{f}_x: \cC_{d-2} \to \mathbb R, \tilde{f}_x(v) = f(v\cup x)$ satisfies 
    \[
        \mathbb E[\tilde{f}_x(\bs 1_{0\leftarrow W})]< (1+\nicefrac1{(d-1)})\mathbb E[\tilde{f}_x(\bs 0_{0\leftarrow W})].
    \]
  (We recall that $\bs 1$ is the vector of $\mathcal C_{d-2}$ whose coordinates are all equal to~1. Therefore, $\bs 1_{0\leftarrow W} = (W, 1, \ldots, 1)$ and $\tilde{f}_x(\bs 1_{0\leftarrow W})$ equals $f((x,W,1, \ldots, 1))$ if $x<W$ and $f((W,x,1, \ldots, 1))$ otherwise.)
\end{itemize}

\begin{rmq}
Similarly to {\bf H2}, we do not believe that Assumption {\bf H2*} is necessary for our results to hold. We use it to apply~\cite[Theorem~1]{maivil} in the proof of Proposition~\ref{prop:fitness-star-process}.
\end{rmq}

In order to analyse the tails of the distribution from Theorem \ref{thm:small_degrees_vague}, we require the following proposition, similar to Proposition \ref{prop:partition}. In the statement of the following proposition, we allow $S^*_0$ to have a centre with a fixed weight $w$ instead of a random weight $W$ with distribution $\mu$. In the construction of $S^*_0$, however, we still choose the face according to  $\pi_{\infty}$. We use $\mathbb P_w$ and $\mathbb E_w$ for probabilities and expectations, respectively with regards to this initial state. 

\begin{prop} \label{prop:fitness-star-process}
Assume {\rm\textbf{H1*}} or {\rm\textbf{H2*}}. Then, if
the centre of $S^{*}_0$ has weight $w \in \Supp(\mu)$,  there exists $\lambda_{w}^{*}$ such that, $\mathbb P_w$-almost surely
\[\frac{F(S^{*}_n)}{n} \rightarrow \lambda_{w}^{*}.\]
\end{prop}
We postpone the proof of  Proposition~\ref{prop:fitness-star-process} to Subsection~\ref{subsec:starprocess}. The following proposition holds under {\bf H1*}: Under Assumption {\bf H1*}, $\mu$ has finite support and thus $\max\{\lambda^*_w : w \in \Supp{(\mu)}\}$ exists and is attained at some value $w^*\in\Supp(\mu)$; we set $\lambda_{\wmax}^* = \max \{\lambda^*_w : w \in \Supp{(\mu)}\}$.
\begin{prop}\label{Power_law}
Assume {\rm\textbf{H1*}}. 
With $\p_k$ as defined in Theorem \ref{thm:small_degrees_vague}, we have
\begin{equation} \label{eq:PL_liminf}
\liminf_{k\to \infty} \log_{k}{\p_k} \geq -\left(1 + \frac{\lambda}{\lambda^*_{\wmax}}\right). 
\end{equation}
\end{prop}

\begin{proof}
Suppose $\Prob{W = \wmax} = \kappa$ (recall that under \textbf{H1*} $\mu$ is finitely supported). Then, by the definition of $\p_k$, we have
\begin{linenomath}
\begin{align*} 
\p_{k} &=  
  \E{\frac{\lambda}{F(S^{*}_k)+\lambda} 
\prod_{j=0}^{k-1} \frac{F(S^{*}_j)}{F(S^{*}_j) + \lambda}}  \geq \E[\wmax]{\frac{\lambda}{F(S^{*}_k)+\lambda} 
\prod_{j=0}^{k-1} \frac{F(S^{*}_j)}{F(S^{*}_j) + \lambda}} \kappa.
\end{align*}
\end{linenomath}
Fix $\delta, \varepsilon' >0$.
By Proposition~\ref{prop:fitness-star-process} (and Egorov's theorem), 
there exists $k_0 = k_0(\eps, \delta)$ such that for all $k \geq k_0$
\[\mathbb P_{w^*}\left(\left|\frac{F(S^{*}_{k})}{k} - \lambda_{\wmax}^{*} \right| < \eps\right) > 1- \delta.\]
Let $\mathcal{G}^{*}_{\eps,\delta}$ be the associated event in the previous display. We may bound the product $\prod_{j=0}^{k_0 - 1} \frac{F(S_j^*)}{F(S_j^*) + \lambda}$ below by a constant by applying \eqref{eq:trivboundsS}. Moreover, for all $k > k_0$, on $\mathcal{G}^{*}_{\eps,\delta}$, we have 

\begin{linenomath}
\begin{align*}
\frac{\lambda}{F(S^{*}_k) + \lambda}
\prod_{\ell=k_0}^{k-1} \frac{F(S^{*}_{\ell})}{F(S^{*}_{\ell}) + \lambda} 
&> \frac{\lambda\left(k(\lambda^{*}_{w^{*}}- \eps) + \lambda\right)}{k(\lambda^{*}_{w^{*}} + \eps) + \lambda} \cdot \frac{1}{k(\lambda^{*}_{w^{*}} - \eps) + \lambda} \prod_{\ell=k_0}^{k-1} \frac{\ell(\lambda_{\wmax}^* - \eps)}{\ell (\lambda_{\wmax}^* - \eps) + \lambda}\\ 
&=  \frac{k(\lambda^{*}_{w^{*}} - \eps) + \lambda}{k(\lambda^{*}_{w^{*}} + \eps) + \lambda}\cdot \frac{\lambda}{\lambda^{*}_{w^{*}} -\eps} \cdot \frac{\Gamma (k_0 + \frac{\lambda}{\lambda_{\wmax}^*-\eps})}{\Gamma (k_0-1)} 
\frac{\Gamma (k)}{\Gamma (k+1 + \frac{\lambda}{\lambda_{\wmax}^* - \eps})}.
\end{align*}
\end{linenomath}
 Therefore, by applying \eqref{eq:stirling_gamma_approx}, we find that there exists a constant $c = c(k_0,\delta,\eps, \kappa)$ such that 
\[ \log_{k}\p_k \geq \log_{k} c - \left(1 + \frac{\lambda}{\lambda^*_{\wmax} - \eps}\right).\]
Equation~\eqref{eq:PL_liminf} follows from taking limits as $k \to \infty$, and sending $\eps$ to $0$. 
\end{proof}

\subsubsection*{Further Discussion}

Applying \eqref{eq:trivboundsS}, it is easy to show that, whenever \eqref{eq:tailsequation} holds,  
\[
\liminf_{k\to \infty} \log_{k}\p_k \geq 
\begin{cases}
- \left(1 + \frac{\lambda}{(d-1)f_{\min}}\right), &  \text{in Model \textbf{A};} \\
- \left(1 + \frac{\lambda}{(d-2)f_{\min}}\right), &  \text{in Model \textbf{B},}
\end{cases} 
\]
and likewise, 
\[ 
\limsup_{k\to \infty} \log_{k}\p_k \leq 
\begin{cases}
- \left(1 + \frac{\lambda}{(d-1)f_{\max}}\right), &  \text{in Model \textbf{A};} \\
- \left(1 + \frac{\lambda}{(d-2)f_{\max}}\right), &  \text{in Model \textbf{B}.}
\end{cases} 
\]
Thus, when $d > 1$ in Model \textbf{A} and $d > 2$ in Model \textbf{B}, the degree distribution is bounded above and below by a power law. This leads to the scale-free behaviour observed in \cite{bianconi2015ComplexQN_1} and \cite{bianconi2016NetworkGW_4}. 

In general, by counting the edges in the complex in two different ways, we find that 
$\sum_{k=0}^{\infty} k\p_{k} \leq d,$
so that $\p_k$ cannot obey a power law with a fixed exponent less than $2$ (otherwise the sum would diverge). However, we cannot deduce from these methods that the degree distribution in each case follows a power law with a fixed exponent. 

\ifthenelse{\boolean{wrrt}}{ 
\section{The weighted recursive tree} \label{sec:wrrt}

The one dimensional case in Model \textbf{A} and initial simplicial complex given by a node, is a type of \emph{weighted recursive tree}, introduced in \cite{wrt1} (see also \cite{wrt2delphin}) for some more general results).\footnote{Note that Model \textbf{B} is trivial for $d=1$ as the tree is a single path.}
At time $n=1$, we start with a single node labelled $1$ to which we associate a 
random variable $W_1$ with distribution $\mu$. Then, given the tree with $n$ vertices carrying labels from $1$ to $n$ and weights $W_1, \ldots, W_n$, we add a new node labeled $n+1$ and connect it to a randomly selected node $1 \leq I \leq n$ where $\Prob{I=k}$ is proportional to $W_k$. Node $n+1$ is then assigned  weight $W_{n+1}$ distributed according to $\mu$ independently of all previously defined quantities. We write $T_n$ for the tree of size $n$.

For $k \geq 0$, recall that $N_{k}(n)$ counts the number of vertices in $T_n$ with $k$ children. In the standard random recursive tree (that is, $\mu = \delta_1$), it is well-known that  $N_k(n)/n \to 2^{-k-1}$ as $n \to \infty$ \cite{na_rapo_70}. The following proposition generalizes this classical result.

\begin{prop} \label{prop:WRRT}
Assume that $m = \E{W_1} < \infty$. We have
\[\frac{N_k(n)}{n} \to p_k := \E{\frac{ m W_1^k}{(W_1+m)^{k+1}}}\]
almost surely, as $n\to\infty$.
\end{prop}

The proof of the proposition relies on an embedding of $T_n$ into a continuous-time particle process. To this end, we make use of the so-called~\emph{Crump-Mode-Jagers} (C-M-J) processes introduced by Crump and Mode \cite{crump_mode_68} and studied by, among others, Jagers \cite{jagers_74}, Nerman \cite{nerman_81} and Jagers and Nerman \cite{jag_ner_84}. See also Holmgren and Janson \cite{sur_crump} recent elaborate survey for deep connections between random trees and C-M-J processes and further references.

Let $\xi$ be a point process on $[0, \infty)$. In the C-M-J process, we associate particles with a subset of the infinite \emph{Ulam-Harris tree} $\mathbb U := \bigcup_{n \geq 0} \mathbb N^n$, where $\mathbb{N}^0=\varnothing$. 
Let $\{\xi_x : x \in \mathbb U\}$ be a family of independent random variables, where each $\xi_x$ is a copy of $\xi$. Write $0 < \sigma_1^{(x)} < \sigma_2^{(x)} < \ldots$ for the points in the process $\xi_x$. For simplicity of presentation, we assume that both $\xi(\{0\}) = 0$ and $\xi([0, \infty)) = \infty$ almost surely. Now, set $\sigma_\emptyset := 0$ and, 
for a node $x = x_1\dots x_k$ with $x_1, \ldots, x_k \in \mathbb N$, recursively, $\sigma_{xi} = \sigma_x + \sigma^{(x)}_i, i \geq 1$. Finally, for $t \geq 0$, let $\mathbb T_t = \{ x \in \mathbb U : \sigma_x \leq t\}$. We write $N(t)$ for the cardinality of $\mathbb T_t$. Let $\nu$ be the intensity measure of $\xi$, that is,  $\nu(A) := \E{\xi(A)}$ for Borel sets $A \subseteq [0, \infty)$. A standard condition under which $\E{N(t)} < \infty$ and therefore $N(t) < \infty$ almost surely for all $t \geq 0$ is that $\nu([0, t]) < \infty$ for all $t \geq 0$. We assume this from now on.

Informally, the dynamics of the  tree-valued c\`adl\`ag process $\mathbb T_t$  can be described as follows: at time zero, only one particle denoted by $\emptyset$ is alive. Then, for $i \geq 1$, at time $\sigma_i^{(\emptyset)}$, particle $i$ is born and assigned point process $\xi_i$. New particles behave independently from each other and reproduce according to the same law as the initial particle.
Let $x \preceq y$ mean that $x$ is a prefix of $y$.
 By construction, setting $\mathbb T^{(x)}_t = \{ y \in \mathbb T_{t+ \sigma_x}: x \preceq y\}$, the random processes $(\mathbb T^{(x)}_t)_{t \geq 0}$ are identically distributed. These processes are independent for particles $x_1, \ldots, x_k \in \mathbb U$, if there do not exist $1 \leq i \neq j \leq k$ satisfying $x_i \preceq x_j$. (In other words, the $k$ subtrees of $\mathbb U$ rooted at $x_1, \ldots, x_k$ are pairwise node-disjoint.)

We can enrich the model by equipping each particle with a score function $\phi$: formally, we let
$\phi_x$ be a deterministic real-valued function of $\xi_x$ (the same for all $x$) such that $\phi_x(t), t \geq 0$ is c\`adl\`ag.
We then define the score function of the process $\mathbb T_t$ as
\[Z^{\phi}(t) := \sum_{x \in \mathbb U: \sigma_x \leq t} \phi_x(t-\sigma_x).\]
A crucial parameter in the study of C-M-J processes is the Malthusian parameter $\lambda^\star$ defined as the solution (if it exists) of
\[1=\int_0^{\infty} \mathrm e^{-\lambda^\star s} \nu(ds).\]
Nerman \cite{nerman_81} proved the following remarkable result: for two score functions $\phi^{(1)}, \phi^{(2)}$ for which there exists $\beta < \lambda^*$ satisfying
$$\int e^{-\beta s} \nu(ds) < \infty, \quad \E{\sup_{t \geq 0} e^{-\beta t} \phi^{(i)}(t)} < \infty, \quad i=1, 2,$$
one has, almost surely,

\begin{align} \label{eq:Nerman} \lim_{t \to \infty} \frac{Z^{\phi^{(1)}}(t)}{Z^{\phi^{(2)}}(t)} = \frac{\int_0^\infty \mathrm e^{-\lambda^\star s}\E{\phi	^{(1)}(s)} ds}{\int_0^\infty \mathrm e^{-\lambda^\star s}\E{\phi^{(2)}(s)} ds}. \end{align}
\begin{proof}[Proof of Proposition \ref{prop:WRRT}]
A moment of thought shows that, with a homogeneous Poisson process $\xi$ on $[0, \infty]$ with (random) intensity $W$ distributed according to $\mu$, the
particles in the C-M-J process at time $\tau_n := \inf \{ t \geq 0: N(t) = n\}$ can be identified with the vertices of $T_n$. More precisely, the two sequences
$(T_n)_{n \geq 1}$ and $(\mathbb T_{\tau_n})_{n \geq 1}$ are identically distributed.  Note that the intensity measure $\nu$ of the point process $\xi$ is given by 
$\nu(ds) = m ds$. Hence, the Malthusian parameter $\lambda^*$ of the corresponding C-M-J process is equal to $m = \E{W_1}$. 

We choose the two following score functions:
\[\phi^{\sss (1)}(t)=\begin{cases}
1 & \text{ if }t\geq 0\\
0 & \text{ otherwise,}
\end{cases} \quad 
\phi^{\sss (2)}(t) =\begin{cases}
1 & \text{ if }t\geq 0 \text{ and }\xi(0,t)=k\\
0 & \text{ otherwise,}
\end{cases}\]
where $k$ is a fixed integer.
The two score functions associated with $\phi^{\sss (1)}$ and $\phi^{\sss (2)}$ are
\[Z^{\sss (1)}(t) = X(t) \quad \text{ and }\quad
Z^{\sss (2)}(t) = X_k(t),\]
where $X(t)$ is the total number of particles in $\mathbb T_t$, and 
$X_k(t)$ is the number of particles in $\mathbb T_t$ which have produced offspring exactly $k$ times by time $t$. These particles correspond to vertices having out-degree $k$ in the tree.
As $\tau_n \to \infty$, using \eqref{eq:Nerman}, we have, almost surely,
\[\lim_{n \to \infty} \frac{N_k(n)}{n} = \lim_{t\to\infty} \frac{X_k(t)}{X(t)}
= \frac{\int_0^{\infty} \mathrm e^{-\lambda^\star s} \mathbb E[\phi^{\sss(2)}(s)]\mathrm ds}
{\int_0^{\infty} \mathrm e^{-\lambda^\star s} \mathbb E[\phi^{\sss(1)}(s)]\mathrm ds}.\]
Note that
\[\int_0^{\infty} \mathrm e^{-\lambda^\star s} \mathbb E[\phi^{\sss(1)}(s)]\mathrm ds
=\int_0^{\infty} \mathrm e^{-\lambda^\star s} \mathrm ds = 1/\lambda^\star= 1/m,\]
and
\begin{linenomath}
\begin{align*}
\int_0^{\infty} \mathrm e^{-\lambda^\star s} \mathbb E[\phi^{\sss(2)}(s)]\mathrm ds
&= \int_0^{\infty} \mathrm e^{-\lambda^\star s} \mathbb P(\xi(0,s)=k) \mathrm ds \\
&= \int_0^{\infty} \E{\mathrm e^{-(\lambda^\star+ W_1)s} \frac{(W_1s)^k}{k!}} \mathrm ds\\
& =  \E{\frac{ W_1^k}{(W_1+m)^{k+1}}}.
\end{align*}
\end{linenomath}
Combining the last three displays yields the result.
\end{proof}

\begin{rmq}
For the random recursive tree, (Janson \cite{sva_rec}) showed multivariate central limit theorems for $(N_{k_1}(n), \ldots, N_{k_m}(n))$ under an appropriate re-scaling based on his general limit theorems for asymptotic proportions of balls of different colors in {P}\'olya urn models \cite{janson_urns}. If $W$ has finite support, (hypothesis \textbf{H1}) his results extend to this model, i.e.\ there exist positive constants $\gamma_k, k \geq 0$ such that
 \[\frac{N_k(n) - n p_k}{\sqrt{n}} \stackrel{d}{\longrightarrow} \mathcal N(0, \gamma_k).\]
It is plausible that such statements hold for a general distribution $\mu$.
 \end{rmq}
 \medskip 
 
\begin{rmq}
Let $A \subseteq [0, \infty)$ be a measurable set and $N_k^A(n)$ be the number of vertices with exactly $k$ children that have weights in the set $A$. Then, the above proof can easily be adapted to show that, almost surely, 
$$\frac{N^A_k(n)}{n} \to \E{\frac{m W_1^k}{(W_1+m)^{k+1}} \mathbf 1_{A}(W_1)}.$$
\end{rmq}
 }
 
 \section{Convergence of the empirical distribution} \label{sec:emp}
The aim of this section is to prove the following almost sure limit theorem for the empirical distribution $\Pi_n$. 

\begin{thm} \label{empir_limit}
Assume {\rm\textbf{H1}} or {\rm\textbf{H2}}. Then, there exists a deterministic, positive, finite measure $\pi$ on $\cCd$, which  does not depend on the choice of $\C{0}{}{}$ such that, almost surely,  
\[\frac{\Pi_n}{n} \to \pi\]
with respect to the weak topology.
\end{thm}
\medskip 

Proposition \ref{prop:partition} follows from the theorem above where $\lambda=\int_{\cCd}f(x)\,\mathrm d\pi(x).$
Likewise, Proposition \ref{prop:MC} follows immediately where $Y_\infty$ has law $\pi_{\infty}$ defined by
\[\pi_{\infty} (A)= \frac{\int_{A}f(x) \dd \pi (x)}{\int_{\cCd}f(x) \dd \pi(x)}, 
\]
for any measurable set $A \subseteq \cCd$.

\subsection{Hypothesis {\rm\textbf{H1}}}
To prove Theorem~\ref{empir_limit} assuming \textbf{H1}, we view the collection of faces as balls in a \emph{generalised P\'olya urn process}. In this set-up, one considers an \textit{urn} consisting of \textit{balls} with a finite number of possible \textit{colours}. A ball of colour $j$ is sampled at random from the urn with probability proportional to its \textit{activity} $a_j$, and replaced with a number of different coloured balls according to a (possibly random) \textit{replacement rule}. In the common set-up, the configuration of the urn after $n$ replacements is represented as a \textit{composition vector} $X_{n}$ with entries labelled by colour, and the activities of colours are encoded in an \textit{activity vector} $\mathbf{a}$. In this vector, the $i$th entry corresponds to the number of balls with a colour $i$. Let $(\xi_{ij})$ be the matrix whose $ij$th component denotes the random number of balls of colour $j$ added, if a ball of colour $i$ is drawn. 
The following is a well known result by Athreya and Karlin, implied by Proposition~2 in \cite{at_embedding} and Theorem~5 of \cite{split_times}. We state a version implied by a result of Janson \cite{janson_urns}. 
\begin{thm}[\cite{janson_urns}] \label{thm:polyurnconvergence}
Assume $\xi_{ii} \geq -1$, $\xi_{ij} \geq 0$ for $i \neq j$, and the matrix $A_{ij} : = a_j \E{\xi_{ji}}$ is \textit{irreducible}. Moreover, denote by $\lambda_1$ the principal eigenvalue of $A$, and $v_1$ the corresponding right-eigenvector normalised so that $\mathbf{a}^{T} v_1 = 1$. For any non-empty initial configuration of the urn, we have
\[\frac{X_n}{n} \xrightarrow{n \rightarrow \infty} \lambda_1 v_1, \]
almost surely, and independently of the initial configuration of the urn. 
\end{thm}
 
Note that when $\mu$ is finitely supported  (so that, for some integer $M > 0$, $\mu := \sum_{i=1}^{M} \mu(w_i) \delta_{w_i}$) the number of possible face types in the complex is finite. We denote the (finite) set of possible face types by $\cCd^{f} \subseteq \cCd$. 
Moreover, the empirical distribution of face types corresponds to the distribution of balls in a generalised P\'olya urn; where the colours correspond to the types of the $(d-1)$-faces, and the activities are the fitnesses.
In each step, we draw a ball of type $x$ in the urn with probability proportional to its fitness $f(x)$, choose a weight $W$ independently according to $\mu$, and 
add $d$ new balls of respective types $x_{i \leftarrow W}$, for $i\in \{0, \ldots, d-1\}$. In Model \textbf{B} we also remove the ball we drew from the urn. 

Let $X_n = (X_{x}(n)), \, {x \in \cCd^{f}}$ denote the  vector  whose coordinate $X_{x} (n)$ counts the number of balls of type $x$ in the urn after $n$ steps. For $x \in \cCd^{f}$ and $k \in \{1, \ldots, M\}$, let $n_x(k)$ be the number of entries in $x$ equal to $w_k$. We call $x \neq x'$ neighbours if $x'$ can be obtained from $x$ by changing exactly one entry $\ell_1 = \ell_1(x, x')$ into $\ell_2 = \ell_2(x, x')$.

In Model \textbf{A}, this urn has the following replacement rule:
\[
\xi_{x x'} = \begin{cases} \sum_{k=1}^M n_x(k) \mathbf{1}_{\{w_k\}}(W) & x = x', \\
  n_{x}(\ell_1) \mathbf{1}_{\left\{w_{\ell_2(x, x')}\right\}}(W) & \text{if } x, x' \text{ are neighbours},\\
 0 & \text{otherwise}; \end{cases}
\]
whilst in Model \textbf{B} the replacement rule is
\[
\xi_{x x'} = \begin{cases} \sum_{k=1}^M n_x(k) \mathbf{1}_{\{w_k\}}(W) - 1 & x = x', \\
  n_{x}(\ell_1) \mathbf{1}_{\left\{w_{\ell_2(x, x')}\right\}}(W) & \text{if } x, x' \text{ are neighbours},\\
 0 & \text{otherwise}. \end{cases}
\]
If we define the matrix $A_{x x'} = f(x') \E{\xi_{x'x}}$, since $f>0$ it is easy to see that $A$ is irreducible. Thus we may deduce Theorem \ref{empir_limit} by applying Theorem \ref{thm:polyurnconvergence}.

\subsection{Hypothesis {\rm\textbf{H2}}}
In order to prove Theorem \ref{empir_limit} assuming \textbf{H2}, we show that $\Pi_n, n \geq 0$ is a measure-valued P\'olya process (MVPP), a concept recently introduced in~\cite{bantha} and~\cite{maimar}. We then apply results from~\cite{maivil}.
Let $\mathcal S$ be a locally compact Polish space and $\mathcal M(\mathcal S)$ be the set of finite, non-negative measures on~$\mathcal S$. Recall that $\mathcal M(\mathcal S)$ is also Polish when equipped with the Prokhorov metric (which metrises the weak topology when we view $\mathcal M(\mathcal S)$ as the dual of the space of bounded continuous functions from $\mathcal S$ to $\mathbb{R}$). For a given kernel $P$ on $\mathcal S$ and $\mu \in \mathcal M(\mathcal S)$, we define the  measure
\[ 
(\mu \otimes P)(\cdot) := \int_{\cS} P_x(\cdot) \,\mathrm d \mu (x).
\]
Thanks to, e.g., \cite[Section~4.1]{Kallenberg17} (and because of the local compactness) a random function $R$ with values in $\mathcal M(\cS)$ is a random variable (that is, measurable) if and only if, for all Borel sets $B \subseteq \cS$, $R(B)$ is a real-valued random variable.  
We call a family $R_x, x \in \cS$ of random variables with values in $\mathcal M(\cS)$ a \emph{random kernel} if, almost surely, $x \mapsto R_x$ is continuous. 
Note that, for a random kernel $R_x, \, x \in \cS$, the annealed quantity $\bar R_x(\cdot) = \E{R_x(\cdot)}$ is a kernel on $\cS$ (and the map $x \mapsto \bar R_x$ is continuous). We call two random kernels $R_x, R_x'$ for $x \in \cS$ independent if, for all $x \in \cS$, the random measures $R_x, R'_x$ are independent. 

\begin{df} Let $(R_x^{\sss (n)}, x \in \cS)_{n \geq 1}$ be a sequence of i.i.d.\ random kernels.
The 
measure-valued P\'olya process with  $m_0 \in \mathcal M(\mathcal S)$ satisfying $m_0(\mathcal S) > 0$, replacement kernels  $(R_x^{\sss (n)}, x \in \cS)_{n \geq 1}$ 
 and non-negative weight kernel $P$ is the sequence of random non-negative measures $(m_n)_{n \geq 0}$ defined recursively as follows:
 given $m_{n-1}, n \geq 1$:
\begin{itemize}
    \item [(i)] Sample a random variable $\xi$ from $\cS$ according to the probability measure
    \[\frac{(m_{n-1}\otimes P)(\,\cdot\,)}{(m_{n-1}\otimes P)(\cS)}.\]   
    \item [(ii)] Set $m_n = m_{n-1} + R_\xi^{\sss (n)}$.
\end{itemize}
\end{df}

The next lemma allows us to express the empirical distribution of the $(d-1)$-faces in Model \textbf{A} as an MVPP.
\begin{lem}\label{lem:mvpp}
For all $n\geq 1$ and $x\in \cCd$ let 
\[R_{x}^{\sss (n)}
= \sum_{i=0}^{d-1} \delta_{x_{i\lf W_n}}.
\] 
The sequence $\Pi_n, n\geq 0$ is the MVPP with initial composition $\Pi_0$, replacement kernel $(R_x^{\sss (n)}, x \in \cCd)_{n\geq 1}$ and weight kernel $P_x = f(x)\delta_x$, $x\in \cCd$.
\end{lem}

\begin{proof}
Let $\sigma$ be the face chosen and subdivided at step $n$ and $\xi$ be its type. By construction, 
\[\Pi_n = \Pi_{n-1} + \sum_{i=0}^{d-1} \delta_{\xi_{i \lf W_n}} = \Pi_{n-1} + R^{\sss (n)}_\xi,\]
and, for all Borel sets $B\subseteq \cCd$,
\[\Prob{\xi \in B| \Pi_{n-1}} = \frac{\sum_{\sigma \in \C{n}{}{d-1}} f(\sigma)  \delta_{\omega(\sigma)}(B)}{\sum_{\sigma \in\C{n}{}{d-1}} f(\sigma)}=\frac{(\Pi_{n-1} \tns P)(B)}{(\Pi_{n-1}\tns P)(\cCd)}. 
\]
This concludes the proof. 
\end{proof}

We now state ~\cite[Theorem~1]{maivil}. We will apply this theorem to  the MVPP $\Pi_n, n \geq 0$ to deduce Theorem~\ref{empir_limit}. We require the following definitions. 
For an i.i.d.\ sequence of random kernels
$(R^{\sss (n)}_x, x \in \cS)_{n \geq 1}$ and a weight kernel $P$, let $\bar R_x(\cdot) = \E{R^{\sss (1)}_x(\cdot)}$ and
\[Q_x^{(n)}(\cdot) := (R_{x}^{\sss (n)}\tns P)(\cdot) = \int_{\mathcal S} P_y(\cdot) \,\dd R_{x}^{\sss (n)}(y)
\quad\text{ and }
\bar Q_x(\cdot)
:=(\bar R_x \tns P) (\cdot)
= \int_{\mathcal S} P_y(\cdot) \,\mathrm d \bar R_{x}(y).
\]
\begin{thm}[Mailler \& Villemonais~\cite{maivil}]\label{th:MV}
Let $(m_n)_{n\geq 0}$ be the MVPP on $\mathcal S$ with initial composition $m_0$, replacement kernel $(R_x^{\sss (n)}, x \in \cS)_{n\geq 1}$ and weight kernel~$P$.
Assume that:
\begin{itemize}
\item[\emph{\bf A1}] For all $x\in\mathcal S$, $\bar {Q}_x(\mathcal S)\leq 1$, and there exists a probability distribution $\eta \neq \delta_0$ on $[0,\infty)$ such that, for all $x \in \mathcal S$, the law of $Q_x^{\sss (1)}(\cS)$ stochastically dominates $\eta$. 
\item[\emph{\bf A2}] The space $\cS$ is compact.
\item[\emph{\bf A3}] Denote by $(X_t)_{t\geq 0}$ the continuous-time Markov process defined on $\cS \cup \{\varnothing\}$ absorbed at $\varnothing$ with infinitesimal generator 
 given by $\bar Q_x - \delta_x + (1-\bar Q_x(\cS))\delta_{\varnothing}$.
There exists a probability distribution $\nu$ such that
\[
\mathbb P_x(X_t\in  \cdot|X_t\neq \varnothing)\to \nu(\cdot),
\]
 with respect to the total variation distance on $\cCd$ uniformly over $x \in \cCd$.
 \item[\emph{\bf A4}] For all bounded and continuous functions $g : \cS \to \mathbb R$, the functions $x \mapsto \int_{\cS} g(y) d \bar R_x(y)$ and $x \mapsto \int_{\cS} g(y) d \bar Q_x(y)$ are continuous.
\end{itemize}
Then, almost surely as $n\to\infty$, $m_n/n$ converges to~$\nu\otimes \bar R$ with respect to the weak topology on $\mathcal M(\cS)$. 
\end{thm}

\begin{proof}[Proof of Theorem~\ref{empir_limit}, assuming \textbf{H2}]
The idea of the proof is to apply Theorem~\ref{th:MV} to the MVPP $(\Pi_n)_{n\geq 0}$ (see Lemma~\ref{lem:mvpp}).
In this set-up, we have, for all $x\in\cCd$,
\[Q_x^{\sss (n)}(\cdot) 
=(R_x^{\sss (n)} \otimes P)(\cdot)
= \sum_{i=0}^{d-1} f(x_{i\lf W_n})\,\delta_{x_{i\lf W_n}}(\cdot),
\]
and \[ \bar Q_x(\cdot)
= (\bar R_x \otimes P)(\cdot)
= \mathbb E\left[\sum_{i=0}^{d-1} f(x_{i\lf W})\,\delta_{x_{i\lf W}}(\cdot) \right].
\]

In order to satisfy the normalization requirements in Theorem \ref{th:MV}, we consider a suitable rescaling.
We define 
\begin{equation} 
\label{eq:em-def}
M = d \cdot \mathbb E [f(\mathbf{1}_{0 \lf W})],
\end{equation} 
and for all $n\geq 0$, set 
${\Pi}_n' 
= {\Pi_n}/{M}$.
It is immediate (using Lemma~\ref{lem:mvpp}) that $({\Pi}_n')_{n\geq 0}$ is a MVPP with weight kernel 
$P$ whose replacement kernel and
associated $Q$-kernel 
are given by
\[\mathcal R_x^{\sss (n)} = \frac{R_x^{\sss (n)}}{M}, \quad 
\mathcal Q_x^{\sss (n)}= \frac{Q_x^{\sss (n)}}{M}.
\]
The corresponding annealed kernels are defined analogously by
$\bar {\mathcal R}_x(\cdot) = \E{\mathcal R_x^{\sss (1)} (\cdot)}$ and $\bar {\mathcal Q}_x(\cdot) = 
\E{\mathcal Q^{\sss (1)}_x(\cdot)}$.
Note that, by monotonicity of $f$ in all its coordinates, and symmetry,
\[\sup_{x\in\cCd} \mathbb E\left[\sum_{i=0}^{d-1} f(x_{i\lf W})\right]
\leq d  \cdot \mathbb E\big[f(\mathbf{1}_{0 \lf W})\big],\]
implying that, for all $x \in\cCd$, $\bar {\mathcal Q}_x (\cCd)\leq 1$. 
We also have that, for all $x \in\cCd$, by monotonicity of $f$
\[\mathcal Q^{\sss (1)}_x(\cCd)\geq \frac{d \cdot f(\bs 0)}M
\stackrel{\eqref{eq:em-def}}{=} \frac{d \cdot f(\bs 0)}{d \cdot \mathbb E [f(\mathbf{1}_{0 \lf W})]}
\geq \frac{f(\bs 0)}{f(\bs 1)}>0,\]
implying that Assumption \textbf{A1} of Theorem~\ref{th:MV} is satisfied with $\eta = \delta_{f(\bs 0)/f(\bs 1)}$. 
Assumption \textbf{A2} is immediately satisfied since $\cCd$ is compact. Next, 
as $\int_{\cCd} g(y) \mathrm d \bar R_x(y) = \sum_{i=0}^{d-1} \E{g(x_{i \lf W})}$, continuity of $x \mapsto \int_{\cCd}
  g(y) \mathrm d   \bar R_x(y)$ for a bounded and continuous function $g : \cCd \to \R$ is immediate. 
	Analogously, one can prove the statement for the $Q$-kernel and establish Assumption \textbf{A4} as the rescaling leaves continuity properties unaltered.

\noindent It thus remains to check that the rescaled P\'olya process $({\Pi}_n')_{n\geq 0}$ satisfies Assumption \textbf{A3}. 
Let $(X_t)_{t\geq 0}$ be the jump-process with infinitesimal generator $\bar {\mathcal Q}_x - \delta_x + (1-\bar{\mathcal Q}_x(\cCd)) \delta_{\varnothing}$, for all $x\in\cCd$. By definition, when $X_t$ sits at $x$, it jumps to $\varnothing$ at rate
\[
1 - \frac1M\sum_{i=0}^{d-1} \mathbb E[f(x_{i\lf W})],
\]
and, at rate $\frac1M\sum_{i=0}^{d-1} \mathbb E[f(x_{i\lf W})]$, it jumps to a random position chosen according to the probability distribution
\[\frac{\sum_{i=0}^{d-1} \mathbb E[f(x_{i\lf W})\delta_{x_{i\lf W}}(\cdot)]}{\sum_{i=0}^{d-1} \mathbb E[f(x_{i\lf W})]}.\]
Thus, in total, $X$ jumps at rate $1$ at all times. In particular, discrete skeleton and jump times of the process are independent.

To prove {\bf A3}, we apply~\cite[Theorem~3.5 and Lemma~3.6]{CV17} (where we take $t_1=t_2=1$ - note that, although this is not clear in the current version of \cite{CV17}, $t_1$ and $t_2$ need to be positive) to the jump process $(X_t)_{t\geq 0}$. Since $X$ is a pure jump process and satisfies the strong Markov property, condition (F0) in~\cite[Theorem~3.5]{CV17} is satisfied. It is therefore enough to prove that
there exist a set $L\subseteq \cCd$ 
and a probability measure $\varrho$ on $L$ such that:
\begin{itemize}
\item[{\bf B1}] There exist $c_1>0$ such that, for all $x\in L$, $\mathbb P_x(X_1\in\cdot)\geq c_1\varrho(\cdot \cap L)$, where $\mathbb P_x(\cdot)$ denotes the probability measure associated with the Markov process $X$
initiated by $x$. 
\item[{\bf B2}] There exist $0<\gamma_1<\gamma_2$ such that 
\[\sup_{x\in \cCd}\mathbb E_x[\gamma_1^{-\tau_L\land \tau_\varnothing}]<+\infty,
\text{ and }
\gamma_2^{-t}\mathbb P_x(X_t\in L)\to+\infty \text{ when }t\to+\infty \,(\forall x\in L),
\]
where $\tau_\varnothing$ and $\tau_L$ stand for the respective hitting times of $\varnothing$ and $L$.
\item[{\bf B3}] There exists $c_2>0$ such that
\[\sup_{t\geq 0}\frac{\sup_{y\in L}\mathbb P_y(t<\tau_{\varnothing})}{\inf_{y\in L}\mathbb P_y(t<\tau_{\varnothing})}\leq c_2.\]
\end{itemize}

In order to prove the above, we define the partial order `$\preccurlyeq$' on $\cCd$ such that for $x, y \in \cCd$, $x\preccurlyeq y$ if and only if, for all $i \in \{0, \ldots, d-1\}$, $x_i\leq y_i$ 
 (recall that the coordinates of $x$ and $y$ are ordered in increasing order). We then define $L=L(\varepsilon) = \{x\in \cCd \colon x\preccurlyeq (1-\varepsilon)\boldsymbol 1\}$. 
 
\medskip 

\noindent
{\bf Proof of {\bf B1}:} We denote by $(\sigma_i)_{i\geq 1}$ the random jump-times of $X$. In order for these times to be well-defined for all $n \geq 1$, we let the process jump from $\varnothing$ to $\varnothing$ at rate one. Fix a Borel set $B\subseteq \cCd$.
Then, by monotonicity and symmetry, we have 
\[\mathbb P_x(X_{\sigma_1} \in B) 
= \frac1M \sum_{i=0}^{d-1} \mathbb E[f(x_{i\leftarrow W}) {\boldsymbol 1}_{B}(x_{i\leftarrow W})]
\geq \frac{f(\boldsymbol 0)}M \sum_{i=0}^{d-1} \mathbb P(x_{i\leftarrow W}\in B).\]
 
By the strong Markov property, we have 
\[\Prob[x]{X_{\sigma_2} \in B | X_{\sigma_1} = x'} 
= \frac1M \sum_{i=0}^{d-1} \mathbb E[f(x'_{i\leftarrow W}) {\boldsymbol 1}_{B}(x'_{i\leftarrow W})]
\geq \frac{f(\boldsymbol 0)}M \sum_{i=0}^{d-1}\Prob{x'_{i\leftarrow W}\in B},\]

so that,
\begin{linenomath}
\begin{align*}
\int_{\cCd} \Prob[x]{X_{\sigma_2} \in B | X_{\sigma_1} = x'} \mathbb P_x(X_{\sigma_1} \in \dd x') &\geq \int_{\cCd} \frac{f(\boldsymbol 0)}M \sum_{i=0}^{d-1}\Prob{x'_{i\leftarrow W'}\in B} \Prob[x]{X_{\sigma_1} \in \dd x'}
\\ &\geq \left(\frac{f(\boldsymbol 0)}M\right)^2 \sum_{0 \leq i, j \leq d-1} \Prob{(x_{j \lf W})_{i \lf W'} \in B}
\end{align*}
\end{linenomath}
for i.i.d copies $W, W'$.
Iterating this argument,  we obtain
\[ 
\mathbb P_x(X_{\sigma_d} \in B) \geq \left(\frac{f(\boldsymbol 0)}{M}\right)^d \sum_{i_0, \ldots, i_{d-1} \in \{0, \ldots, d-1\}^d} \Prob{\left(\left(\left(x_{i_0 \lf W_0}\right)_{i_1 \lf W_1}\right) \ldots\right)_{i_{d-1} \lf W_{d-1}} \in B},
\]
where $W_0, \ldots, W_{d-1}$ are i.i.d.\ random variables with law $\mu$. Let $W_{(0)} \leq W_{(1)} \leq \ldots \leq W_{(n)}$ denote the order statistics of $W_0, \ldots, W_{d-1}$. Then, for an appropriate (random) choice of $i_0, \ldots, i_{d-1}$ we have 
$\left(\left(\left(x_{i_0 \lf W_0}\right)_{i_1 \lf W_1}\right) \ldots\right)_{i_{d-1} \lf W_{d-1}} = (W_{(0)}, \ldots, W_{(d-1)})$. Therefore 
\begin{linenomath}
\begin{align*} 
\mathbb P_x(X_{\sigma_d} \in B) & 
\geq \left(\frac{f(\boldsymbol 0)}{M}\right)^d \E { \sum_{i_0, \ldots, i_{d-1} \in \left\{0, \ldots, d-1\right\}^d}  {\boldsymbol 1}_B \left(\left(\left(\left(x_{i_0 \lf W_0}\right)_{i_1 \lf W_1}\right) \ldots\right)_{i_{d-1} \lf W_{d-1}}\right)    } \\
& \geq \left(\frac{f(\boldsymbol 0)}{M}\right)^d
\Prob{(W_{(0)}, \ldots, W_{(d-1)})\in B}.
\end{align*}
\end{linenomath}
As the probability that $X$ jumps exactly $d$ times before time~$1$ is positive and skeleton and jump times are independent (since $X$ always jumps with rate $1$), {\bf B1} is satisfied with $\varrho$ being the probability distribution induced by $\mu^{\otimes d}$ restricted to $L$ in the natural way.

\medskip

\noindent
{\bf Proof of {\bf B2}:} 
For $x \in \cCd$, let $n_{x}(x_i)$ denotes the number of co-ordinates of $x$ equal to $x_i$. $X$ jumps from a position $x$ such that $x_i> 1-\varepsilon$ to a position $x_{i\leftarrow v}$ for some $v\leq 1-\varepsilon$ at rate 
\[\frac{n_{x}(x_i)\mathbb  E\big[f(x_{i\leftarrow W})\boldsymbol 1_{W\leq 1-\varepsilon}\big]}M
\geq 
\frac{n_{x}(x_i)\mathbb E[f(\mathbf{0}_{0 \lf W})\boldsymbol 1_{W\leq 1-\varepsilon}]}{M}
=:n_{x}(x_i)\varpi_\varepsilon,\]
for all $i \in \{0, \ldots, d-1\}$ (where we have applied the symmetry and monotonicity of $f$).
Similarly, the walk jumps from a position $x$ such that $x_i\leq 1-\varepsilon$ to a position $x_{i\leftarrow v}$ for some $v> 1-\varepsilon$ at rate
\[\frac{n_{x}(x_i)\mathbb  E\big[f(x_{i\leftarrow W})\boldsymbol 1_{W> 1-\varepsilon}\big]}M
\leq \frac{n_{x}(x_i)\mathbb E[f(\bs 1_{0\leftarrow W})\bs 1_{W> 1-\varepsilon}]}{M}=: n_{x}(x_i)\vartheta_\varepsilon,\]
for all $i \in \{0, \ldots, d-1\}$.
Let $\mathscr{C}(X_t)$ denote the number of coordinates of $X_t$ that are larger than $1-\eps$, where we set $\mathscr{C}(\varnothing) = 0$. 
Consider a pure jump Markov process with rates given in 
Figure \ref{fig:jump-rates}.
\begin{figure}[H]
\begin{center}
\begin{tikzpicture}[->,>=stealth',shorten >=1pt,auto,node distance=2.3cm,semithick]
  \tikzstyle{every state}=[fill=white,draw=black,text=black,minimum size=0.9cm]
  \node[state] (A) {0};
  \node[state] (B) [right of=A] {$1$};
  \node[state] (C) [right of = B] {$2$};
  \node[state] (D) [right of=C]{$3$};
  \node[minimum size=1cm] (E) [right of=D] {$\cdots \cdots$};
  \node[state] (F) [right of=E] {$d-2$};
  \node[state] (G) [right of=F] {$d-1$};
  \node[state] (H) [right of=G] {$d$};
  \path (B) edge [bend left] node {$\varpi_\varepsilon$} (A);
  \path (C) edge [bend left] node {$2\varpi_\varepsilon$} (B);
  \path (B) edge [bend left] node {$(d-1) \vartheta_\varepsilon$} (C);
  \path (D) edge [bend left] node {$3 \varpi_\varepsilon$} (C);
 \path (C) edge [bend left] node {$(d-2) \vartheta_\varepsilon$} (D);
  \path (F) edge [bend left] node {$2\vartheta_\varepsilon$} (G);
  \path (G) edge [bend left] node {$(d-1)\varpi_\varepsilon$} (F);
  \path (G) edge [bend left] node {$\vartheta_\varepsilon$} (H);
  \path (H) edge [bend left] node {$d\varpi_\varepsilon$} (G);
\end{tikzpicture}
\end{center}
\caption{Jump rates of the associated Markov chain $N^\eps$.}
\label{fig:jump-rates}
\end{figure}
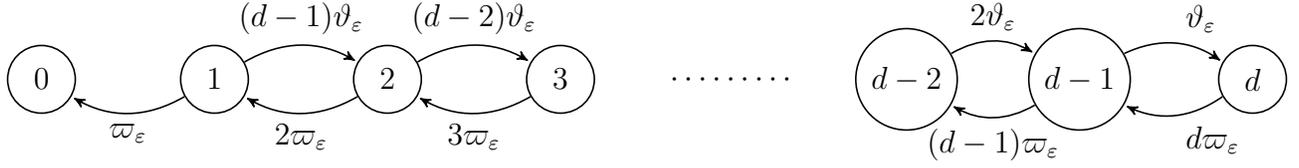
If, for some $t \geq 0$, this Markov chain has the same non-zero value as $\mathscr{C}(X_t)$, then it jumps upwards (resp.\ downwards) at a faster (resp.\ lower) rate than $\mathscr{C}(X_t)$. 
This observation motivates the following lemma whose proof is given in
Appendix \ref{ap:couplingb2}. 
Note that $\tau_L \wedge \tau_\varnothing$ is the first time $t$ when $\mathscr{C}(X_t) = 0$. \begin{lem} \label{lem:couplingXN}
For all sufficiently small $\varepsilon > 0$, there exists a coupling of the process $X$ with a realisation $N^\varepsilon$ of the Markov process with jump rates given in Figure \ref{fig:jump-rates} and $N^\varepsilon_0 = \mathscr C(X_0)$ such that, $\mathscr{C}(X_t) \leq  N^{\varepsilon}_t$ for all $t \leq \tau_L \wedge \tau_\varnothing$.
\end{lem}

(This lemma is where we use the assumption $\mu(\{1\}) =0$.) By Lemma \ref{lem:couplingXN}, we deduce that 
\begin{equation} \label{eq:b2couplingmono}
\Prob[x]{\tau_L \wedge \tau_\varnothing \geq t} \leq 
\Prob[\mathscr C(x)]{N^{\eps}_t \neq 0}.
\end{equation}
Here, we use the notation $\mathbb P_\ell$, $\ell \in \{0, \ldots, d\}$ to indicate that the Markov process $N^\eps_t, t \geq 0$ is initiated at position $\ell$.
Note that, since $\mu$ does not contain an atom at $1$, we have $\vartheta_\varepsilon\to 0$ and  $\varpi_\varepsilon \to \mathbb E[ f(\bs 0_{0\leftarrow W})]/M=:\varpi_0 \in (0,1]$ as $\varepsilon \to 0$. Therefore,  as $\varepsilon\to 0$ the generator $\mathcal L_{\varepsilon}$ of the Markov chain $N^\eps$ converges to the generator 
\[\mathcal L= \begin{pmatrix}
0 & 0 & \ldots & &  & 0\\
\varpi_0 & -\varpi_0 & 0 & \ldots & & 0\\
0 & 2\varpi_0 & -2\varpi_0 & 0 & \ldots & 0\\
 & & \ddots & \ddots & &  \\
  & & & \ddots & \ddots &  \\
 0 & \ldots &  & 0 & d\varpi_0 & -d\varpi_0
\end{pmatrix}\] 
whose eigenvalues are $0, -\varpi_0, \ldots, -d\varpi_0$ (and thus whose spectral gap is $\varpi_0$), and whose stationary distribution on $\{0, \ldots, d\}$ is given by $\delta_0$ as $0$ is an absorbing state.  

Since $\mathcal L_\varepsilon$ converges entry-wise to $\mathcal L$ when $\varepsilon \to 0$, their respective characteristic polynomials converge, and thus the eigenvalues of $\mathcal L_\varepsilon$ converge to the eigenvalues of $\mathcal L$. Since the eigenvalues of  $\mathcal L$ are all distinct it follows that for $\varepsilon$ sufficiently small all eigenvalues of $\mathcal L_\varepsilon$ are simple. Thus, $\mathcal L_\varepsilon$ is diagonalisable, and may be written as $\mathcal L_\varepsilon = V_\varepsilon^{-1} D_{\varepsilon} V_\varepsilon$, where $D_{\varepsilon}$ is a diagonal matrix consisting of the eigenvalues of $\mathcal L_\varepsilon$, and the rows of $V_\varepsilon^{-1}$ are the corresponding unit-norm (left) eigenvectors. This condition allows us to apply \cite[Theorem~3.1]{mitrophanov}. Since, for each $\eps > 0$, the stationary distribution of $N^{\eps}$ is $\delta_{0}$, for all $\ell \in \{0, \ldots, d\}$ and for all $t\geq 0$, 
\begin{equation}
\label{eq:conc-mark-proc}
|\mathbb P_\ell(N^\eps_t=0) - 1|
\leq C(\varepsilon) 
\mathrm e^{-\rho(\varepsilon) t},
\end{equation}
where $\rho(\varepsilon)$ is the spectral gap of the generator of $N^{\varepsilon}$, and  $C(\varepsilon)=\|V_\varepsilon\|_{\infty}\|V_\varepsilon^{-1}\|_{\infty}$ (where $\| \cdot \|_{\infty}$ denotes the $\infty$-norm, i.e.\ maximum absolute row sum). 
Note that as $\eps \to 0$,
$\rho(\varepsilon) \to \varpi_0$.
Moreover, using the basis of unit-norm (left) eigenvectors introduced above, we have $C(\varepsilon) = \|V_\eps\|_{\infty}\|V^{-1}_\eps\|_{\infty} \to C:=\|V\|_{\infty}\|V^{-1}\|_{\infty}$, as $\eps \to 0$, where the rows of $V^{-1}$ are a basis of unit-norm (left) eigenvectors of $\mathcal L$.
Now, by Equation~\eqref{eq:b2couplingmono} and \eqref{eq:conc-mark-proc}, we have
\begin{equation} \label{eq:prob-neq}
\mathbb P_x(\tau_L \wedge \tau_\emptyset \geq t) \leq \mathbb P_{\mathscr C(x)}(N^\eps_t \neq 0) = 1- \mathbb P_{\mathscr C(x)}(N^\eps_t = 0) \leq C(\varepsilon) \exp(-\rho(\varepsilon)t).
\end{equation}

Therefore, for all $\gamma_1 < 1$ and $x\in \cCd$, using the fact that $\log{\gamma_1} < 0$ in the second equality,
\begin{linenomath}
\begin{align*}
\mathbb E_x[\gamma_1^{-\tau_L\wedge \tau_{\varnothing}}]
&= 1 + \int_1^{\infty} \mathbb P_x(\gamma_1^{-\tau_L\wedge \tau_{\varnothing}}\geq u)\mathrm du
= 1 + \int_1^{\infty} \mathbb P_x\left(\tau_L\wedge \tau_{\varnothing}\geq \frac{\log u}{\log(\nicefrac1{\gamma_1})}\right)\mathrm du\\
&\stackrel{\eqref{eq:prob-neq}}{\leq} 1 + \int_1^{\infty} C(\varepsilon) u^{-\rho(\varepsilon)/\log (\nicefrac1{\gamma_1})}\mathrm du
<+\infty 
\end{align*}
\end{linenomath}
as long as $\log (\nicefrac1{\gamma_1})< \rho(\eps)$.
Also note that, for all $x\in L$,
\[
\mathbb P_x(X_t\in L)
\geq \mathbb P_x(X_{\sigma_i}\in L \text{ for all }0\leq i\leq N(t)),
\]
where $N(t)$ is the number of jumps of $X$ by time $t$, and
\begin{linenomath}
\begin{align*}
\mathbb P_x(X_{\sigma_1}\in L)
&=\frac1M\sum_{i=0}^{d-1} \mathbb E[f(x_{i\leftarrow W})\boldsymbol 1_{x_{i\leftarrow W}\in L}]
=\frac1M\sum_{i=0}^{d-1} \mathbb E[f(x_{i\leftarrow W})\boldsymbol 1_{W\leq 1-\varepsilon}]\\
&\stackrel{\eqref{eq:em-def}}{\geq} \frac{\mathbb E[f(\bs 0_{0\leftarrow W})\bs 1_{W\leq 1-\eps}]}{\mathbb E[f(\bs 1_{0\leftarrow W})]}
=:\chi_\varepsilon.
\end{align*}
\end{linenomath}
Since the walk jumps at rate one, we have that the number of jumps before time $t$ is Poisson distributed with parameter $t$. As skeleton and jump times are independent, it follows that, for all $x\in L$,
\[
\mathbb P_x(X_t\in L)
\geq \mathbb P_x(X_{\sigma_i}\in L \text{ for all }0\leq i\leq N(t))
\geq \mathbb E[\chi_\eps^{N(t)}]
= \mathrm e^{-(1-\chi_\varepsilon) t}.
\]
If $1-\chi_\varepsilon<\log(\nicefrac1{\gamma_2})$, then $\gamma_2^{-t}\mathbb P_x(X_t\in L)\to+\infty$ as required.
In other words, {\bf B2} is satisfied if we can choose $\gamma_1<\gamma_2<1$ such that
\[1-\chi_\varepsilon<\log(\nicefrac1{\gamma_2})< \log(\nicefrac1{\gamma_1})
< \rho(\eps).
\]
As $\eps\to 0$, we have $\chi_\varepsilon \to \mathbb E[f(\bs 0_{0\leftarrow W})]/\mathbb E[f(\bs 1_{0\leftarrow W})] = d\varpi_0$ while $\rho(\eps) \to \varpi_{0} > 1- d\varpi_0$
by Equation~\eqref{eq:extra_cond}.
It is thus possible to choose $\eps$ small enough such that $1-\chi_\varepsilon < \rho(\eps)$.
For this value of $\varepsilon$, a choice of $\gamma_1$ and $\gamma_2$ is possible, which concludes the proof of {\bf B2}.

\medskip

\noindent
{\bf Proof of {\bf B3}:}  
We require the following coupling lemma, where we adopt the convention that $\varnothing \preccurlyeq x$ for all $x \in \cCd$ and $\varnothing \preccurlyeq \varnothing$. We defer the proof of this lemma to Appendix \ref{ap:couplingb3}
\begin{lem}\label{lem:couplingb3}
Let $x, y \in \cCd$ with $x \preccurlyeq y$. There exist processes $X^{(x)}, X^{(y)}$ such that $X^{(x)}$ is distributed as $X$ with respect to $\mathbb P_x$ and
$X^{(y)}$ is distributed as $X$ with respect to $\mathbb P_y$ satisfying that, almost surely, $X^{(x)}_t \preccurlyeq X^{(y)}_t$ for all $t \geq 0$.
\end{lem}

Thanks to Lemma~\ref{lem:couplingb3}, we have that, if $x\preccurlyeq y\in \cCd$, then
\begin{equation} \label{eq:coupling-order}
\mathbb P_x(t<\tau_\varnothing)\leq
\mathbb P_y(t<\tau_\varnothing).
\end{equation}
In particular, this implies that
\[\inf_{y\in L} \mathbb P_y(t<\tau_\varnothing) = \mathbb P_{\boldsymbol 0}(t<\tau_\varnothing), \text{ and } \sup_{y\in L} \mathbb P_y(t<\tau_\varnothing) = \mathbb P_{(1-\varepsilon)\boldsymbol 1}(t<\tau_\varnothing).\]
Also, since $1 \in \Supp{(\mu)}$, with positive probability, every coordinate of $(X_t)_{t\geq 0}$ is at least $1-\varepsilon$ after $d$ jumps.
 If we denote this probability by $\kappa_1 = \kappa_1(\varepsilon)$, we obtain
\[\mathbb{P}_{\boldsymbol 0}(t<\tau_\varnothing)
\geq \mathbb{P}_{\boldsymbol 0}(\sigma_d < t <\tau_\varnothing)
\geq\kappa_1 \mathbb{P}_{\boldsymbol 0}(\sigma_d < t <\tau_\varnothing | (1-\varepsilon)\boldsymbol 1\preccurlyeq X_{\sigma_d}),\]
where $(1-\varepsilon)\boldsymbol 1\preccurlyeq X_{\tau_d}$ denotes the event that all coordinates of $X_{\tau_d}$ are at least $1-\varepsilon$. Next, observe that for all $t \leq 1$, 
\[\frac{\Prob[(1-\varepsilon)\boldsymbol 1]{t < \tau_\varnothing}}{\Prob[\boldsymbol 0]{t < \tau_\varnothing}} \leq \frac{1}{\mathrm{e}^{-1}} = \mathrm e,\]
since the probability the process has not jumped by time $t$ is $\mathrm e^{-t}$. Now, by Equation~\eqref{eq:coupling-order} and the strong Markov property, for Lebesgue almost all $0 \leq u \leq 1 < t$, 
\begin{linenomath}
\begin{align*}\Prob[\boldsymbol{0}]{t < \tau_{\varnothing}|(1- \eps) \boldsymbol{1} \preccurlyeq X_{\sigma_d}, \, \sigma_d = u} 
&= \E[\bs 0]{\Prob[X_{\sigma_d}]{t-u < \tau_{\varnothing}}|(1- \eps) \boldsymbol{1} \preccurlyeq X_{\sigma_d}, \, \sigma_d = u} \\ 
& \geq \Prob[(1-\varepsilon)\boldsymbol 1]{t-u < \tau_{\varnothing}} \geq \Prob[(1-\varepsilon)\boldsymbol 1]{t < \tau_{\varnothing}}.\end{align*}
\end{linenomath}
Thus, for $t > 1$, since jump times and skeleton are independent 
\begin{linenomath}
\begin{align*}
\mathbb{P}_{\boldsymbol 0}(t<\tau_\varnothing)
 & \geq \kappa_1 \mathbb{P}_{\boldsymbol 0}(\sigma_d \leq 1 \leq t <\tau_\varnothing | (1-\varepsilon)\boldsymbol 1\preccurlyeq X_{\sigma_d}) \\
 & \geq \kappa_1 \int_{0}^{1} \Prob[0]{t < \tau_{\varnothing}|(1- \eps) \boldsymbol{1} \preccurlyeq X_{\sigma_d}, \, \sigma_d = u} \Prob[\boldsymbol{0}]{\sigma_d \in \dd u \, | \, (1- \eps) \boldsymbol{1} \preccurlyeq X_{\sigma_d}} \\ &
 = \kappa_1 \int_{0}^{1} \Prob[0]{t < \tau_{\varnothing}|(1- \eps) \boldsymbol{1} \preccurlyeq X_{\sigma_d}, \, \sigma_d = u} \Prob[\boldsymbol{0}]{\sigma_d \in \dd u}
 \\ &= \kappa_1 \Prob[\boldsymbol{0}]{\sigma_d < 1} \Prob[(1- \eps) \boldsymbol{1}]{t-u < \tau_{\varnothing}} \geq \kappa_1 \Prob[\boldsymbol{0}]{\sigma_d < 1} \Prob[(1- \eps) \boldsymbol{1}]{t-u < \tau_{\varnothing}}.
\end{align*}
\end{linenomath}
Thus, if we set $\Prob[\boldsymbol{0}]{\sigma_d < 1} := \kappa_2$, taking $c_2 = \max{\left\{\frac{1}{\kappa_1\kappa_2}, \mathrm e\right\}}$ completes the proof. 
\end{proof} 
\subsection{The Star Process}
\label{subsec:starprocess}
In the remainder of this section, we revisit the companion Markov process $(S^{*}_n)_{n\geq 0}$ defined in Subsection~\ref{sub:star_proc_def}.
We wish to apply the same theory of P\'olya processes to study the distribution of $(d-1)$-faces in $(S^{*}_n)_{n \geq 0}$. Note, however, that by definition, in this process every face contains the central vertex of $S^{*}_0$. Therefore, if the central vertex has weight $x$, we may view the empirical distribution of $(d-1)$-faces as a measure on $\cCdd$, which represents the weights of the other vertices in the $(d-1)$-faces in $S^*_n$.  

Thus, we can interpret the evolving empirical measure as a homogeneous Markov process $(S_n)_{n \geq 0}$ on $\cC' := [0, \infty) \times \mathcal M(\cCdd)$ (recall that $\mathcal M (\cCdd)$ is the space of non-negative, finite measures on $\cCdd$). 

Given $S_n = (x, \nu) \in \cC'$ for some $n\geq 0$: 
\begin{itemize}
    \item [(i)] Set $c^* = \int_{\cCdd} f((x, y)) \dd \nu(y)$ and sample $z \in \cCdd$ according to the distribution admitting density $f((x, y))/c^*$ with respect to $\nu$.
    \item [(ii)] Let $W$ be a random variable with distribution $\mu$ which is independent of the past of the process. Then, set 
    \[S_{n+1} = \begin{cases} (x, \nu+ \sum_{i=0}^{d-2} \delta_{z_i \lf W}), & \text{in Model }\textbf{A}, \\  (x, \nu+ \sum_{i=0}^{d-2} \delta_{z_i \lf W} - 
    \delta_z), & \text{in Model }\textbf{B}.\end{cases}\]
\end{itemize}
For a completely rigorous definition, we also set $S_{n+1} = S_n$ if the measure component of $S_n$ is the zero measure and step (i) cannot be executed.
We write $\mathbb P^*_{(x, \nu)}, \mathbb E^*_{(x, \nu)}$ for probabilities and expectations, respectively with respect to this process when the initial state $S_0$ satisfies $S_0 = (x, \nu)$. Note that this implies that the first component of $S_n$ remains equal to $x$ for all $n \geq 0$. Let us write $\bS_n$ for the measure component of $S_n$. Then, provided that
$\bS_0$ is a non-trivial sum of Dirac measures, we have 
\[\bS_n(\cCdd) = \begin{cases} (d-1)n+\bS_0(\cCdd), & \text{in Model }\textbf{A}, \\  (d-2)n+ \bS_0(\cCdd), & \text{in Model }\textbf{B}.\end{cases}\]
Upon identifying faces with their types, we may consider $\St{i}{\C{n}{}{}}$  as a $\cC'$-valued random variable by separating the weight of vertex $i$ from the remaining vertices.  Let $\tau_0 = i$ and, for $n \geq 1$, let $\tau_n$ be the $n$-th time, the randomly chosen face in the construction of  $(\C{m}{}{})_{m \geq 0}$ contains vertex $i$. Formally, letting $\sigma_n$ denote the face chosen and subdivided in step $n$, we have 
\[\tau_n := \inf \{m > \tau_{n-1}: i \in \sigma_m\}, \quad n \geq 1.\]
It is easy to see that $\tau_n < \infty$ almost surely for all $n \geq 1$. Indeed, under either Hypothesis \textbf{H1} or \textbf{H2}, we have $Z_n = F(\mathcal{K}_{n}) \leq f_{\max} (n + |\mathcal{K}_{0}^{\sss (d-1)}|)$, and if $\tau_{k-1} \leq n < \tau_{k}$,  $F(\St{i}{\C{n}{}{}}) \geq f_{\min}(d-1)(k-1)$. Therefore, (analogous to proof of the Borel-Cantelli lemma) one can bound the probability
\[
    \Prob{\tau_{k} = \infty| \tau_{k-1} = N} \leq \prod_{j=N+1}^{\infty} \left(1 - \frac{f_{\min}(d-1)(k-1) }{f_{\max} (j + |\mathcal{K}_{0}^{\sss (d-1)}|)}\right) \leq \mathrm e^{-\sum_{j=N+1}^{\infty} \frac{f_{\min}(d-1)(k-1) }{f_{\max} (j + |\mathcal{K}_{0}^{\sss (d-1)}|)}} = 0;
\]
and the result follows by induction on $k$. 
\\
Furthermore, the sequence of random variables
\[\left(W_i, \sum_{\sigma \in \St{i}{\C{\tau_n}{}{}}} \delta_{\omega(\sigma) \setminus \{W_i\}}\right)_{n \geq 0}\] is equal in distribution to $S_n, n \geq 0$ with respect to $\mathbb P^*_{(x, \nu)}$, when the configuration $(x, \nu)$ is chosen with respect to the law of $(W_i, \sum_{\sigma \in \St{i}{\C{i}{}{}}} \delta_{\omega(\sigma) \setminus \{W_i\}})$.

 Let $\varphi : [0, \infty) \times \cCd \to \cC'=[0, \infty)\times \mathcal M(\mathcal C_{d-2})$ be the map 
 \begin{equation}\label{eq:def_phi}
 \varphi(w, x) = \left(w, \sum_{i=0}^{d-1} \delta_{\tilde x_i}\right),
 \end{equation}
 where we recall that for all $x\in\mathcal C_{d-1}$,  $\tilde x_i\in\mathcal C_{d-2}$ is the vector $x$ from which we have removed the $i$-th coordinate. 
 We also let $\psi : [0,\infty) \times \cCdd \rightarrow \cCd$ be such that 
 \begin{equation} \label{eq:def-psi}
 \psi(w,x) = w \cup x,
 \end{equation} 
 where we recall that $w\cup x$ is obtained by adding a coordinate equal to $w$ to the vector~$x$, and reordering the coordinates of the obtained vector in non-decreasing order.
 For $(w,\nu)\in \cC'$, we define the fitness 
 
 \begin{equation} \label{def:weight_s} F(w,\nu) = \int_{\cCd} f\, \dd\psi_{*}(\delta_w \otimes \nu),\end{equation}
 where $\psi_{*}(\delta_w \otimes \nu)$ is the pushforward of $\delta_w \otimes \nu$ under $\psi$
 (in other words, $\psi_{*}(\delta_w \otimes \nu)$ is the distribution of $\psi(w,X)$ where $X\in \mathcal C_{d-2}$ is a $\nu$-distributed random variable).
Note that, when $S_0$ is chosen according to the law of $(W, Y_{\infty})$, we have $(F(S_n))_{n \geq 0} = (F(S^*_n))_{n \geq 0}$ in distribution.
Moreover, for any $x \in \Supp{(\mu)}$, 
assuming \textbf{H1*} or \textbf{H2*}, Theorem \ref{empir_limit} implies almost sure convergence of the rescaled measure valued process $(\frac{1}{n}\bS_n)_{n > 0}$ on $\cCdd$ to a positive limiting measure depending on $x$. Thus, we get the following:
\begin{thm}\label{th:star}
Assume {\rm\textbf{H1*}} or {\rm\textbf{H2*}} and recall the definition of $\psi$ in Equation~\eqref{eq:def-psi}, and that $\bS_n$ denotes the measure-valued component of the star process $S_n\in \cC'$. Then, for any $x \in \Supp{(\mu)}$, there exists a positive measure $\me_x^*$ on $\cCd$, such that, 
for any positive non-zero measure $\nu \in \mathcal M (\cCdd)$, we have
\[\frac1n\psi_{*}(\delta_x \otimes \bS_n) \to \me_x^*, \quad \mathbb P^*_{(x, \nu)} \text{-almost surely as }n\to\infty,\]
with respect to the weak topology. 
\end{thm}

By continuity and boundedness of $f$, this implies that
\[\frac{F(S_n)}{n} \to \lambda_x^* := \int_{\cCd}  f(y) \,\mathrm d m_x^*(y)>0, \quad \mathbb  P^*_{(x, \nu)} \text{-almost surely when }n\to\infty.\]
This yields Proposition \ref{prop:fitness-star-process} by setting the initial state to be $S_0 = \varphi(w, Y_\infty)$, where $Y_{\infty}$ is defined in Proposition \ref{prop:MC} and $\varphi$ in Equation~\eqref{eq:def_phi}.

\section{The degree profile} 
\label{sec:deg-prof}

In this section, we determine the degree profile associated with the sequence of simplicial complexes 
$(\C{n}{}{})_{n \geq 0}$. 
Throughout this section we assume that the conclusion of Theorem \ref{empir_limit} holds, and that $f: [0,1]^d \to (0, \infty)$ is continuous and symmetric. Recall that $f_{\max} = \sup \{f(x) : x \in \cCd\}$.

Let $\pi^*$ be the distribution of the random variable $\varphi(W, Y_\infty)$, where $W$ and $Y_\infty$ are independent, $W$ is $\mu$-distributed and 
$Y_\infty$ is as in Proposition \ref{prop:MC}.
We prove the following equivalent of Theorem \ref{thm:small_degrees_vague} 
the only difference in the two statements comes from the fact that we now use the notation of the previous section; in particular the process~$S$ with initial distribution $\pi^*$ is equal in distribution to the process $S^*$ from Theorem~\ref{thm:small_degrees_vague}): 

\begin{thm} \label{thm:small_degrees}
Denote by $\D{k}{n}$ the (random) number of vertices of degree $d+k$ in $\C{n}{}{}$. 
For all $k\geq 0$, we have, in probability,
\[\lim_{n\to \infty} \frac{1}{n} \D{k}{n} = 
\mathbb E^*_{\pi^*} \left[\frac{\lambda}{F(S_k)+\lambda}\prod_{\ell=0}^{k-1} \frac{F(S_\ell)}{F(S_\ell)+\lambda}\right] = p_k\]
with $\lambda$ as in Proposition \ref{prop:partition}. \end{thm}

Note that $(\p_k)_{k \geq 0}$ is a probability distribution on the set of non-negative integers. Indeed, given $F(S_0), F(S_1), \ldots$ consider a sequence of independent events, where, for $i \geq 0$,  the $i$-th event occurs with probability $\lambda / (F(S_i)+\lambda)$. Then, the integrand is the probability that the $k$-th event is the first to occur. (The fact that, almost surely, 
some event in the sequence occurs follows from  boundedness of $f$, which implies that $F(S_{\ell})$ grows at most linearly.) The probability distribution $(\p_k)_{k \geq 0}$ may thus be regarded as a \emph{generalised geometric distribution}.

\medskip 
The proof of Theorem \ref{thm:small_degrees} consists of two steps. First, we show convergence of the corresponding mean, and then we study
the variance of $\D{k}{n}$ to show convergence in probability by an application of Chebychev's inequality.

To prove convergence of the mean, it is convenient to consider only vertices that arrive after a certain time  $\eta n$ where 
$\eta > 0$ is a small constant; 
this allows us to work in the asymptotic regime of the sequence of simplicial complexes. 
Hence,  let $N_k^\eta(n)$ be the number of vertices 
of degree $k+d$ in $\C{n}{}{}$ which arrived after time $\eta n$. Obviously, 
\[N_k^{\eta}(n) \leq N_k(n) \leq \eta n + N_k^{\eta}(n),\]
and therefore, 
\[  
\lim_{\eta \to 0}\limsup_{n \to \infty} \frac{1}{n} \left| \Ex{\C{0}{}{}}{{\D{k}{n}}}-  \Ex{\C{0}{}{}}{N_k^{\eta}(n)} \right| =0.
\]

Most of this section is thus devoted to proving that, for all $k\geq 0$,
\[\lim_{\eta \to 0} \lim_{n \to \infty} \frac{1}{n}  \Ex{\C{0}{}{}}{ N_k^\eta(n)}
= \p_k.\]
Let $\dgl{i}{n}$ be the number of vertices which are neighbors of node $i$ and arrived after node $i$. 
By construction, we have that
\begin{equation} \label{sum_s_up} \Ex{\C{0}{}{}}{N^\eta_k(n)} = \sum_{\eta n< i\leq n-k} \PIdx{\C{0}{}{}}{\dgl{i}{n} = k}. \end{equation}
Henceforth, we use the simplified notation $\cI_k = \{i_1, \ldots, i_k\}$ for a collection of natural numbers $i < i_1 < \ldots < i_k \leq n$. Let $\cE_i(\cI_k)$ denote the event that $i \sim \ell$ for all $\ell \in \cI_k$ and
$i \not\sim \ell$ for all $\ell \notin \cI_k$ with $\ell \in \{i+1, \ldots, n\}$. We have 
\begin{equation} \label{eq:summation}
\PIdx{\C{0}{}{}}{\dgl{i}{n}=k } = \sum_{\cI_k \in {\{i+1, \ldots, n\} \choose k}} \PIdx{\C{0}{}{}}{\cE_i(\Ind_k)},
\end{equation}
where ${\{i+1, \ldots, n\} \choose k}$ denotes the set of all subsets of $\{i+1, \ldots, n\}$ of size $k$. (For $k=0$, the sum consists only of the term $\cI_0 = \varnothing$.)

\subsubsection*{Proof Overview}

The proof now consists of three steps. First, we provide sufficient upper and lower bounds 
for $\mathbb P(\dgl{i}{n}=k)$ using the fact that, for 
$i \geq \eta n$, with high probability, for all $i \leq j \leq n$, the partition function 
$Z_j$ is concentrated around $\lambda j$ - see Proposition \ref{prop:partition}.
On the event of concentration, we can estimate the probability that insertions in the star of vertex $i$ or its complement occur. 
Second, we use Proposition  \ref{prop:MC} to incorporate the stationary distribution of the Markov chain $Y_n$ when passing to the limit as $n \to \infty$. Third, we apply
a probabilistic argument to evaluate the sums in \eqref{sum_s_up} and \eqref{eq:summation}. In the following section, we state the necessary tools to work out the second and third step. The corresponding proofs are deferred to the appendix in order not to disrupt the flow of the main arguments.
\\\\
The main part of the work
involves exploiting the concentration of the partition function to derive 
upper and lower bounds on (a variant of) $\Prob{\cE_i(\Ind_k)}$ and are proved in Subsections~\ref{subsec:upper} and \ref{subsec:lower}, respectively. Note that the proof of the upper bound in Subsection \ref{subsec:upper} is significantly less technical,  as we can `drop' the event of concentration from probability computations.

We recommend the reader to study this case first. Second moment calculations which allow one to deduce stochastic convergence from convergence of the mean in Theorem \ref{thm:small_degrees} are presented in Subsection~\ref{subsec:second} and follow the arguments developed in Subsection~\ref{subsec:upper} closely.
The proof of the lower bound in Subsection~\ref{subsec:lower} requires additional work, due, in part, to the `migration' of faces into the complement on the event of an insertion into the star of vertex $i$ (see Figure~\ref{fig:modelbcompdim3}). We deal with this technical challenge by bounding the total number of `descendants' of a small number of faces 
by the sum of geometrically distributed random variables with sufficiently small success probability (Lemmas~\ref{lem:help1} and \ref{lem:help2}). The rest of the proof then involves some lengthy computations to control error terms.

\subsection{Technical Lemmas}
This subsection is dedicated to the statements of some technical lemmas that will be important in the sequel. We defer the proofs of these lemmas to the appendix.

\subsubsection{A continuity statement for the star Markov chain}

The following result concerns continuity of the $k$-step transition kernel of the star Markov chain with respect to its starting point. Recall that the function~$F$ is defined in Equation~\eqref{eq:def_F}, and the process $(S_n)_{n \geq 0}$ has been defined in Subsection~\ref{subsec:starprocess}.

\begin{prop} \label{help1}
Let $k \geq 0, w \in [0, \infty)$ and $x, x_1, x_2, \ldots \in \cCd$ with 
$x_n \to x$. Then, in the sense of weak convergence on $[0, \infty)^{k+1}$, we have, as $n \to \infty$, 
\[\mathbb P^*_{\varphi(w, x_n)}((F(S_0), F(S_1), \ldots, F(S_k)) \in \cdot) \to \mathbb P^*_{\varphi(w,x)}((F(S_0), F(S_1), \ldots, F(S_k)) \in \cdot).\]
\end{prop}

\subsubsection{Evaluating sums} 
For all $\alpha_0, \ldots, \alpha_k \geq 0$, and $0 \leq \eta < 1$, let
\[
\Gamma_n(\alpha_0, \ldots, \alpha_k, \eta) := \frac{1}{n} \sum_{\eta n < i_0 < \cdots < i_k \leq n} \prod_{\ell=0}^{k-1}\left( \left(\frac{i_\ell}{i_{\ell+1}} \right)^{\alpha_\ell} \cdot \frac{1}{i_{\ell+1} -1} \right) \left(\frac{i_k}{n} \right)^{\alpha_k}.  
\]

\begin{lem} \label{lem:prob-sum}
Uniformly in $\alpha_0, \ldots, \alpha_k \geq 0$, $0 \leq \eta \leq 1/2$, asymptotically in $n$ we have
\[
\Gamma_n(\alpha_0, \ldots, \alpha_k, \eta) = \prod_{\ell=0}^{k} \frac{1}{\alpha_\ell + 1} + \theta(\eta) + O \left(\frac{1}{n^{1/(k+2)}} + \frac{\sum_{j =0}^k \alpha_j \log^{k+1}(n) }{n} \right).
\]
Here, $\theta(\eta)$ is a term satisfying $ |\theta(\eta)| \leq M\eta^{1/(k+2)}$ for some universal constant $M$ depending only on $k$.
\end{lem}
This lemma is proved in Section~\ref{sub:proof_lem_prob_sum}. An immediate corollary is the following:
 \begin{cor} \label{cor:summation}
 For $\alpha_{0}, \ldots, \alpha_{k}, \beta_{0}, \ldots, \beta_{k-1} \geq 0$, $0 \leq \eta \leq 1/2$, asymptotically in $n$ we have
 \begin{linenomath}
 \begin{align*}
&\frac 1 n \sum_{\eta n < i\leq n} \sum_{\Ind_k \in {\{i+1, \ldots, n\} \choose k} }  \prod_{\ell=0}^{k-1}\left( \left(\frac{i_\ell}{i_{\ell+1}} \right)^{\alpha_\ell} \cdot \frac{\beta_{\ell}}{i_{\ell+1} -1} \right) \left(\frac{i_k}{n} \right)^{\alpha_k}\\
& \hspace {2cm} = \frac{1}{\alpha_{k} +1}\prod_{j=0}^{k-1} \frac{\beta_{j}}{\alpha_{j}+1}  +  \theta'(\eta) + O \left(\frac{1}{n^{1/(k+2)}}\right).
\end{align*}
\end{linenomath}
Here, $\theta'(\eta)$ is a term satisfying $ |\theta'(\eta)| \leq M'\eta^{1/(k+2)}$ for some constant $M'$ depending only on $k$ and $\beta_0, \ldots, \beta_{k-1}$, and the constant in the big $O$-term may depend on $\alpha_{0}, \ldots, \alpha_{k}, \beta_{0}, \ldots, \beta_{k-1}$. 
\end{cor}

\subsection{Convergence of the mean: bounds from above} \label{subsec:upper}
The aim of this section is to prove that
\begin{equation}\label{eq:aim_UB}
    \lim_{\eta \to 0} \limsup_{n \to \infty} \E{N_k^\eta(n)}/n \leq \p_k.
\end{equation}
Recall that we write $\Pi_{n} = \sum_{\sigma \in \K_{n}^{(d-1)}} \delta_{w(\sigma)}$ for the empirical distribution of the weights of all $(d-1)$-faces in the complex after the $n$th step. 
We also define the partition function associated with $\K_n$ by $Z_n = \int_{\cCd} f(x) \dd\Pi_{n}(x)$.
For $\eps > 0$ and $n \geq 0$ and natural numbers $N_1 \leq N_2$, we let 
\begin{equation} \label{def:a}
\mathcal G_\eps (n)= \left\{ \left| Z_n- \lambda n \right| < \eps \lambda n \right\} \quad\text{ and }\quad \mathcal G_\eps(N_1, N_2) = \bigcap_{n=N_1}^{N_2}\mathcal G_\eps(n). \end{equation}
Moreover, for $n \geq 1$, we denote by  $\mathscr G_{n}$ the $\sigma$-field generated by $(\C{\ell}{}{}, W_\ell), 1 \leq \ell \leq n$ containing all information about the process up to time $n$.

By Proposition~\ref{prop:partition} (and Egorov's theorem), for any $\delta, \varepsilon > 0$, there exists $N' = N'(\delta, \varepsilon)$ such that, for all $n \geq N'$,  $\pb(\mathcal G_\varepsilon(N', n)) \geq 1- \delta$.
Therefore, for all $n \geq N'/\eta$, we have
\begin{linenomath}
\begin{align} 
\Ex{\C{0}{}{}}{N_k^\eta(n)}& \leq \Ex{\C{0}{}{}}{N_k^\eta(n) \mathbf 1_{\mathcal G_{\eps}(N', n)}}+ n (1-\pb (\mathcal G_\eps(N', n)))\notag\\
& \leq \sum_{\eta n < i\leq n} \sum_{\cI_k \in {\{i+1, \ldots, n\} \choose k}} 
\pb(\cE_i(\cI_k) \cap \mathcal G_\eps(i, n)) +\delta n.\label{eq:up_bound_exp}
\end{align}
\end{linenomath}
Finally, for $x > 0$ and $\alpha \in \mathbb R$, we set $\alpha_{\pm x} := \alpha(1\pm x)$. The following proposition gives an upper bound on the summands in the right-hand side of Equation~\eqref{eq:up_bound_exp}. 
For simplicity, we subsequently write \begin{equation}\label{eq:def_sti}
    \St{i}{\C{n}{}{}} 
    = \Big(W_i, \sum_{\sigma \in \St{i}{\C{n}{}{}}} \delta_{\omega(\sigma) \setminus \{W_i\}}\Big)
    \in \cC'=[0, \infty)\times \mathcal M(\mathcal C_{d-2})
\end{equation} when considering the $\cC'$-valued random variable associated with the star around vertex~$i$ at step~$n$. 
\begin{prop} \label{prop:upper_bounds}
Let $0 < \varepsilon, \eta  \leq 1/2$. As $n \to \infty$, uniformly in $\eta n < i \leq n-k$, $\cI_k \in {\{i+1, \ldots, n\} \choose k}$ and the choice of $\varepsilon$, we have
 \begin{linenomath}
\begin{align*}
&\PIdx{\C{0}{}{}}{\cE_i(\cI_k)\cap \mathcal G_\eps(i, n)}\\ &\leq  \left(1+O\left(\frac 1 n \right)\right) 
\E{\mathbb E^*_{\St{i}{\C{i}{}{}}} \left[ \left(\frac{i_k}{i_{k+1}} \right)^{F(S_k)/\lambda_{+\eps}} \cdot \prod_{\ell=0}^{k-1} \left(\frac{i_\ell}{i_{\ell+1}} \right)^{F(S_\ell)/\lambda_{+\eps}}  \frac{F(S_{\ell})}{\lambda_{-\eps} (i_{\ell+1}-1)} \right]}. 
\end{align*}
\end{linenomath}
\end{prop}
Applying Corollary~\ref{cor:summation} to this, we will deduce the following upper bound.
\begin{cor} \label{cor:1}
Let $0 < \delta, \varepsilon, \eta \leq 1/2$. Then, there exists $N = N(\delta, \varepsilon, \eta)$ such that, for all $n \geq N$,
\[
\frac{\Ex{\C{0}{}{}}{N_k^\eta(n)}}{n} \leq (1+\delta) \left(\frac{1+\eps}{1-\eps}\right)^k \mathbb E_{\pi^*}^* \left[ \frac{\lambda_{+\eps}}{F(S_k)+\lambda_{+\eps}} \prod_{\ell=0}^{k-1} \frac{F(S_\ell)}{F(S_\ell)+\lambda_{+\eps}}\right] + C \eta^{1/(k+2)} +  \delta,
\]
where the constant $C$ may depend on $k, f$ and $\mu$ but not on $n$ and not on the choices of $\delta, \varepsilon, \eta$.
In particular, Equation~\eqref{eq:aim_UB} is satisfied.
\end{cor}
To prove Proposition \ref{prop:upper_bounds}, let $0 < \varepsilon, \eta \leq 1/2$. For $\eta n < i \leq n$ and
$\cI_k \in {\{i+1, \ldots, n\} \choose k}$, set $i_0:=i, i_{k+1} := n+1$. Then, for $j \in \{i+1, \ldots, n\}$, let
\begin{equation} \label{def:D} \cD_{j} := 
\begin{cases}
\{i\sim j\}, & \text{ if } j\in\Ind_k,\\
\{i\not\sim j\}, & \text{ otherwise},
\end{cases} \quad\text{ and }\quad
\tilde \cD_j  = \cD_j \cap \mathcal G_{\eps}(j),
\end{equation}
where $\mathcal G_{\eps}(j)$ is defined as in Equation~\eqref{def:a}.

For simplicity, we write $D_j$ and $\tilde D_j$ for the indicator random variables $\mathbf{1}_{\cD_j}$ and $\mathbf{1}_{\tilde \cD_j}$ respectively.
Note that $\mathcal E_i(\Ind_k) \cap \mathcal G_\eps(i,n) = \bigcap_{j=i}^n \tilde \cD_j$.
To estimate the probability of this event, we shall decompose the indices $j \in \{i, \ldots, n\}$ into groups $\{i_{\ell}, \ldots, i_{\ell+1}-1\}$ for
$\ell\in \{0, \ldots, k\}$.  More precisely, we define
\begin{equation} \label{def:xl} X_\ell = \Ex{\C{0}{}{}}{\prod_{j=i_{\ell}+1}^{n} \tilde D_j  \bigg | \mathscr G_{i_\ell}}  \tilde D_{i_\ell}, 
\quad \ell\in \{0, \ldots, k\}.\end{equation}
To prove Proposition \ref{prop:upper_bounds}, 
we need to estimate $\Ex{\C{0}{}{}}{X_0} = \PIdx{\C{0}{}{}}{\bigcap_{j=i}^n \tilde \cD_j}$. 

From the tower property of conditional expectation, it follows that
\begin{equation} 
X_\ell =  \Ex{\C{0}{}{}}{\prod_{j=i_{\ell}+1}^{i_{\ell+1}-1} \tilde D_j~X_{\ell+1}  \bigg |\mathscr G_{i_\ell}} \tilde D_{i_\ell}, 
\quad \ell\in \{0, \ldots, k-1\}, \label{rec:x} \end{equation}
which suggests a backwards recursive approach. 
We need more notation: for $S \in \cC'=[0, \infty)\times \mathcal M(\mathcal C_{d-2})$ and $\ell \in \{0, \ldots, k\}$, we let 
\begin{equation} \label{eq:h-def}
h_\ell(S) = \prod_{j=i_\ell+1}^{i_{\ell+1}-1} \left(1- \frac{F(S)}{\lambda_{+\eps}(j-1)}\right),
\end{equation}
where $F$ is as defined in \eqref{def:weight_s},
and set 
\begin{equation} \label{def:ff} 
f_k = h_k \quad \text{ and }\quad 
f_\ell(S) = \frac{F(S)}{\lambda_{-\eps}(i_{\ell+1}-1)} h_\ell(S), \quad 0 \leq \ell \leq k-1. 
\end{equation}
For the sake of presentation, we do not indicate that the definitions of the $\tilde {\mathcal D}_j, X_\ell, h_\ell, f_\ell$ depend on $\Ind_k$ and $\varepsilon$. 
\begin{lem} \label{lem:u} For $\ell\in \{0, \ldots, k\}$, and $h_{\ell}$ as defined in Equation~\eqref{eq:h-def}, we have
\begin{equation} \label{eq:lem-4.7}
\Ex{\C{0}{}{}}{\prod_{j=i_{\ell}+1}^{i_{\ell+1}-1} \tilde D_j  \bigg | \mathscr G_{i_\ell}} 
\leq h_\ell(\St{i}{\C{i_\ell}{}{}}).
\end{equation}
\end{lem}
Recall that, by definition, $\St{i}{\C{i_\ell}{}{}}\in\mathcal C'$ (see Equation~\eqref{eq:def_sti}) and thus $h_\ell(\St{i}{\C{i_\ell}{}{}})$ is well-defined.
\begin{proof}
First note that for all $\ell\in\{1, \ldots, k\}$, by the tower property,
\begin{linenomath}
\begin{align*}
 \Ex{\C{0}{}{}}{\prod_{j=i_{\ell}+1}^{i_{\ell+1}-1} \tilde D_j  \bigg |\mathscr G_{i_\ell}} 
 & =  \Ex{\C{0}{}{}}{
 \E{\tilde D_{i_{\ell+1}-1} \bigg | \mathscr G_{i_{\ell+1}-2}} \prod_{j=i_{\ell}+1}^{i_{\ell+1}-2} \tilde D_j   \bigg |\mathscr G_{i_\ell}} \leq  \Ex{\C{0}{}{}}{\E{D_{i_{\ell+1}-1} \bigg | \mathscr G_{i_{\ell+1}-2}} \prod_{j=i_{\ell}+1}^{i_{\ell+1}-2} \tilde D_j   \bigg |\mathscr G_{i_\ell}},
\end{align*}
\end{linenomath}
where we have used the fact that, by definition, $\tilde {\mathcal D}_j  = \mathcal D_j \cap \mathcal G_{\eps}(j)$ and thus $\tilde D_j \leq D_j$ (recall that the latter denote the indicators of the events $\tilde{\mathcal{D}}_j$ and $\mathcal{D}_j$ respectively). If $i_{\ell+1} -1 \notin \mathcal{I}_{k}$ we have that
\[\E{D_{i_{\ell+1}-1} \big| \mathscr G_{i_{\ell+1}-2}}
=\mathbb P(\mathcal D_{i_{\ell+1}-1}\big| \mathscr G_{i_{\ell+1}-2})
=1 - \frac{F(\St{i}{\C{i_{\ell+1}-2}{}{}})}{Z_{i_{\ell+1}-2}},
\]
where we recall that $F(\St{i}{\C{i_{\ell+1}-2}{}{}})$ is the sum of the fitnesses of the faces in the complex that contains node~$i$ at time~$i_{\ell +1} -2$ (see Equation~\eqref{eq:def_F}).
Thus,
\begin{linenomath}
\begin{align*}
    \Ex{\C{0}{}{}}{\prod_{j=i_{\ell}+1}^{i_{\ell+1}-1} \tilde D_j  \bigg |\mathscr G_{i_\ell}}
& \leq \Ex{\C{0}{}{}}{ \left(1 - \frac{F(\St{i}{\C{i_{\ell+1}-2}{}{}})}{Z_{i_{\ell+1}-2}}  \right)    
\prod_{j=i_{\ell}+1}^{i_{\ell+1}-2} \tilde D_j   \bigg | \mathscr G_{i_\ell}} \\
&\leq \left(1 - \frac{F(\St{i}{\C{i_{\ell}}{}{}})}{\lambda_{+\eps}(i_{\ell+1}-2)}  \right)
\Ex{\C{0}{}{}}{\prod_{j=i_{\ell}+1}^{i_{\ell+1}-2} \tilde D_j  \bigg |\mathscr G_{i_\ell}},
\end{align*}
\end{linenomath}
where we recall that, by definition, $\lambda_{+\varepsilon} = \lambda(1+\varepsilon)$ and $F(\St{i}{\C{i_{\ell+1}-2}{}{}}) = F(\St{i}{\C{i_{\ell}}{}{}})$. In the last inequality, we have used the fact that on the event $\tilde{\mathcal D}_{i_{\ell+1}-2}$, we have $Z_{i_{\ell+1}-2}\leq \lambda_{+\varepsilon}(i_{\ell+1}-2)$.
Iterating the argument shows the claim.
\end{proof}
We now use the above lemma to derive an almost-sure upper bound for $X_\ell$.

\begin{prop} \label{prop:u}
For $\ell \in \{0, \ldots, k\}$, and $f_{\ell}$ as defined in Equation~\eqref{def:ff}, we have
\[X_\ell \leq  \mathbb E^*_{\St{i}{\C{i_\ell}{}{}}} \left[ \prod_{j=\ell}^k f_j(S_{j-\ell}) \right] \tilde D_{i_\ell}.\]
In particular, 
\[\Ex{\C{0}{}{}}{X_0} \leq \Ex{\C{0}{}{}}{\mathbb E^*_{\St{i}{\C{i}{}{}}} \left[ \prod_{j=0}^k f_j(S_{j}) \right]}.\]
\end{prop}
\begin{proof}
We proceed by backwards induction. 
For $\ell=k$, the statement is identical to the one in Lemma~\ref{lem:u}. 
Now, assume the claim holds for some $1 \leq \ell \leq k$. Using Equation~\eqref{rec:x} and the induction hypothesis in the second inequality, we get
\begin{linenomath}
\begin{align}
X_{\ell-1} & = \Ex{\C{0}{}{}}{\prod_{j=i_{\ell-1}+1}^{i_{\ell}-1} \tilde D_j~X_{\ell}  \bigg |\mathscr G_{i_{\ell-1}}} \tilde D_{i_{\ell-1}} \notag\\
& \leq \Ex{\C{0}{}{}}{ \Ex{\C{0}{}{}}{ \mathbb E^*_{\St{i}{\C{i_\ell}{}{}}} 
\left[ \prod_{j=\ell}^k f_j (S_{j-\ell}) \right] 
D_{i_\ell}  \bigg | \mathscr G_{i_{\ell}-1} }  
\prod_{j=i_{\ell-1}+1}^{i_{\ell}-1} \tilde D_j \bigg | \mathscr G_{i_{\ell-1}} } \tilde D_{i_{\ell-1}}.
\label{eq:display1}
\end{align}
\end{linenomath}
The event $\mathcal D_{i_\ell} = \{i_\ell\sim i\}$ indicates that an insertion has been made into $\St{i}{\C{i_\ell-1}{}{}}$. Therefore, conditionally on $\mathscr{G}_{i_{\ell}-1}$, on the event $\mathcal D_{i_\ell}$, 
the sequence $(S_0, \ldots, S_{k-\ell})$ initiated by $\St{i}{\C{i_\ell}{}{}}$ is equal in distribution to $(S_1, \ldots, S_{k-\ell+1})$
initiated by $\St{i}{\C{i_\ell-1}{}{}}$. Thus,
\begin{linenomath}
\begin{align}
    \Ex{\C{0}{}{}}{ \mathbb E^*_{\St{i}{\C{i_\ell}{}{}}} 
\left[ \prod_{j=\ell}^k f_j (S_{j-\ell}) \right] 
D_{i_\ell}  \bigg | \mathscr G_{i_{\ell}-1} }& = 
\Prob{\mathcal D_{i_\ell} | \mathscr G_{i_{\ell}-1} } \mathbb E^*_{\St{i}{\C{i_{\ell}-1}{}{}}} \left[\prod_{j=\ell}^k 
f_j(S_{j-\ell+1})\right] \notag\\
& = \frac{F(\St{i}{\C{i_{\ell}-1}{}{}})}{Z_{i_\ell-1}} \mathbb E^*_{\St{i}{\C{i_{\ell}-1}{}{}}} \left[\prod_{j=\ell}^k 
f_j(S_{j-\ell+1})\right].
\label{eq:display2}
\end{align}
\end{linenomath}
On the other hand, on the events $\bar{\mathcal D}_j$, $j \in \{i_{\ell-1}+1, \ldots, i_{\ell}-1\}$, we have $\St{i}{\C{i_{\ell}-1}{}{}} = \St{i}{\C{i_{\ell-1}}{}{}}$, and thus
$F(\St{i}{\C{i_{\ell}-1}{}{}}) =F(\St{i}{\C{i_{\ell-1}}{}{}})$. 
Combining~\eqref{eq:display1} and~\eqref{eq:display2} and the fact that on $\tilde{\mathcal D}_{i_\ell-1}$, $Z_{i_\ell-1}\geq \lambda_{-\varepsilon}(i_\ell-1)$ in the first inequality, we obtain
\begin{linenomath}
\begin{align*}
X_{\ell-1} & 
 \leq  \mathbb E^*_{\St{i}{\C{i_{\ell-1}}{}{}}} \left[\frac{F(S_{0})}{\lambda_{-\eps}(i_{\ell}-1)} \prod_{j=\ell}^k 
f_j(S_{j-\ell+1}) \right]  \Ex{\C{0}{}{}}{\prod_{j=i_{\ell-1}+1}^{i_{\ell}-1} \tilde D_j  \bigg |\mathscr G_{i_{\ell-1}}} 
\tilde D_{i_{\ell-1}} \\
& \stackrel{\eqref{eq:lem-4.7}}{\leq}  \mathbb E^*_{\St{i}{\C{i_{\ell-1}}{}{}}} \left[\frac{F(S_0)}{\lambda_{-\eps}(i_{\ell}-1)} \prod_{j=\ell}^k f_j(S_{j-\ell+1}) \right] h_{\ell-1}(\St{i}{\C{i_{\ell-1}}{}{}}) \tilde D_{i_{\ell-1}}  \\
&= \mathbb E^*_{\St{i}{\C{i_{\ell-1}}{}{}}} \left[ \prod_{j=\ell-1}^k f_j(S_{j-\ell+1}) \right] \tilde D_{i_{\ell-1}}. 
\end{align*}
\end{linenomath}
This concludes the induction argument, and thus the proof.
\end{proof}
The following elementary lemma is an easy consequence of Stirling's approximation (using  Equation~\eqref{eq:stirling_gamma_approx}), so we state it without proof.
\begin{lem} \label{simple_stir}
Let $\delta, C > 0$. Then, as $m \to \infty$, uniformly over $\delta m \leq a \leq b$ and $0 \leq \beta \leq C$, we have
\[\prod_{j=a+1}^{b-1} \left(1-\frac{\beta}{j-1}\right) =  \left(\frac{a}{b}\right)^\beta \left( 1 + O\left(\frac 1 m\right) \right).\]
\end{lem}
The statement of Proposition \ref{prop:upper_bounds} follows immediately from Proposition \ref{prop:u} and Lemma \ref{simple_stir}.
\begin{proof}[Proof of Corollary \ref{cor:1}]
In view of the statement of Proposition \ref{prop:upper_bounds}, it remains to replace $\St{i}{\C{i}{}{}}$ by its distributional limit $\varphi(W, Y_\infty)$ and to evaluate the sum over the possible values of $i, i_1, \ldots, i_k$. We start with the first task and show that, for any 
$0 < \delta, \varepsilon, \eta \leq 1/2$, there exists $N=N(\delta, \eta)$ such that, for all $\eta n < i \leq n-k, \cI_k \in {\{i+1, \ldots, n\} \choose k}$ and $n \geq N$, we have
\begin{linenomath}
\begin{align} \label{234}
&\PIdx{\C{0}{}{}}{\cE_i(\cI_k)\cap \mathcal G_\eps(i, n)} \nonumber \\ &\hspace{0.5cm}\leq  \left(1+\delta \right) \mathbb E_{\pi^*}^* \left[ \left(\frac{i_k}{i_{k+1}} \right)^{F(S_k)/\lambda_{+\eps}} \cdot \prod_{\ell=0}^{k-1} \left(\frac{i_\ell}{i_{\ell+1}} \right)^{F(S_\ell)/\lambda_{+\eps}}  \frac{F(S_{\ell})}{\lambda_{-\eps} (i_{\ell+1}-1)} \right]. 
\end{align}
\end{linenomath}

Note that the statement of the corollary immediately follows from this identity and Corollary \ref{cor:summation}. 
To verify the last statement, let $\pi_n^*$ be the law of $\St{n}{\C{n}{}{}}$ considered as $\cC'$-valued random variable, that is, $\varphi(W_n, Y_n)$ (see Equation~\eqref{eq:def_phi} for the definition of $\varphi$).  Thanks to Proposition \ref{prop:upper_bounds}, it is sufficient to prove that, uniformly in $\eta n < i < i_1 < i_2 <\ldots < i_k \leq n$ and $\eps \in (0,1/2]$, as $n \to \infty$ 

\begin{linenomath}
\begin{align} 
\nonumber &  \mathbb E^*_{\pi_i^*}\left[  \left(\frac{i_k}{i_{k+1}} \right)^{F(S_k)/\lambda_{+\eps}} \cdot \prod_{\ell=0}^{k-1} \left(\frac{i_\ell}{i_{\ell+1}} \right)^{F(S_\ell)/\lambda_{+\eps}}  F(S_{\ell}) \right]  \\  & \hspace{2cm} - 
\mathbb E^*_{\pi^*} \left[ \left(\frac{i_k}{i_{k+1}} \right)^{F(S_k)/\lambda_{+\eps}} \cdot \prod_{\ell=0}^{k-1} \left(\frac{i_\ell}{i_{\ell+1}} \right)^{F(S_\ell)/\lambda_{+\eps}}  F(S_{\ell}) \right]  \to 0. \label{ass2} \end{align}
\end{linenomath}
To this end, we prove the following stronger statement: uniformly in $\eta \leq x_0, \ldots, x_k \leq 1$ and the choice of $\eps$, as $n \to \infty$,
\[
 \mathbb E^*_{\pi_n^*}\left[  x_k^{F(S_k)/\lambda_{+\eps}} \cdot \prod_{\ell=0}^{k-1} x_\ell^{F(S_\ell)/\lambda_{+\eps}}  F(S_{\ell}) \right]
- \mathbb E^*_{\pi^*} \left[ x_k^{F(S_k)/\lambda_{+\eps}} \cdot \prod_{\ell=0}^{k-1} x_\ell^{F(S_\ell)/\lambda_{+\eps}}  F(S_{\ell}) \right]  \to 0.   
\]
By continuity of $\varphi$, Propositions \ref{prop:MC} and \ref{help1}, we have
$\mathbb P^*_{\pi_n^*}((F(S_0), \ldots, F(S_k)) \in \cdot) \to$ \sloppy $\mathbb P^*_{\pi^*}((F(S_0), \ldots, F(S_k)) \in \cdot)$ weakly.
Note that, for all $0 \leq \ell \leq k$, $F(S_\ell) \leq C$, where $C = (d+1)(k+1)f_{\max}$ 
(recall that $f_{\max}$ is the maximum of the fitness function $f$).
For all  $\eta \leq x_0, \ldots, x_k \leq 1$ and $0 \leq \varepsilon \leq 1/2$, the function $\J(y_0, \ldots, y_k) = x_k^{y_k/\lambda_{+\eps}} \cdot \prod_{\ell=0}^{k-1} x_\ell^{y_\ell/\lambda_{+\eps}}  y_\ell$ defined on $[0, C]^{k+1}$ satisfies 
\begin{equation} \label{eq:bound-jacob}
\| \nabla \J \| \leq \alpha_{\eta} := C^k \left(1-\log \eta/\lambda \right)
\end{equation}
uniformly in $x_0, \ldots, x_k, \eps$.
For any two probability distributions $\nu$ and $\nu'$ on $[0,C]^{k+1}$, 
let
\begin{linenomath}
\begin{align} \label{eq:dual-lipschitz}
\nonumber &d(\nu, \nu') = \sup_{g \in \mathcal F} \left |  \int g \dd \nu - \int g \dd \nu' \right|\\
 & \hspace{1cm}\text{ where }\mathcal F :=\{ g : [0,C]^{k+1} \to \R \mid \forall x, y \in [0,C]^{k+1} \quad |g(x) - g(y)| \leq \alpha_{\eta} \|x - y\|\}.
\end{align}
\end{linenomath}

It is well-known that $d(\nu_n, \nu) \to 0$ if and only if $\nu_n \to \mu$ weakly (see for example, Example 19, page 74 \cite{pollard}). 
 This concludes the proof of \eqref{ass2} and of the corollary.
\end{proof}

 \subsection{Stochastic convergence: second moment calculations} \label{subsec:second}
By counting the number of unordered pairs of vertices with degree $d+k$, arguments similar to those applied in Subsection~\ref{subsec:upper} allow us to compute asymptotically the second moment of $N^{\eta}_{k}(n)$ 
(recall this is the number of vertices 
of degree $k+d$ in $\C{n}{}{}$ that arrived after time $\eta n$). 
Note that
\[
 \E{(N^{\eta}_{k}(n))^2} = \sum_{\eta n < i,j \leq n}
\Prob{\dgl{i}{n} = k, \dgl{j}{n} = k}.  
\] 
We prove that 
\begin{equation} 
\label{aim:second} \lim_{\eta \to 0} \limsup_{n \to \infty}  \frac{\E{(N^{\eta}_{k}(n))^2}}{n^2} \leq \p_k^2. 
\end{equation} This shows that 
$\lim_{n \to \infty} \E{(N^{\eta}_{k}(n))^2} / n^2 = \p_k^2$ which is sufficient to deduce the convergence in probability stated in Theorem \ref{thm:small_degrees} from  convergence of the mean by a standard application of Chebychev's inequality.

Recall that we use the notation $\Ind_{k} = \left\{i_1, \ldots, i_{k}\right\}$ for a collection of natural numbers $i < i_1 < \ldots < i_{k} \leq n$. Similarly, we  write $\Jind_{k} = \left\{j_1, \ldots, j_{k}\right\}$ for a collection of natural numbers such that $j < j_1 < \ldots < j_k \leq n$. As before, we let $\mathcal E_i (\Ind_k)$ denote the event $i \sim \ell$ for $i < \ell \leq n$ if and only if $\ell \in \Ind_k$ and define the event $\mathcal E_j(\Jind_k)$ analogously for $j, j_1, \ldots, j_k$. 

With these definitions, we have 
\begin{equation}  \label{eq:second-moment-sum}
\E{(N^{\eta}_{k}(n))^2} = \sum_{\eta n < i,j \leq n} \sum_{\mathcal{I}_{k}, \mathcal{J}_{k}} \
\PIdx{\C{0}{}{}}{\mathcal E_i\left(\Ind_k\right) \cap \mathcal E_j \left(\Jind_k \right)},
\end{equation}
where the inner sum is over all $\Ind_k \in {\{i+1, \ldots, n\} \choose k}$ and $\Jind_k \in {\{j+1, \ldots, n\} \choose k}$.
As in Subsection~\ref{subsec:upper}, we fix $0 \leq \delta, \varepsilon \leq 1/2$ and choose $N'$ such that for all $n \geq N'$, 
$\PIdx{\C{0}{}{}}{\mathcal{G}_{\eps}(N', n)} \geq 1 - \delta$. 

Note that, on $\mathcal E_i (\Ind_k) \cap \mathcal E_j(\Jind_k)$,  if $\Ind_k \cap \Jind_k \neq \varnothing$ we either have $i = j$ or $i \sim j$.
 If $i=j$ then $\Ind_k = \Jind_k$, and  
the contribution of these terms to the right hand side of  \eqref{eq:second-moment-sum} is at most $\E{N^\eta_k(n)} \leq n$.
On the event $\{\dgl{i}{n} = k\}$ we have $F(\St{i}{\C{\ell}{}{}}) \leq (k+1)d f_{\max}$ for all $i+1 \leq \ell \leq n$. Therefore, for $\eta n < i < j \leq n$, we have
\begin{linenomath}
\begin{align*}
& \PIdx{\C{0}{}{}}{\left\{\dgl{i}{n} = k\right\} \cap \left\{\dgl{j}{n} = k\right\} \cap \{j \sim i\} \cap \mathcal G_{\eps}(i,n)} \\& \hspace{3.5cm}\leq \PIdx{\C{0}{}{}}{\left\{j \sim i\right\}\mid  \mathcal G_{\eps}\left(i, j-1\right), \dgl{i}{j-1} \leq k}  
 \leq \frac{(k+1)d f_{\max}}{\lambda_{- \eps} \eta n}. 
\end{align*}
\end{linenomath}
It follows that, for all $n$ sufficiently large (depending on $\delta, \varepsilon$ and $\eta$), 
\[\Ex{\C{0}{}{}}{(N^{\eta}_{k}(n))^2} \leq 2 \sum_{\eta n < i < j \leq n} \sum_{\Ind_{k} \cap \Jind_{k} = \varnothing} \PIdx{\C{0}{}{}}{ \mathcal E_i \left(\mathcal I_k\right) \cap \mathcal E_j\left(\mathcal J_k \right) \cap 
\mathcal G_{\eps}(i, n)} + \delta n^2 + Cn /\eta,\]
for a constant $C \geq 0$ which is independent of $n, \delta, \varepsilon$ and $\eta$.
The following proposition is the analogue of Proposition \ref{prop:upper_bounds}.
\begin{prop}\label{prop:upper_bounds_second_moment}
Let $0 < \eps, \eta \leq 1/2$. As $n \rightarrow \infty$, uniformly in $\eta n < i < j \leq n-k$, $\Ind_k \in {\{i+1, \ldots, n\} \choose k}$ and $\Jind_k \in {\{j+1, \ldots, n\} \choose k}$ with $\Ind_k \cap \Jind_k = \varnothing$ and the choice of $\varepsilon$, we have 
\begin{linenomath}
\begin{align*}
&\PIdx{\C{0}{}{}}{\mathcal E_i\left(\mathcal I_k\right) \cap \mathcal E_j\left(\mathcal J_k \right) \cap \mathcal G_{\eps}(i, n)}\nonumber \\ &\leq \left(1 + O\left (\frac 1 n \right ) \right) 
\mathbb E \Bigg [ \mathbb E^*_{\St{i}{\C{i}{}{}}}\left[ \left(\frac{i_k}{n}\right)^{F(S_k)/\lambda_{+\eps}}\cdot \prod_{\ell=0}^{k-1} \left(\frac{i_\ell}{i_{\ell+1}}\right)^{F(S_\ell)/\lambda_{+\eps}}\frac{F(S_{\ell})}{\lambda_{-\eps} (i_{\ell+1}-1)} \right] \nonumber \\ &\hspace{3.5cm}
\mathbb E^*_{\St{j}{\C{j}{}{}}}\left[\left(\frac{j_k}{n}\right)^{F(S_k)/\lambda_{+\eps}} \cdot \prod_{\ell=0}^{k-1} \left(\frac{j_\ell}{j_{\ell+1}}\right)^{F(S_\ell)/\lambda_{+\eps}}\frac{F(S_{\ell})}{\lambda_{-\eps} (j_{\ell+1}-1)} \right] \Bigg ] .
\end{align*}
\end{linenomath}
\end{prop}

The proof of this proposition is completely analogous to the proof of Proposition \ref{prop:upper_bounds} and relies on a backward induction argument and an application of Lemma \ref{simple_stir}. We omit the details as no new arguments are necessary at this point. 

We move on to show the following analogue of \eqref{234}: for any $0 < \delta, \varepsilon, \eta \leq 1/2$, there exists $N = N(\delta, \eta)$  such that, for all $n \geq N$, $\eta n <i < j \leq n-k$ and disjoint sets $\Ind_k, \Jind_k$, we have
\begin{linenomath}
\begin{align} \label{345} \nonumber&\PIdx{\C{0}{}{}}{\mathcal E_i\left(\mathcal I_k\right) \cap \mathcal E_j\left(\mathcal J_k \right) \cap \mathcal 
G_{\eps}(i, n)} \\ &\leq (1+\delta) \bigg (\mathbb E^*_{\pi^*} \left[ \left(\frac{i_k}{n}\right)^{F(S_k)/\lambda_{+\eps}}\cdot \prod_{\ell=0}^{k-1} \left(\frac{i_\ell}{i_{\ell+1}}\right)^{F(S_\ell)/\lambda_{+\eps}}\frac{F(S_{\ell})}{\lambda_{-\eps} (i_{\ell+1}-1)} \right]  \nonumber \\ &\hspace{1.5cm}
\mathbb E^*_{\pi^*} \left[\left(\frac{j_k}{n}\right)^{F(S_k)/\lambda_{+\eps}} \cdot \prod_{\ell=0}^{k-1} \left(\frac{j_\ell}{j_{\ell+1}}\right)^{F(S_\ell)/\lambda_{+\eps}}\frac{F(S_{\ell})}{\lambda_{-\eps} (j_{\ell+1}-1)}  \right] \bigg ).  
\end{align}
\end{linenomath}

The details are very similar to the approach in Subsection~\ref{subsec:upper}, and we only give the necessary additional results entering the proof.
\begin{prop}\label{prop:Ynm}
As $n, m \to \infty$ with $n  \neq m$, we have $(Y_n, Y_m) \to (Y_\infty, Y_\infty')$, in distribution, for independent random variables $Y_\infty, Y_\infty'$ both distributed according to $\pi^*$.
\end{prop}

\begin{proof}
This follows easily from Theorem \ref{empir_limit}. Let $g_1, g_2 : \cCd \to \R$ be bounded and continuous and $Y_\infty, Y_\infty'$ be independent realisations of $\pi^*$. 
We have
\begin{linenomath}
\begin{align} \left|\E{g_1(Y_n)g_2(Y_m)} -  \E{g_1(Y_\infty)g_2(Y_\infty')}\right| \label{start_bound}
& \leq \left|\E{g_1(Y_n)g_2(Y_m)} - \E{g_1(Y_n)}\E{g_2(Y_\infty')} \right| \\
& + \left| \E{g_1(Y_n)}\E{g_2(Y_\infty')} -  \E{g_1(Y_\infty)g_2(Y_\infty')}\right|. \nonumber
\end{align}
\end{linenomath}
Since $Y_\infty, Y_\infty'$ are independent, the second term on the right hand side is equal to 
\begin{equation} \label{1st_bound}
 |\E{g_2(Y_\infty)}| \cdot |\E{g_1(Y_n)} - \E{g_1(Y_\infty)}|.\end{equation}
As $n \to \infty$, \eqref{1st_bound} converges to zero by Theorem \ref{empir_limit}. For $n < m$, we have $\E{g_1(Y_n) g_2(Y_m)} = 
\E{g_1(Y_n) \E{g_2(Y_m)\mid \mathscr G_{m-1}}}$. Hence, the first term on the right hand side of \eqref{start_bound} is bounded from above by  
\begin{equation} \label{2nd_bound}  \| g_1 \| \cdot \E{|\E{g_2(Y_m)\mid \mathscr G_{m-1}} - \E{g_2(Y_\infty)}|}. \end{equation}
Write $\nu_m$ for the law of $Y_m$ given $\mathscr G_{m-1}$, that is, for all measurable $A \subseteq \cCd$,
\[\nu_m(A) = \frac{\int_A f(x) \dd \Pi_{m-1}(x)}{\int_{\cCd} f(x) \dd \Pi_{m-1}(x)}. \] 
By Theorem \ref{empir_limit}, we have, almost surely,  $\nu_m \to \pi^*$ weakly. Thus, 
$\E{g_2(Y_m) \mid \mathscr G_{m-1}} \to \E{g_2(Y_\infty)}$. Hence, by the dominated convergence theorem, \eqref{2nd_bound} converges to zero as $m  \to \infty$. This concludes the proof for $n, m \to \infty$ with $n < m$ and the case $n > m$ can be treated analogously.
\end{proof}

In the remainder, we write $\mathbb P^{**}_{x,x'}$ and $\mathbb E^{**}_{x,x'}$ with $x, x' \in \cC'$ for probabilities and expectations, respectively, involving a pair of independent copies of the star Markov chain $(S_0, S_0'), (S_1, S_1'), \ldots$, where $S_0 = x$ and $S_0' = x'$. 
\begin{prop}\label{prop:last_prop}
Let $k \geq 0$, $w, w' \geq 0$ and $x, x', x_1, x_1', x_2, x_2', \ldots \in \cCd$ with $x_n \to x$ and $x_n' \to x'$. Then, in the sense of weak convergence on 
$[0, \infty)^{2k+2}$, we have, as $n\to \infty$,
\begin{linenomath}
\begin{align*} 
& \mathbb P^{**}_{\varphi(w, x_n), \varphi(w', x_n')}((F(S_0), F(S_0'), F(S_1), F(S_1'), \ldots, F(S_k), F(S_k')) \in \cdot) \\
&  \to \mathbb P^{**}_{\varphi(w,x), \varphi(w', x')}((F(S_0), F(S_0'), F(S_1), F(S_1'), \ldots, F(S_k), F(S_k')) \in \cdot). \end{align*}
\end{linenomath}
\end{prop}

\begin{proof}
This follows from the independence of the two star processes involved and Proposition~\ref{help1}.
\end{proof}

Using Propositions~\ref{prop:Ynm} and~\ref{prop:last_prop}, the continuity of $\varphi$, and an argument analogous to the proof of Corollary~\ref{cor:1} (using a probability metric similar to \eqref{eq:dual-lipschitz}), Equation~\eqref{345} follows upon verifying the following: For any
$\eta \leq x_0, x_0', \ldots, x_k, x_k' \leq 1$ and $0 \leq \eps \leq 1/2$, with the function \[\Jo(y_0, y_0', \ldots, y_k, y_k') = x_k^{y_k/\lambda_{+\eps}}  \cdot \prod_{\ell=0}^{k-1} x_\ell^{y_\ell/\lambda_{+\eps}}  y_\ell \cdot (x'_k)^{y'_k/\lambda_{+\eps}}  \cdot \prod_{\ell=0}^{k-1} (x'_\ell)^{y'_\ell/\lambda_{+\eps}}  y'_\ell\] defined on $[0, C]^{2k+2}$, 
we have that $\| \nabla \Jo \|$ is bounded uniformly in $x_0, \ldots, x_k$, $x_0', \ldots, x_k'$ and $\eps$. 
This follows from that the fact that $\Jo$ factorizes, $\| \Jo\| \leq C^{2k}$, and Equation~\eqref{eq:bound-jacob} (inside the proof of Corollary~\ref{cor:1}).

Now, when evaluating the sum over $\eta n < i \neq j \leq n$ and disjoint $\Ind_k \in {\{i+1, \ldots, n\} \choose k}, \Jind_k \in {\{j+1, \ldots, n\} \choose k}$ in \eqref{345}, since the summands are non-negative, and we are looking for an upper bound, we may remove the conditions $i \neq j$ and $\Ind_k \cap \Jind_k = \varnothing$. 
But Corollary~\ref{cor:summation} shows that, uniformly in $\eps$ and $\eta$,
\begin{linenomath}
\begin{align*}
& \sum_{\eta n < i, j \leq n} \sum_{\Ind_k, \Jind_k} \mathbb E^*_{\pi^*} \left[ \left(\frac{i_k}{n}\right)^{F(S_k)/\lambda_{+\eps}}\cdot \prod_{\ell=0}^{k-1} \left(\frac{i_\ell}{i_{\ell+1}}\right)^{F(S_\ell)/\lambda_{+\eps}}\frac{F(S_{\ell})}{\lambda_{-\eps} (i_{\ell+1}-1)} \right] \nonumber \nonumber \\ &\hspace{2.5cm}
\times \mathbb E^*_{\pi^*} \left[\left(\frac{j_k}{n}\right)^{F(S_k)/\lambda_{+\eps}} \cdot \prod_{\ell=0}^{k-1} \left(\frac{j_\ell}{j_{\ell+1}}\right)^{F(S_\ell)/\lambda_{+\eps}}\frac{F(S_{\ell})}{\lambda_{-\eps} (j_{\ell+1}-1)}  \right] \\
& \leq \left(\frac{1+\eps}{1-\eps}\right)^{2k} \left( \mathbb E_{\pi^*}^* \left[ \frac{\lambda_{+\eps}}{F(S_k)+\lambda_{+\eps}} \prod_{\ell=0}^{k-1} \frac{F(S_\ell)}{F(S_\ell)+\lambda_{+\eps}}\right] \right)^2+ O\left(n^{-1/(k+2)}\right) + C'\eta^{1/{k+2}},
\end{align*}
\end{linenomath}
for some universal constant $C' >0$.
 From here, identity \eqref{aim:second} follows easily as in Subsection~\ref{subsec:upper}. 

\subsection{Convergence of the mean: bounds from below} \label{subsec:lower}
In this section, we prove that, for all $k\geq 0$,
\begin{equation}
\label{eq:aim_low_bound}
\lim_{\eta \to 0} \liminf_{n \to \infty} \frac{\Ex{\C{0}{}{}}{N_k^\eta(n)}}{n} \geq \p_k,
\end{equation}
where we recall that $N_k^\eta(n)$ is the number of vertices 
of degree $k+d$ in $\C{n}{}{}$ 
that arrived after time $\eta n$,
and $p_k$ is defined in Theorem~\ref{thm:small_degrees}.
To do this, we need further notation.
First, let $\mathbf C$ be the set of all finite $d$-dimensional simplicial complexes with integer vertices.
To add weights, let $\mathbf C^w = \mathbf C \times [0, \infty)^{\mathbb Z}$, where, for $t = (c, x) \in \mathbf C^w$, $x_i, i \in \mathbb Z$ keeps track of the weight assigned to the vertex $i$. (If no such vertex exists, simply set $x_i=0$.) We then consider $\mathcal K_n$ as a $\mathbf C^w$-valued random variable incorporating vertex weights. 
For a simplicial complex $\C{}{}{} \in \mathbf C$, 
let $\C{\m i}{}{} := \{\sigma \in \K: i \notin \sigma\}$ be the sub-complex obtained from $\K$, when we remove the faces which contain vertex $i$. (Set $\K_{\m i} := \K$ if $i \notin \K$.) 
When applied to the random dynamical process, we write
$\C{n \m i}{}{}$ for $(\mathcal K_n)_{\m i}$. Let 
\begin{equation} 
\label{eq:parti-i-removed}
\Pi_{n\m i}=\sum_{\sigma \in \C{n\m i}{}{d-1}} \delta_{\w(\sigma)}, \text{ and} \quad Z_{n\m i} = \int_{\cCd} f(x) \dd \Pi_{n\m i}(x)
\end{equation}
be the empirical measure of the types 
of active faces in $\C{n\m i}{}{}$ and the corresponding partition function, respectively.  Note that $\C{n}{}{d-1}  = \C{n\m i}{}{d-1} \cup  \St{i}{\C{n}{}{}}$, where the union is disjoint and therefore
$Z_n = Z_{n \m i} + F(\St{i}{\C{n}{}{}})$.  

To prove a suitable lower bound on the probability that vertex $i$ receives edges at certain times,  we need to control $Z_{n\m i}$ throughout the process. It is reasonable to expect $Z_{n\m i}$ to behave similarly to $Z_n$. To this end, for all $\varepsilon>0$, $n\geq i\geq 1$ and $m\geq 1$, we let 
\begin{equation}\label{eq:def_Geps}
\mathcal G_\eps^{(i)} (n) = \left\{\left| Z_{n\m i} - \lambda n \right| < \eps \lambda n \right\}
\quad\text{ and }\quad
\mathcal G_\eps (n;m) = \left\{\left| Z_{n} - \lambda m \right| < \eps \lambda m \right\}.
\end{equation}
(Note the difference between the notation $\mathcal G_\eps (n;m)$ and the notation for concentration along an interval $\mathcal G_\eps (N_1, N_2)$ defined in Subsection~\ref{subsec:upper}.) 

For  $1 \leq i \leq n, \, \Ind_k \in {\{i+1, \ldots, n\} \choose k}$ and $j=i, \ldots, n$, we let 
\begin{equation} \label{eq:p-j-def}
    p(j) \in \{0, \ldots, k\}
\text{ be such that }
i_{p(j)} \leq j \leq i_{p(j)+1}-1.
\end{equation}
(Recall that we use the conventions
$i_0=i$ and $i_{k+1} = n+1$). 

As opposed to the arguments in Subsection~\ref{subsec:upper}, the inductive proof in this section requires us to modify the value of $\eps$ in different intervals $\{i_{\ell}, \ldots, i_{\ell+1}-1\}, \ell = 0, \ldots, k$. We thus need more notation. 
First, for a fixed $\eps > 0$, and $\ell \in \{0, \ldots, k\}$ 
we set $\eps_{\ell} : = (1+ \ell)\eps$ 
(we apply this notation only to the symbol $\eps$, 
to avoid confusion with subscripts). 
Next, for $j \in \{i+1, \ldots, n\}$, recalling the events $\cD_j$ from \eqref{def:D}, and $\mathcal G_{\varepsilon}^{\sss (i)}(j)$, $\mathcal G_{\varepsilon}(i;i)$ from~\eqref{eq:def_Geps},
we set 
\begin{equation}\label{eq:def_barD}
\bar \cD_j(\eps)  
= \cD_j \cap 
\mathcal G_{\eps_{p(j)}}^{(i)}(j)
\quad\text{ and }\quad
\bar \cD_i(\eps) = \mathcal G_\eps(i;i).
\end{equation}
Similarly to before, we write $D_j(\eps) : = \mathbf{1}_{\cD_j(\eps)}$ and  $\bar D_j(\eps) : = \mathbf{1}_{\bar \cD_j(\eps)}$. 
With this notation, we have \begin{equation} \label{start:lower} \Ex{\C{0}{}{}}{N_k^\eta(n)} \geq \sum_{\eta n < i \leq n} \sum_{\cI_k \in {\{i+1, \ldots, n\} \choose k}} 
\mathbb P\bigg(\bigcap_{j=i}^n \bar \cD_j(\eps)\bigg). 
\end{equation}
We then have the following analogue of Proposition \ref{prop:upper_bounds}.

\begin{prop} \label{p2}
Let $0 < \delta,  \eps, \eta \leq 1/2$. 
There exists a constant $C' > 0$, $N = N(\delta, \varepsilon, \eta)$ and $0 \leq \varrho \leq 1$ such that, for all $n \geq N$, \begin{linenomath}
\begin{align} 
&\Ex{\C{0}{}{}}{N_k^\eta(n)} \geq  - C'\delta n \nonumber \\
& + \varrho (1-\delta) \cdot \sum_{\eta n < i \leq n} \sum_{\cI_k \in {\{i+1, \ldots, n\} \choose k}}   \Ex{\C{0}{}{}}{\mathbb E^*_{\St{i}{\C{i}{}{}}} \left[ \left(\frac{i_k}{i_{k+1}} \right)^{\frac{F(S_k)}{\lambda_{-\eps_k}}} \cdot \prod_{\ell=0}^{k-1} \left(\frac{i_\ell}{i_{\ell+1}} \right)^{\frac{F(S_\ell)}{\lambda_{-\eps_\ell}}}  \frac{F(S_{\ell})}{\lambda_{+\eps_\ell} (i_{\ell+1}-1)} \right]}, \label{eq:prop-p2}
\end{align}
\end{linenomath}
where $\varrho$ depends only on $\eps, \eta$ and, for any fixed $0 < \eta \leq 1/2$, we have $\varrho \to 1$ as $\eps \to 0$.
\end{prop}
Similar arguments leading from Proposition~\ref{prop:upper_bounds} to Corollary~\ref{cor:1} then give the following result.
\begin{cor}\label{p3}
Let $0 < \delta, \varepsilon, \eta \leq 1/2$. Then, there exists $N = N(\delta, \varepsilon, \eta)$ and a universal constant $C > 0$ not depending on any of these parameters, such that, for all $n \geq N$,
\begin{linenomath}
\begin{align*} 
&\frac{\Ex{\C{0}{}{}}{N_k^\eta(n)}}{n} \geq \varrho(1-\delta)  \left(\frac{1-\eps_k} {1+\eps_k}\right)^k \cdot \mathbb E_{\pi^*}^* \left[ \frac{\lambda_{-\eps_k}}{F(S_k)+\lambda_{-\eps_k}} \prod_{\ell=0}^{k-1} \frac{F(S_\ell)}{F(S_\ell)+\lambda_{-\eps_\ell}}\right] \\ & \hspace{6cm}- C(\eta^{1/(k+2)} + 1/n^{1/(k+2)}) - \delta,
\end{align*}
\end{linenomath}
where $\varrho$ is as in the Proposition~\ref{p2}.
In particular, Equation~\eqref{eq:aim_low_bound} holds.
\end{cor}

We now define analogues of $h_{\ell}$ and $f_{\ell}$ from Equations~\eqref{eq:h-def} and \eqref{def:ff} in Subsection~\ref{subsec:upper}  (here, however, it is necessary to indicate the dependence of these functions on $\eps$).  
For  $S \in \cC'$ and $\ell\in \{0, \ldots, k\}$, let 
\begin{equation} \label{eq:h2-def}
\mathfrak h^\eps_\ell(S) =\prod_{j=i_\ell+1}^{i_{\ell+1}-1} \left(1- \frac{F(S)}{\lambda_{- \eps_\ell}(j-1)}\right)
\end{equation}
and, for $\ell\in \{0, \ldots, k-1\}$, 
\begin{equation} \label{eq:ffrak_def}
\mathfrak f^\eps_\ell(S) =\frac{F(S)}{F(S)+ \lambda_{+\eps_\ell}(i_{\ell+1}-1)} \mathfrak h^\eps_\ell(S) \quad \mbox{while} \ \mathfrak f^\eps_k = \mathfrak h^\eps_k.\end{equation}  

We follow the arguments from the proof of the upper bound (see Subsection~\ref{subsec:upper}) 
and show analogues of Lemma \ref{lem:u} and Proposition \ref{prop:u}. 
To this end, we need 
to make use of the more general framework introduced at the beginning of this subsection: we write $\mathbb P_x(\cdot), \mathbb E_x(\cdot)$ for probabilities and expectations respectively, when the initial weighted configuration is equal to $x = (c, z)$ with $c \in \mathbf C, z \in [0, \infty)^{\mathbb Z}$. 
(Here, if $m \in \mathbb Z$ is the maximum vertex label occuring in $c$, then the vertex inserted in step $i$ of the process carries label $m+i$.) 
Then, for a real-valued function $g$ depending on the path of the process and $u(x) = \mathbb E_x[g((\K_n)_{n \geq 0})]$, we use the slightly inaccurate but standard notation $\mathbb E_{X} [g((\K_n)_{n \geq 0})]$ for $u(X)$ and a random variable $X$ which is typically defined in terms of $\K_n, n \geq 0$. 
Probabilities $\mathbb P$ and expectations $\mathbb E$ appearing in the following without subscript are with respect to the initial process with given $\K_0$. 

Proving analogues of  Lemma \ref{lem:u} and Proposition \ref{prop:u} becomes more intricate since we can no longer drop the concentration conditions
relying on the events $\mathcal G_\varepsilon(j)$ as we did in Subsection~\ref{subsec:upper}.
Nevertheless, upon ignoring the dependency structure of the evolution of the process in the star of vertex $i$ and outside, we still expect (at least morally) to bound $\PIdx{\C{0}{}{}}{\bigcap_{j=i}^n \bar \cD_j}$ from below by a term similar to 
\begin{equation} \label{b1} \Ex{\C{0}{}{}}   { \mathbb E_{\mathcal K_{i \setminus i}}\left[\prod_{j=i+1}^{n-k} \mathbf{1}_{\mathcal{G}_{\varepsilon_{p(j)}}(j-i;j+ p(j))} \right]
 \mathbb E^*_{\St{i}{\C{i}{}{}}} \left[ \prod_{j=0}^k \mathfrak f^\eps_j(S_{j}) \right]} . \end{equation}
 The two main hurdles to prove such a lower bound are the following: first, while the process outside the star of vertex $i$ follows the Markovian transition rule, there is a subtle dependence between the star and its complement as the addition of faces to the star adds faces to its complement.
More formally, on $\cD_{i_\ell}$, we have $\C{i_\ell \m i}{}{} \neq  \C{(i_\ell-1)\m i}{}{}$. The reason is 
that when a face in $\St{i}{\C{i_{\ell}-1}{}{}}$ is subdivided during step $i_\ell$, one of the faces that 
are created does not contain vertex $i$ and therefore migrates into $\C{i_\ell \m i}{}{}$ (see Figure~\ref{fig:modelbcompdim3}).
Second, in order to exploit the concentration of
the partition function $Z_j$ for $j \geq i> \eta n$, an argument is needed to replace $\mathbb P_{\K_{i \setminus i}}$ by $\mathbb P_{\K_{i}}$. 
In order to overcome these difficulties, we use the following two lemmas, whose  proofs we delay to the end of the section.

\begin{lem} \label{lem:help1} For any $\delta, \eps > 0, 0 < \eta < 1$, there exists $N = N(\delta, \eps, \eta)$ such that, for all $n \geq N, \eta n < i < n-k$, we have
\[\E{\mathbb P_{\C{i\m i}{}{}} \left(\bigcap_{j=i+1}^{n} \mathcal G_{\eps}(j-i;j)\right)} \geq 1 - \delta.\]
\end{lem}

\begin{lem} \label{lem:help2} For any $\eps_1, \eps_2, \eps_3 > 0, 0 < \eta_1  < 1$ and $C_1, C_2 > 0$, there exists $N$ depending on these six quantities, such that
the following is satisfied for all $n \geq N$: for any weighted simplicial complexes $\mathcal X, \mathcal Y \in \textbf{C}^w$ such that
\begin{itemize}
\item [\emph{(i)}] $|\mathcal X^{(d-1)} \triangle \mathcal Y^{(d-1)}| \leq C_1$, where $\mathcal X^{(d-1)} \triangle \mathcal Y^{(d-1)} = (\mathcal X^{(d-1)} \setminus \mathcal Y^{(d-1)}) \cup  (\mathcal Y^{(d-1)} \setminus \mathcal X^{(d-1)})$;
\item [\emph{(ii)}] any vertex contained in a face in $\mathcal X^{(d-1)} \cap \mathcal Y^{(d-1)}$ has the same weight in both complexes;
\item [\emph{(iii)}] each face in $ \mathcal X^{(d-1)} \triangle \mathcal Y^{(d-1)}$ has at most fitness $C_2$ in the complex it belongs to;
\item [\emph{(iv)}] $F(\mathcal X) \geq \eps_1 u$ for some $\eta_1 n \leq u \leq n$ (where we recall that $F(\mathcal X)$ is the sum of fitnesses of faces in $\mathcal X$),
\end{itemize}
we have, for any $u < m \leq n$, that 
\[\mathbb P_{\mathcal X} \left(\bigcap_{j=u+1}^m \mathcal G_{\eps_2}(j-u;j) \right) \geq \mathbb P_{\mathcal Y} \left(\bigcap_{j=u+1}^m \mathcal G_{\eps_2/2} (j-u;j) \right) - \eps_3.\]
\end{lem}

 Intuitively, Lemma~\ref{lem:help1} states that, for the process initiated by $\C{i\m i}{}{}$, the partition function remains concentrated with high probability at each of the $n-i$ steps after the arrival of vertex $i$. Lemma~\ref{lem:help2} states that any sufficiently large (linear in $n$) simplicial complexes $\mathcal{X}$ and $\mathcal{Y}$, which differ by at most a constant number of faces, have partition functions that evolve in a similar manner. This is due to the fact that the contribution of the descendants of faces in $\mathcal{X} \triangle \mathcal{Y}$ may be bounded by the sum of geometrically distributed random variables with small success parameter, and is thus negligible.

For brevity, for all $\ell\in \{0, \ldots, k\}$ and $\eps > 0$, recalling the definition of $p(j)$ in \eqref{eq:p-j-def}, we define 
\begin{equation}\label{eq:def_alpha}
G_\ell(\eps) = \bigcap_{j=i_\ell +1}^{n-(k-\ell)}  \mathcal G_{\eps_{p(j)}} (j-i_\ell; j+ p(j)-\ell) 
\quad \text{ and }\quad
\alpha_\ell(\K, \eps) =  \mathbb P_{\K} (G_\ell(\eps)), \quad \K \in \mathbf C^w.
\end{equation}
Thus, in $\alpha_\ell(\K_{i_\ell \setminus i}, \eps)$ the term $\mathcal G_{\eps_{p(j)}} (j-i_\ell; j+ p(j)-\ell)$ represents concentration of $Z_{j - i_\ell}$
(initiated with $\K_{i_\ell \setminus i}$) around $\lambda (j+ p(j) - \ell)$. 
When $p(j)$ increases, the values of $\eps_{p(j)}$ and $j + p(j) - \ell$ change to account for the additional `step' that has occurred in the underlying process without a step occurring in the process initiated with $\K_{i_\ell \setminus i}$. 
Lemma \ref{lem:help2} has the following corollary which justifies this notation, showing that the migration of the additional face into $\K_{i_\ell \setminus i}$ at the step $i_{\ell}$ is insignificant. 

\begin{cor} \label{cor:help2} 
For any $0 < \eta, \delta, \eps' < 1$, there exists $N = N(\delta, \eps', \eta)$ such that the following holds for all $n \geq N$: for all $0 < \eps < 1 / (2k+2)$, $\ell \in \{1, \ldots, k\}$ and  
$\eta n < i < i_1 < \ldots < i_k \leq n$, on the event $\mathcal G_{\eps_\ell}^{(i)}(i_\ell)$, with $\alpha_{\ell}$ as defined in \eqref{eq:def_alpha}, we have
\begin{equation} \label{eq:cor-help-2}
\alpha_\ell(\K_{i_\ell \m i}, \eps') \geq \alpha_\ell(\K_{(i_\ell-1) \m i}, \eps'/4(k+1)) - \delta.
\end{equation}
\end{cor}
\begin{proof}
For sufficiently large $n$ (depending on $\varepsilon'$ and $\eta$), we clearly have that, for all  $\K \in \mathbf C^w$
\begin{linenomath}
\begin{equation*} \label{eq:help2-1}
\alpha_\ell(\K, \eps') \geq \mathbb P_{\K}\left(\bigcap_{j=i_\ell + 1}^{n-(k-\ell)} \mathcal G_{3\eps'_\ell/4}(j-i_\ell; j)\right)
\end{equation*}
\end{linenomath}
and 
\begin{equation} \label{eq:help2-2}
 \mathbb P_{\K}\left(\bigcap_{j=i_\ell + 1}^{n-(k-\ell)} \mathcal G_{3\eps'_\ell/8}(j-i_\ell; j)\right) \geq  \alpha_\ell(\K, \eps'/4(k+1)).
\end{equation}
Note that, on $\mathcal G_{\eps_\ell}^{(i)}(i_\ell)$, we have $Z_{i_\ell \m i} \geq \lambda i_\ell / 2$. Hence,
Lemma \ref{lem:help2} applied with $\eps_1 = \lambda/2, \eps_2 = 3 \eps'_\ell / 4 , \eps_3  = \delta, u =i_\ell, \eta_1 = \eta, \mathcal Y =  \K_{(i_{\ell}-1)\m i}, 
\mathcal X =  \K_{i_{\ell}\m i}, C_1 = d+1, C_2 = f_{\max}$ shows that, on the event $\mathcal G_{\eps_\ell}^{(i)}(i_\ell)$,
\begin{equation} \label{eq:help2-3}
\mathbb P_{\K_{i_\ell \m i} }\left(\bigcap_{j=i_\ell + 1}^{n-(k-\ell)} \mathcal G_{3\eps'_\ell/4}(j-i_\ell;j)\right) \geq  \mathbb P_{\K_{(i_\ell-1) \m i}}\left(\bigcap_{j=i_\ell + 1}^{n-(k-\ell)} \mathcal G_{3\eps_\ell'/8}(j-i_\ell;j)\right) - \delta
\end{equation}
for $n$ sufficiently large (depending on $\delta, \eps', \eta$).
Equations~\eqref{eq:help2-2} and \eqref{eq:help2-3} together imply Equation~\eqref{eq:cor-help-2}.
\end{proof}

Once we have Corollary \ref{cor:help2}, 
the arguments to prove the lower bound are similar to the upper bound, however, the details are more technical. The following lemma is the analogue of Lemma \ref{lem:u}.
\begin{lem} \label{lem:down} For any $0, \delta, \eta < 1$ and $0 < \eps < 1/(2k+2)$ there exists $N = N(\delta, \eps, \eta)$, such that, for all $n \geq N$ and 
$\eta n < i < i_1 < \ldots < i_k \leq n$,  with $\mathfrak h^\eps_j$ as defined in \eqref{eq:h2-def}, we have
\begin{linenomath}
\begin{equation}
 \Prob{\bigcap_{j=i_k+1}^{n} \bar \cD_j(\eps) \bigg |  \mathscr G_{i_k}}\bar D_{i_k}(\eps) \geq (\alpha_k(\C{(i_k-1)\m i}{}{}, \eps/(4(k+1))) - \delta)
\mathfrak h^\eps_k(\St{i}{\C{i_k}{}{}}) \bar D_{i_k}(\eps) \label{s1}
\end{equation}
\end{linenomath}
and, for all $\ell\in \{1, \ldots, k-1\}$,
\begin{linenomath}
\begin{align*}
 &\E{\prod_{j=i_\ell+1}^{i_{\ell+1}-1} \bar D_j(\eps)~\alpha_{\ell+1}(\C{(i_{\ell+1}-1)\m i}{}{}, \eps)
 \bigg |   \mathscr G_{i_\ell}} \bar D_{i_\ell}(\eps)
\\ & \hspace{3cm} \geq (\alpha_\ell (\C{(i_{\ell}-1)\m i}{}{}, (k+1))-\delta) 
\mathfrak  h^\eps_\ell(\St{i}{\C{i_\ell}{}{}}) \bar D_{i_\ell}(\eps), \text{ while, } \\
 &\E{ \prod_{j=i+1}^{i_{1}-1} \bar D_j(\eps)~\alpha_{1}(\C{(i_{1}-1)\m i}{}{}, \eps)
 \bigg |   \mathscr G_{i}}  \bar D_{i}(\eps)
\geq \alpha_0 (\C{i \m i}{}{}, \eps)
\mathfrak  h^\eps_0(\St{i}{\C{i}{}{}}) \bar D_{i}(\eps).  \end{align*}
\end{linenomath}
\end{lem}

\begin{proof} 
We write $\bar D_j$ for $\bar D_j(\eps)$ throughout the proof. We have
\begin{linenomath}
\begin{align}
\E{ \prod_{j=i_k+1}^n \bar D_j \bigg |   \mathscr G_{i_k}}       
&= \E{  \E{ \bar D_n \bigg | \mathscr G_{n-1}} \prod_{j=i_k+1}^{n-1} \bar D_j  \bigg | \mathscr G_{i_k}}\notag \\
&= \E{   \left(1-\frac{F(\St{i}{\C{n-1}{}{}})}{Z_{n-1}}\right) \Pxx{\C{(n-1)\m i}{}{}}{\mathcal G_{\eps_k}(1; n)} 
\prod_{j=i_k+1}^{n-1} \bar D_j \bigg |  \mathscr G_{i_k}},\label{eq:cond_exp}
\end{align}
\end{linenomath}
because, by definition (see~\eqref{eq:def_Geps}), $\mathcal G_{\eps_k}(1; n) = \{|Z_1-\lambda n|<\varepsilon_k\lambda n\}$.
First note that, on the event $\bigcap_{j=i_k+1}^{n-1} \bar{\mathcal D}_j$, we have,
for any $j=i_k+1, \ldots, n-1$,
$F(\St{i}{\C{j}{}{}}) = F(\St{i}{\C{i_k}{}{}})$.
On the event $\bar {\mathcal D}_j$ 
we have 
\begin{equation} \label{eq:subs0}
1-\frac{F(\St{i}{\C{n-1}{}{}})}{Z_{j}} \geq 1-\frac{F(\St{i}{\C{i_k}{}{}})}{\lambda_{-\eps_k}j}.
\end{equation}
Furthermore, by the tower property, we may substitute 
\[\Ex{\C{(n-1)\m i}{}{}}{\Pxx{\C{(n-1)\m i}{}{}}{\mathcal G_{\eps_k}(1; n)}\bar D_{n-1} \bigg | \mathscr G_{n-2}} \quad
\text{ for } \quad \Pxx{\C{(n-1)\m i}{}{}}{\mathcal G_{\eps_k}(1; n)} \bar D_{n-1}\]
inside the conditional expectation, and together with~\eqref{eq:cond_exp} and \eqref{eq:subs0}, this gives
\begin{equation} \label{eq:subs1}
\E{ \prod_{j=i_k+1}^n \bar D_j \bigg |   \mathscr G_{i_k}}  
\geq \left(1-\frac{F(\St{i}{\C{i_k}{}{}})}{\lambda_{-\eps_k}(n-1)}\right)  \E{ \Ex{\C{(n-1)\m i}{}{}}{\Pxx{\C{(n-1)\m i}{}{}}{\mathcal G_{\eps_k}(1; n)} 
\bar D_{n-1}   \bigg | \mathscr G_{n-2}}  \prod_{j=i_k+1}^{n-2} \bar D_j \bigg | \mathscr G_{i_k}}.
\end{equation}
We also have
\begin{equation}\label{eq:subs2} \Ex{\C{(n-1)\m i}{}{}}{\Pxx{\C{(n-1)\m i}{}{}}{\mathcal G_{\eps_k}(1; n)} 
\bar D_{n-1}   \bigg | \mathscr G_{n-2}} =\left(1-\frac{F(\St{i}{\C{n-2}{}{}})}{Z_{n-2}}\right)  \Pxx{\C{(n-2)\m i}{}{}}{\mathcal G_{\eps_k}(1; n-1) \cap \mathcal G_{\eps_k}(2; n) }.
\end{equation}
Thus, using Equations~\eqref{eq:subs1} and \eqref{eq:subs2} in the first inequality,
and~\eqref{eq:subs0} in the second,
\begin{linenomath}
\begin{align*}
&\E{ \prod_{j=i_k+1}^n \bar D_j \bigg |   \mathscr G_{i_k}}\\
& 
\geq 
 \left(1-\frac{F(\St{i}{\C{i_k}{}{}})}{\lambda_{-\eps_k}(n-1)}\right)  \E{ \left(1-\frac{F(\St{i}{\C{n-2}{}{}})}{Z_{n-2}}\right)  \Pxx{\C{(n-2)\m i}{}{}}{\mathcal G_{\eps_k}(1; n-1) \cap \mathcal G_{\eps_k}(2; n) }   \prod_{j=i_k+1}^{n-2} \bar D_j \bigg | \mathscr G_{i_k}} \\
& 
\geq\left(1-\frac{F(\St{i}{\C{i_k}{}{}})}{\lambda_{-\eps_k}(n-1)}\right)  \left(1-\frac{F(\St{i}{\C{i_k}{}{}})}{\lambda_{-\eps_k}(n-2)}\right)  
\E{\Pxx{\C{(n-2)\m i}{}{}}{\mathcal G_{\eps_k}(1; n-1) \cap \mathcal G_{\eps_k}(2; n) }   \prod_{j=i_k+1}^{n-2} \bar D_j \bigg | \mathscr G_{i_k}}. 
\end{align*}
\end{linenomath}

Iterating this process leads to 
$\Prob{\bigcap_{j=i_k+1}^{n} \bar \cD_j(\eps) \bigg |  \mathscr G_{i_k}}\bar D_{i_k}(\eps)\geq 
\alpha_k(\C{i_k\m i}{}{}, \eps) \mathfrak h^\eps_k(\St{i}{\C{i_k}{}{}}) \bar D_{i_{k}}. $
Equation~\eqref{eq:cor-help-2} from Corollary \ref{cor:help2} concludes the proof of Equation~\eqref{s1}
as $\bar {\mathcal D}_{i_k} \subseteq \mathcal G_{\eps_k}^{(i)}(i_k)$.

We use the same ideas to prove the general case, for $\ell\in \{0, \ldots, k-1\}$. 
Here, we want to provide a lower bound to 
$\E{\alpha_{\ell+1}(\C{(i_{\ell+1}-1)\m i}{}{},\eps)~\prod_{j=i_\ell+1}^{i_{\ell+1}-1} \bar D_j  \bigg |
\mathscr G_{i_\ell}}$. 
First, for any $j=i_\ell+1, \ldots, i_{\ell+1}-1$, we have $F(\St{i}{\C{j}{}{}}) = F(\St{i}{\C{i_\ell}{}{}}$.
Thus,  
on the event $\bar{\mathcal{D}}_j$, we have 
\begin{equation} \label{eq:subs0_gen}
1-\frac{F(\St{i}{\C{j}{}{}})}{Z_{j}} \geq 1-\frac{F(\St{i}{\C{i_\ell}{}{}})}{\lambda_{-\eps_\ell}j}.
\end{equation}
Second, using the tower property, we  substitute 
\begin{equation} \label{eq:subs1_gen}
\E{\alpha_{\ell+1}(\C{(i_{\ell+1}-1)\m i}{}{}, \eps)~\bar D_{i_{\ell+1}-1} \bigg |
\mathscr G_{i_{\ell+1}-2} } 
\,\text{ for }\,
\alpha_{\ell+1}(\C{(i_{\ell+1}-1)\m i}{}{},\eps)\bar D_{i_{\ell+1}-1}
\end{equation}
inside the conditional expectation.
Third, 
\begin{eqnarray}
\label{eq:subs3_gen}
\nonumber\lefteqn{\E{\alpha_{\ell+1}(\C{(i_{\ell+1}-1)\m i}{}{}, \eps)~\bar D_{i_{\ell+1}-1} \bigg |
\mathscr G_{i_{\ell+1}-2} } = \left( 1-\frac{F(\St{i}{\C{i_{\ell+1}-2}{}{}})}{Z_{i_{\ell+1}-2}}\right)\times} \\
 & & \Pxx{\C{(i_{\ell+1}-2)\m i}{}{}}{
   \mathcal G_{\eps_{\ell}}(1;i_{\ell+1}-1) \cap \bigcap_{j = i_{\ell+1}+1}^{n-(k-\ell-1)}  \mathcal{G}_{\varepsilon_{p(j)}}(j-i_{\ell+1}+1;j+p(j)-\ell-1)}. 
\end{eqnarray}
So we write:
\begin{linenomath}
\begin{align*}
& \E{\alpha_{\ell+1}(\C{(i_{\ell+1}-1)\m i}{}{},\eps)~\prod_{j=i_\ell+1}^{i_{\ell+1}-1} \bar D_j  \bigg |
\mathscr G_{i_\ell}} \\
& \stackrel{\eqref{eq:subs1_gen}}{=} \E{\E{\alpha_{\ell+1}(\C{(i_{\ell+1}-1)\m i}{}{}, \eps)~\bar D_{i_{\ell+1}-1} \bigg |
\mathscr G_{i_{\ell+1}-2} }
\prod_{j=i_\ell+1}^{i_{\ell+1}-2} \bar D_j \bigg |  \mathscr G_{i_\ell}} \\
& \stackrel{\eqref{eq:subs3_gen}}{=} \mathbb E \Bigg [ \left( 1-\frac{F(\St{i}{\C{i_{\ell+1}-2}{}{}})}{Z_{i_{\ell+1}-2}}\right) \prod_{j=i_\ell+1}^{i_{\ell+1}-2} \bar D_j \times \\
& \hspace{1cm}
\Pxx{\C{(i_{\ell+1}-2)\m i}{}{}}{
   \mathcal G_{\eps_{\ell}}(1;i_{\ell+1}-1) \cap \bigcap_{j = i_{\ell+1}+1}^{n-(k-\ell-1)}  \mathcal{G}_{\varepsilon_{p(j)}}(j-i_{\ell+1}+1;j+p(j)-\ell-1)}
 \bigg |  \mathscr G_{i_\ell}\Bigg].
\end{align*}
\end{linenomath}
Now, the lower bound of~\eqref{eq:subs0_gen} yields:
\begin{linenomath}
\begin{align*}
&\E{\alpha_{\ell+1}(\C{(i_{\ell+1}-1)\m i}{}{},\eps)~\prod_{j=i_\ell+1}^{i_{\ell+1}-1} \bar D_j  \bigg |
\mathscr G_{i_\ell}}\\
&\geq \left(1-\frac{F(\St{i}{\C{i_\ell}{}{}})}{\lambda_{-\eps_\ell}(i_{\ell+1}-2)}\right)
\E{\Pxx{\C{(i_{\ell+1}-2)\m i}{}{}}{
  \bigcap_{j = i_{\ell+1}-1}^{n-(k-\ell)}  \mathcal{G}_{\varepsilon_{p(j)}}(j-i_{\ell+1}+2; j+p(j)-\ell)} \prod_{j=i_\ell+1}^{i_{\ell+1}-2} \bar D_j
\bigg |   \mathscr G_{i_\ell}}.
\end{align*}
\end{linenomath}
By the tower property again, we substitute
\begin{eqnarray} \label{eq:subs4_gen}
\nonumber \lefteqn{\E{ \Pxx{\C{(i_{\ell+1}-2)\m i}{}{}}{
  \bigcap_{j = i_{\ell+1}-1}^{n-(k-\ell)}  \mathcal{G}_{\varepsilon_{p(j)}}(j-i_{\ell+1}+2; j+p(j)-\ell)} \bar D_{i_{\ell+1}-2} | \mathscr G_{i_{\ell+1}-3}}}  \\ 
&& \quad \text{ for } \quad \Pxx{\C{(i_{\ell+1}-2)\m i}{}{}}{
  \bigcap_{j = i_{\ell+1}-1}^{n-(k-\ell)}  \mathcal{G}_{\varepsilon_{p(j)}}(j-i_{\ell+1}+2; j+p(j)-\ell)} \bar D_{i_{\ell+1}-2}.
\end{eqnarray}
Also, 
\begin{eqnarray}  \label{eq:subs5_gen}
\nonumber \lefteqn{\E{ \Pxx{\C{(i_{\ell+1}-2)\m i}{}{}}{
  \bigcap_{j = i_{\ell+1}-1}^{n-(k-\ell)}  \mathcal{G}_{\varepsilon_{p(j)}}(j-i_{\ell+1}+2; j+p(j)-\ell)} \bar D_{i_{\ell+1}-2} | \mathscr G_{i_{\ell+1}-3}} =}  \\ 
  & &
  \left(1-\frac{F(\St{i}{\C{i_{\ell+1}-3}{}{}})}{Z_{i_{\ell+1}-3}} \right)
    \Pxx{\C{(i_{\ell+1}-3)\m i}{}{}}{
  \bigcap_{j = i_{\ell+1}-2}^{n-(k-\ell)}  \mathcal{G}_{\varepsilon_{p(j)}}(j-i_{\ell+1}+3; j+p(j)-\ell)}.
\end{eqnarray}
Bounding the first factor as  in~\eqref{eq:subs0_gen}, and combining Equations~\eqref{eq:subs4_gen} and~\eqref{eq:subs5_gen} give 
\begin{linenomath}
\begin{align*}
&\E{\alpha_{\ell+1}(\C{(i_{\ell+1}-1)\m i}{}{},\eps)~\prod_{j=i_\ell+1}^{i_{\ell+1}-1} \bar D_j  \bigg |
\mathscr G_{i_\ell}}\\
& \geq \left(1-\frac{F(\St{i}{\C{i_\ell}{}{}})}{\lambda_{-\eps_\ell}(i_{\ell+1}-2)}\right) \left(1-\frac{F(\St{i}{\C{i_\ell}{}{}})}{\lambda_{-\eps_\ell}(i_{\ell+1}-3)}\right) 
\times \\ &
\ \ \ \ 
\E{\Pxx{\C{(i_{\ell+1}-3)\m i}{}{}}{
  \bigcap_{j = i_{\ell+1}-2}^{n-(k-\ell)}  \mathcal{G}_{\varepsilon_{p(j)}}(j-i_{\ell+1}+3; j+p(j)-\ell)} \prod_{j=i_\ell+1}^{i_{\ell+1}-3} \bar D_j
\bigg |   \mathscr G_{i_\ell}}.
\end{align*}
\end{linenomath}
Iterating the argument shows that the right hand side multiplied by $\bar D_{i_\ell}$ is bounded from below by \sloppy $\alpha_\ell(\C{i_\ell \m i}{}{}, \eps) \mathfrak h^\eps_\ell(\St{i}{\C{i_\ell}{}{}}) \bar D_{i_\ell}$.
We conclude the proof by appying Equation~\eqref{eq:cor-help-2} from Corollary~\ref{cor:help2}.
\end{proof}

\begin{lem} \label{prop:d}

For  any $\delta > 0, 0 <  \eta < 1$ and $0 < \eps < 1 / (2k+2)$, there exists $N = N(\delta, \eps, \eta)$ such that, for all $n \geq N$, $\ell\in \{1, \ldots, k\}$ and $\eta n < i < i_1 < \ldots < i_k \leq n$, with $\mathfrak  f^\eps_j$ as defined in \eqref{eq:ffrak_def} we have
\begin{linenomath}
\begin{align} & \E{\alpha_{\ell+1}(\C{(i_{\ell+1}-1)\m i}{}{}, \eps) 
~\mathbb E^*_{\St{i}{\C{i_{\ell+1}}{}{}}} \left[ \prod_{j=\ell+1}^k \mathfrak  f^\eps_j(S_{j-\ell-1}) \right] \prod_{j=i_{\ell}+1}^{\min(i_{\ell+1},n)} \bar D_j(\eps) \bigg | \mathscr G_{i_{\ell}} } \bar D_{i_{\ell}}(\eps)  \nonumber \\ 
& \geq (\alpha_{\ell}(\C{(i_{\ell}-1)\m i}{}{}, \eps/(4(k+1))) - \delta)
 \mathbb E^*_{\St{i}{\C{i_{\ell}}{}{}}} \left[ \prod_{j=\ell}^k \mathfrak  f^\eps_j(S_{j-\ell}) \right] \bar D_{i_\ell}(\eps) \label{xk}, 
 \end{align}
 \end{linenomath}
where we use the convention $\alpha_{k+1}(\cdot) = 1$,  while
\begin{linenomath}
 \begin{align*} & \E{\alpha_{1}(\C{(i_{1}-1)\m i}{}{}, \eps)
 ~\mathbb E^*_{\St{i}{\C{i_{1}}{}{}}} \left[ \prod_{j=1}^k \mathfrak  f^\eps_j(S_{j-\ell-1}) \right]  \prod_{j=i+1}^{i_{1}} \bar D_j(\eps) \bigg |  \mathscr G_{i} }\bar D_{i}(\eps) \\
& \geq \alpha_{0}(\C{i\m i}{}{}, \eps)
 \mathbb E^*_{\St{i}{\C{i}{}{}}} \left[ \prod_{j=0}^k \mathfrak  f^\eps_j(S_{j}) \right] \bar D_{i}(\eps). \end{align*}
\end{linenomath}
\end{lem}

\begin{proof}
Equation~\eqref{xk} coincides with the statement of Lemma~\ref{lem:down} for $\ell = k$. Let $0 \leq \ell \leq k-1$. 
Note that, for all $1 \leq i \leq n$, we have $|Z_{n \m i} - Z_{(n-1) \m i}| \leq (d+1) f_{\max} $. Thus, for all $n$ sufficiently large (depending on $\eps$ and $\eta$), we have
\begin{equation} \label{setinclusion}
\mathcal D_{i_{\ell+1}} \cap \mathcal{G}_{\eps_{\ell}}^{(i)}(i_{\ell+1}-1) \subseteq \mathcal{G}_{\eps_{\ell+1}}^{(i)}(i_{\ell+1}).
\end{equation}
Using this observation in the second step, we deduce
\begin{linenomath}
\begin{align*}
& \E{
 \alpha_{\ell+1}(\C{(i_{\ell+1}-1)\m i}{}{}, \eps)~
\mathbb E^*_{\St{i}{\C{i_{\ell+1}}{}{}}} \left[ \prod_{j=\ell+1}^k \mathfrak  f^\eps_j(S_{j-\ell-1}) \right] \prod_{j=i_{\ell}+1}^{i_{\ell+1}} \bar D_j \bigg |   \mathscr G_{i_{\ell}} }\bar D_{i_{\ell}} \\
& =  \E{ \E{ \bar D_{i_{\ell+1}}~\mathbb E^*_{\St{i}{\C{i_{\ell+1}}{}{}}} \left[ \prod_{j=\ell+1}^k \mathfrak  f^\eps_j(S_{j-\ell-1}) \right]  \bigg | \mathscr G_{i_{\ell+1}-1}}~\alpha_{\ell+1}(\C{(i_{\ell+1}-1)\m i}{}{}, \eps)~\prod_{j=i_{\ell}+1}^{i_{\ell+1}-1} \bar D_j \bigg |  \mathscr G_{i_{\ell}} } \bar D_{i_{\ell}} \\
& \stackrel{\eqref{setinclusion}}{\geq} \E{ \E{  D_{i_{\ell+1}} \mathbb E^*_{\St{i}{\C{i_{\ell+1}}{}{}}} \left[ \prod_{j=\ell+1}^k \mathfrak  f^\eps_j(S_{j-\ell-1}) \right]  \bigg | \mathscr G_{i_{\ell+1}-1}}~ \alpha_{\ell+1}(\C{(i_{\ell+1}-1)\m i}{}{}, \eps)~\prod_{j=i_{\ell}+1}^{i_{\ell+1}-1} \bar D_j \bigg |  \mathscr G_{i_{\ell}} } \bar D_{i_{\ell}}. 
\end{align*}
\end{linenomath}
Recall that (analogous to in the Proof of Proposition~\ref{prop:u}), conditionally on $\mathscr G_{i_{\ell+1}-1}$, on the event $\mathcal D_{i_\ell +1}$, the random variable $\St{i}{\C{i_{\ell+1}}{}{}}$ 
is distributed as $S_1$ for the star Markov process starting at $\St{i}{\C{i_{\ell+1}-1}{}{}}$. 
This yields:
\begin{linenomath}
\begin{align*}
\E{  D_{i_{\ell+1}} \mathbb E^*_{\St{i}{\C{i_{\ell+1}}{}{}}} \left[ \prod_{j=\ell+1}^k \mathfrak  f^\eps_j(S_{j-\ell-1}) \right]  \bigg | \mathscr G_{i_{\ell+1}-1}} &= \Prob {\cD_{i_{\ell+1}} | \mathscr G_{i_{\ell+1}-1}} \cdot
\mathbb E^*_{\St{i}{\C{i_{\ell+1}-1}{}{}}} \left[ \prod_{j=\ell+1}^k \mathfrak  f^\eps_j(S_{j-\ell}) \right]\\
& =\frac{F(\St{i}{\C{i_{\ell+1}-1}{}{}})}{ Z_{i_{\ell+1}-1}} 
\cdot \mathbb E^*_{\St{i}{\C{i_{\ell+1}-1}{}{}}} \left[ \prod_{j=\ell+1}^k \mathfrak  f^\eps_j(S_{j-\ell}) \right].
\end{align*}
\end{linenomath}
\meaneree{
We deduce that 
\begin{linenomath}
\begin{align*}
& \E{
 \alpha_{\ell+1}(\C{(i_{\ell+1}-1)\m i}{}{}, \eps)~
\mathbb E^*_{\St{i}{\C{i_{\ell+1}}{}{}}} \left[ \prod_{j=\ell+1}^k \mathfrak  f^\eps_j(S_{j-\ell-1}) \right] \prod_{j=i_{\ell}+1}^{i_{\ell+1}} \bar D_j \bigg |   \mathscr G_{i_{\ell}} }\bar D_{i_{\ell}} \\
& \geq  \E{ \frac{F(\St{i}{\C{i_{\ell}}{}{}})}{ Z_{i_{\ell+1}-1}} \mathbb E^*_{\St{i}{\C{i_{\ell}}{}{}}} \left[ \prod_{j=\ell+1}^k \mathfrak  f^\eps_j(S_{j-\ell}) \right] \alpha_{\ell+1}(\C{(i_{\ell+1}-1)\m i}{}{}, \eps)~\prod_{j=i_{\ell}+1}^{i_{\ell+1}-1} \bar D_j  \bigg |  \mathscr G_{i_{\ell}} } \bar D_{i_{\ell}}.
\end{align*}
\end{linenomath}
But on the event associated with 
$\bar{D}_{i_{\ell+1}}$ we have 
\[\frac{F(\St{i}{\C{i_{\ell}}{}{}})}{ Z_{i_{\ell+1}-1}} \geq \frac{F(\St{i}{\C{i_{\ell}}{}{}})}{ F(\St{i}{\C{i_{\ell}}{}{}}) +  \lambda_{+ \eps_\ell} (i_{\ell+1}-1)}.\]
So the previous inequality continues as follows:
\begin{linenomath}
\begin{align*}
& \frac{F(\St{i}{\C{i_{\ell}}{}{}})}{ F(\St{i}{\C{i_{\ell}}{}{}}) +  \lambda_{+ \eps_\ell} (i_{\ell+1}-1)} \times \\ 
& \ \ \ \
\mathbb E^*_{\St{i}{\C{i_{\ell}}{}{}}} \left[ \prod_{j=\ell+1}^k \mathfrak  f^\eps_j(S_{j-\ell}) \right] \cdot
\E{\alpha_{\ell+1}(\C{(i_{\ell+1}-1)\m i}{}{}, \eps)~ \prod_{j=i_{\ell}+1}^{i_{\ell+1}-1} \bar D_j 
\bigg |  \mathscr G_{i_{\ell}} } \bar D_{i_{\ell}}.
\end{align*}
\end{linenomath}}
We bound the last term from below using Lemma~\ref{lem:down}: 
\[
\E{\alpha_{\ell+1}(\C{(i_{\ell+1}-1)\m i}{}{}, \eps)~ \prod_{j=i_{\ell}+1}^{i_{\ell+1}-1} \bar D_j 
\bigg |  \mathscr G_{i_{\ell}} } \bar D_{i_{\ell}}
\geq (\alpha_\ell (\C{(i_{\ell}-1)\m i}{}{}, \eps/(4(k+1)))-\delta )\mathfrak  h^\eps_\ell(\St{i}{\C{i_\ell}{}{}})\bar D_{i_{\ell}}.
\]

By \eqref{eq:ffrak_def}, we have
\[
\frac{F(\St{i}{\C{i_{\ell}}{}{}})}{ F(\St{i}{\C{i_{\ell}}{}{}}) +  \lambda_{+ \eps_\ell}(i_{\ell+1}+1)} 
\mathfrak  h^\eps_\ell(\St{i}{\C{i_\ell}{}{}})  
\mathbb E^*_{\St{i}{\C{i_{\ell}}{}{}}} \left[ \prod_{j=\ell+1}^k \mathfrak  f^\eps_j(S_{j-\ell}) \right] =\mathbb E^*_{\St{i}{\C{i_{\ell}}{}{}}} \left[ \prod_{j=\ell}^k \mathfrak  f^\eps_j(S_{j-\ell}) \right],
\]
so the claim follows. 
\end{proof}

The lemma allows us to bound $\mathbb P\big(\bigcap_{j=i+1}^n \bar{\mathcal D}_j\big)$ from below by a term similar to \eqref{b1} using a backward induction argument which is of the same nature as the proof of Proposition \ref{prop:u}. This result needs to be prepared with the following definition. For $0 < \eps < 1/(2k+2), 0 < \eta < 1$ and $C > 0$, set
\begin{equation} \label{gammadef}
\gamma(\eps, \eta, C) = \gamma_k(\eps, \eta, C)^{k(k+1)/2}, \quad  \gamma_\ell(\eps, \eta, C) = \left(1-\eps_\ell \right) \eta^{2 C \eps_\ell / \lambda}, \quad \ell = 1,\ldots,k.
\end{equation}
Note that these terms decrease as $\eps$ or $C$ increase.

\begin{lem} \label{lem:help3}
For $0 < \eps < 1/(2k+2), 0 < \eta < 1$ and $C > 0$ there exists $N = N(\eps, \eta, C)$ such that, for all $n \geq N$, $\eta n < i < i_1 < \ldots < i_k \leq n$ and $0 < \eps' \leq \eps$
\[\mathfrak f_\ell^\eps(S) \geq \gamma_\ell(\eps, \eta, C) \mathfrak f_\ell^{\eps'}(S) \quad \text{for all } S \in \cC' \text{ with } F(S) \leq C. \]
\end{lem}

\begin{proof}
Recalling that $\lambda_{+\eps_\ell} = \lambda(1+\eps_\ell)$
we deduce that 
\[\frac{F(S)}{F(S) + \lambda_{+\eps_\ell}(i_{\ell+1}-1)} 
> \frac{F(S)}{(1 + \eps_{\ell})(F(S) + \lambda(i_{\ell+1}-1))} > \left(1-\eps_\ell \right) \frac{F(S)}{F(S) + \lambda(i_{\ell+1}-1)}.\]
(This statement requires no bounds on $F(S)$ or $i_\ell$.) Hence, it is sufficient to prove that
$\mathfrak h_\ell^\eps(S) \geq \eta^{2 C \eps_\ell / \lambda} \mathfrak h_\ell^{\eps'}(S)$
for sufficiently large $n$. By Lemma \ref{simple_stir}, we have
\[\mathfrak h_\ell^\eps(S) = \left( \frac{i_\ell}{i_{\ell+1}}\right)^{F(S)/\lambda_{-\eps_\ell}} \left(1 + O \left (\frac 1 n \right) \right),\]
where the $O$-term can be chosen uniformly in $\eps, i _\ell, i_{\ell+1}$ and $S$ for given $\eta$ and $C$. Note that $\mathfrak h_\ell^{\eps}(S)$ increases as $\eps$ decreases. Therefore, it is enough to prove that for each $\ell \in \{0, \ldots, k+1\}$
\[\left(\frac{i_{\ell}}{i_{\ell + 1}}\right)^{F(S)/\lambda_{-\eps_\ell}} > \eta^{2 C \eps_\ell / \lambda} \left(\frac{i_{\ell}}{i_{\ell + 1}}\right)^{F(S)/\lambda}\]
for all $S$ with $F(S) \leq C$. This follows easily from the bound on $F$,  the fact that $\eps < 1 / (2k+2)$ (so that for each $\ell$ $1/(1-\eps_{\ell}) \leq 2$) and each ratio satisfies $\eta \leq \frac{i_{\ell}}{i_{\ell + 1}} < 1$. 
\end{proof}

\begin{prop} \label{prop:be_f}
For $\delta > 0,  0 <  \eta < 1$ and $0 < \eps < 1 / (2k+2)$, there exists $N = N(\delta, \eps, \eta) > 0$ such that, for all $n \geq N$ and $\eta n < i \leq i_1 < \ldots < i_k \leq n$, with  $\gamma_k = \gamma_k(\eps, \eta, (d+1)(k+1)f_{\max})$ and $\gamma = \gamma(\eps, \eta, (d+1)(k+1) f_{\max})$,  we have, 
\begin{linenomath}
\begin{align} \label{eq:the-prop}
 \Prob{\bigcap_{j=i+1}^n\bar {\mathcal D}_j(\eps)} \geq 
& \gamma  \E{ \alpha_0\left(\C{i\m i}{}{}, \eps/(4(k+1))^{k+1}\right) \mathbb E^*_{\St{i}{\C{i}{}{}}}  \left[\prod_{j=0}^k \mathfrak f_j^\eps(S_j)\right]\bar D_i(\eps/(4(k+1))^{k+1})} 
\nonumber \\
& - \delta \sum_{\ell=1}^k \E{ 
\prod_{j=i+1}^{i_\ell}\bar D_j(\eps) \mathbb E^*_{\St{i}{\C{i_{\ell}}{}{}}} \left[\prod_{j=0}^{\ell} \mathfrak f^{\eps/(4(k+1))^{k}}_{k + j - \ell}(S_{j})\right] 
\bar D_i(\eps)}.
\end{align}
\end{linenomath}
\end{prop}

\begin{proof}
By Lemma \ref{lem:down}, we have
\begin{linenomath}
\begin{align*}
&\mathbb P\Bigg(\bigcap_{j=i+1}^n \bar \cD_j(\eps)\Bigg)
=\mathbb E \left[ \Prob{\bigcap_{j=i_k+1}^{n} \bar \cD_j(\eps) \bigg |  \mathscr G_{i_k}} 
 \prod_{j=i+1}^{i_k} \bar D_j(\eps)\right] \\
 & \stackrel{\eqref{s1}}{\geq} \mathbb E \left[ \alpha_k(\K_{(i_k-1) \m i}, \eps/(4(k+1))) \mathbb E^*_{\St{i}{\C{i_{k}}{}{}}}  [\mathfrak f_k^\eps(S_0)] \prod_{j=i+1}^{i_k} \bar D_j(\eps) \right]
 - \delta \mathbb E \left[  \mathbb E^*_{\St{i}{\C{i_{k}}{}{}}}   [\mathfrak f_k^\eps(S_0)] \prod_{j=i+1}^{i_k} \bar D_j(\eps) \right].
 \end{align*}
 \end{linenomath}
In order to apply Lemma \ref{prop:d} again in the first term, we may replace $\bar D_j(\eps)$ by $\bar D_j(\eps/(4(k+1)))$. Moreover, by Lemma \ref{lem:help3} and as $F(S_\ell) \leq (d+1)(k+1) f_{\max}$ for $\ell \in \{0, \ldots, k\}$, we may replace $\mathfrak f_k^\eps(S_0)$ by $\gamma_k \mathfrak f_k^{\eps/(4(k+1))}(S_0)$ for sufficiently large $n$. Hence, applying Lemma \ref{prop:d} again after this step, we deduce that the first term in the last display is bounded from below by
\begin{linenomath}
\begin{align*}
& \gamma_k \mathbb E \left[ \alpha_{k-1}(\K_{(i_{k-1}-1) \m i}, \eps/16) \mathbb E^*_{\St{i}{\C{i_{k-1}}{}{}}} \left[\mathfrak f_{k-1}^{\eps/(4(k+1))}(S_0) \mathfrak f_{k}^{\eps/(4(k+1))}(S_1) \right] \prod_{j=i+1}^{i_{k-1}} \bar D_j(\eps/(4(k+1))) \right] \\
& - \delta  \gamma_k  \mathbb E \left[ \mathbb E^*_{\St{i}{\C{i_{k-1}}{}{}}}  \left [\mathfrak f_{k-1}^{\eps/(4(k+1))}(S_0) \mathfrak f_{k}^{\eps/(4(k+1))}(S_1) \right] \prod_{j=i+1}^{i_{k-1}} \bar D_j(\eps/(4(k+1))) \right].
\end{align*}
\end{linenomath}

We now iterate these steps until the main term contains $\alpha_0$. In particular, with the leading term, at the $(\ell + 1)$th step we get an expression of the form
\begin{linenomath}
\begin{align*}
& \E{\alpha_{k-\ell}(\K_{(i_{k-\ell}-1) \m i}, \eps/(4(k+1))^{\ell+1}) \mathbb E^*_{\St{i}{\C{i_{k-\ell}}{}{}}} \left[\prod_{j= 0}^{\ell}\mathfrak f_{k+ j - \ell}^{\eps/(4(k+1))^{\ell}}(S_{j})\right] \prod_{j=i+1}^{i_{k-\ell}} \bar D_j\left(\eps/(4(k+1))^{\ell}\right) } \\
& \geq \left(\prod_{j = 0}^{\ell} \gamma_{k-j}\right) \mathbb E\left[\alpha_{k-(\ell+1)}(\K_{(i_{k-(\ell+1)}-1) \m i}, \eps/(4(k+1))^{\ell+2}) \right.\\ & \left. \hspace{4cm}\times \mathbb E^*_{\St{i}{\C{i_{k-(\ell+1)}}{}{}}} \left[\prod_{j =0}^{\ell+1} \mathfrak f_{k + j-(\ell+1)}^{\eps/(4(k+1))^{\ell+1}}(S_j)\right] \prod_{j=i+1}^{i_{k-(\ell+1)}} \bar D_j(\eps/(4(k+1))^{\ell+1}) \right] 
\\ & - \delta  \left(\prod_{j = 0}^{\ell} \gamma_{k-j}\right) \mathbb E\left[\mathbb E^*_{\St{i}{\C{i_{k-(\ell+1)}}{}{}}} \left[\prod_{j =0}^{\ell+1} \mathfrak f_{k + j-(\ell+1)}^{\eps/(4(k+1))^{\ell+1}}(S_j)\right] \prod_{j=i+1}^{i_{k-(\ell+1)}} \bar D_j(\eps/(4(k+1))^{\ell+1}) \right]. 
\end{align*}
\end{linenomath}
Now, thanks to monotonicity, when we iterate this expression, we may do the following replacements in the procedure. First, for the term not involving $\delta$, any factors of type $\gamma_\ell(\eps', \eta, (d+1)(k+1) f_{\max})$ with $0 < \eps' < \eps$ may be bounded from below by $\gamma_k$. Thus, at the $(\ell+1)$th step, we multiply a product of $\gamma_{k}^{\ell+1}$ to the co-efficient of the main term, leading to the co-efficient $\gamma$ as defined in \eqref{gammadef}. Moreover, in the final product $\prod_{j =0}^{k} \mathfrak f_{j}^{\eps/(4(k+1))^{k}}(S_j)$, we may replace $\eps/(4(k+1))^{k}$ by $\eps$ to get a lower bound. 
 This leads to the first term in the statement of the proposition. Next, in the error term involving $\delta$, we bound each $\gamma_{\ell}$ from above by $1$, and bound each of the factors of the form $\mathfrak f_{k + j -\ell}^{\eps/(4(k+1))^{\ell}}$ from above by $\mathfrak f_{k + j - \ell}^{\eps/(4(k+1))^{k+1}}$. This gives us the error term as stated in Equation~\eqref{eq:the-prop}.
\end{proof}

We are finally ready to prove Proposition \ref{p2}. Recalling \eqref{start:lower}, we bound $\E{N_k^\eta(n)}$ from below by summing the lower bound stated in Proposition~\ref{prop:be_f} over $\eta n < i < i_1 < \ldots < i_k \leq n$.
We start with the error term. Upon dropping the indicator variables $\bar D_j(\eps)$ and bounding $\mathfrak f_j^\eps$ from above by $f_j$ defined in Equation~\eqref{def:ff} (in Subsection~\ref{subsec:upper}), the absolute value of the error term is bounded from above by 
\begin{equation} \label{eq:error-sum-lower}
\delta \sum_{\eta n < i < n} \sum_{\Ind_k \in {\{i+1, \ldots, n\} \choose k}}  \E{\mathbb E^*_{\St{i}{\C{i}{}{}}}  \left[\prod_{j=0}^k f_j (S_j)\right]}.
\end{equation}
From the proof of Corollary~\ref{cor:1} in Subsection~\ref{subsec:upper}, we know that the double sum converges after rescaling by $n$. Hence, there exist $C_1 > 0$ and a natural number  $N$ both depending on $\eps, \eta$, such that, for all $n \geq N$, \eqref{eq:error-sum-lower} is bounded from above by $C_1 \delta n $.

To treat the main term, assume for now that 
there exists a constant $C_2 = C_2(\varepsilon, \eta) > 0$ such that, for all $\eta n < i \leq n$, we have
\begin{equation} \label{lastbound22} \sum_{\cI_k \in {\{i+1, \ldots, n\} \choose k}} \mathbb E^*_{\St{i}{\C{i}{}{}}} \left[\prod_{j=0}^k \mathfrak f^\eps_j(S_j)\right] \leq C_2. 
\end{equation} 
We shall use the following inequality: for a non-negative 
random variable $X$ satisfying $X \leq C$, for some $C>0$, 
and indicator random variables $I_1,I_2$ we have 
\[\E{X} \leq \E{X I_1 I_2} + C(\E{1-I_1}+ \E{1-I_2}).\]
Thanks to this inequality, the main term in the lower bound from Proposition \ref{prop:be_f} summed over $i < i_1 < \ldots < i_k \leq n$ (for fixed $\eta n < i \leq n$) can be bounded from below by
\begin{equation} \label{eq:last-display}
\gamma  \sum_{\cI_k \in {\{i+1, \ldots, n\} \choose k}} \E{\mathbb E^*_{\St{i}{\C{i}{}{}}} \left[\prod_{j=0}^k \mathfrak f^\eps_j(S_j)\right]} - C_2 \gamma \left(1-\E{\alpha_0\left(\C{i\m i}{}{}, \frac{\eps}{4^{k+1}} \right)} + 1 - \E{\bar{D}_i \left (\frac{\eps}{4^{k+1}}\right)}\right).
\end{equation}
Let $\delta'> 0$. Thanks to Lemma \ref{lem:help1} and the fact that $\Prob{\mathcal{G}^{(i)}_{\eps/(4(k+1))^{k+1}}(i)} \to 1$ as $n\to \infty$ uniformly 
in $\eta n < i \leq n$, there exists a natural number $N = N(\delta', \varepsilon, \eta) > 0$  such that, for all $n \geq N$,  the absolute value of the second term in~\eqref{eq:last-display} is bounded from above by $C_2 \gamma \delta' \leq C_2 \delta'$. Collecting all bounds and using Lemma \ref{simple_stir} concludes the proof of Equation~\eqref{eq:prop-p2} upon setting $\varrho = \gamma$. (Note that we may remove the additional $F(S_j)$ in the denominator of $\mathfrak f^\eps_\ell(S_j)$ in the final statement as $F(S_j)$ is bounded by $(k+1)(d+1)f_{\max}$.) 
\\
Therefore, it remains to establish the existence of $C_2$ satisfying \eqref{lastbound22}.
To this end,  we shall bound  $\mathfrak f_j^\eps$ from above by $f_j$ (as defined in Equation~\eqref{def:ff}). Note that if $i \geq 2$, then $\frac{1}{i-1} \leq \frac{2}{\eta n}$. Thus, by applying Stirling's formula and recalling that $F(S_\ell) \leq (d+1)(k+1)f_{\max}$ for all $\ell \in \{0, \ldots, k\}$, we have
\begin{linenomath}
\begin{align*}
 \sum_{\cI_k \in {\{i+1, \ldots, n\} \choose k}} \prod_{j=0}^k f_j(S_j) 
& \leq \left(1+ O \left(\frac 1 n \right)\right) \sum_{i < i_1 < \ldots < i_k \leq n} \prod_{\ell=0}^{k-1} \left(\left(\frac{i_\ell}{i_{\ell+1}}\right)^{\frac{F(S_\ell)}{\lambda_{+\eps}}} \cdot \frac{F(S_\ell)}{\lambda_{-\eps}(i_{\ell+1}-1)}\right) \left(\frac{i_k}{n}\right)^{\frac{F(S_k)}{\lambda_{+\eps}}}  \\
& \leq  \frac{2\prod_{\ell=0}^{k-1} F(S_\ell)}{\lambda_{-\eps}\eta} \left(1+ O \left(\frac 1 n \right)\right) \times \\ 
& \ \ \ \frac1n\sum_{\eta n < i_0 < \ldots < i_{k-1} \leq n} \prod_{\ell=0}^{k-2} \left(\left(\frac{i_\ell}{i_{\ell+1}}\right)^{\frac{F(S_{\ell+1})}{\lambda_{+\eps}}} \cdot \frac{1}{\lambda_{-\eps}(i_{\ell+1}-1)}\right) \left(\frac{i_{k-1}}{n}\right)^{\frac{F(S_k)}{\lambda_{+\eps}}}, 
\end{align*}
\end{linenomath}
where the $O$-term depends only on $\eta$.
From Corollary \ref{cor:summation} (applied with $k-1$ instead of $k$) it follows that the right hand side is uniformly bounded for any $\eps$ and $\eta$.

We conclude the section with the proofs of Lemmas \ref{lem:help1} and \ref{lem:help2}.
\begin{proof}[Proof of Lemma \ref{lem:help1}]
Let $i \in \N$ and $\mathcal X \in \textbf C^w$ contain a vertex with label $i$ and at most $d$ active faces containing $i$, where each $(d-1)$-face containing $i$ has fitness at most $f_{\max}$.
In the random dynamical process $\C{j}{}{}, j \geq 0$ initiated with complex $\mathcal X$, at time $j \geq 1$, to each face $\sigma  \in \C{j}{}{d-1}$, 
we can associate a unique ancestral $(d-1)$-dimensional face in $\mathcal X$.
(Formally, the ancestral face of a face in $\mathcal X$ is the face itself. The ancestral face of any other face $\sigma$ is defined recursively as the ancestral face of the face which was subdivided when $\sigma$ was formed.)
Let $\C{j\not\downarrow i}{}{} \subseteq \C{j}{}{}$ be the subcomplex of faces of $\C{j}{}{}$ whose ancestral face does 
not lie in $\St{i}{\mathcal X}$.  Note that $\C{j\not\downarrow i}{}{} \subseteq 
\C{j \m i}{}{}$ and that this inclusion is typically strict due to migration of faces to the outside of the star at times of insertion in the star.  For $j \geq 1$, let $\varsigma_j$ be $j$-th time the face chosen in the construction of the simplicial complex has its ancestral face in $\mathcal X_{\m i}$. 
Set $\s_0 = 0$. Note that $\s_j \geq j$ and that $\s_j - j$ is non-decreasing in $j$.  The crucial observation is that the sequence $\C{\s_j \not\downarrow i}{}{}, j \geq 0$ under
$\mathbb P_{\mathcal X}$ is distributed as the sequence $\C{j}{}{}, j \geq 0$ under $\mathbb P_{\mathcal X_{\m i}}$ (upon disregarding vertex labels which are irrelevant here). 
 Formally, this follows from $\C{\s_0 \not\downarrow i}{}{} = \mathcal X_{\m i}$ under $\mathbb P_{\mathcal X}$ and the fact that
$\C{\s_j \not\downarrow i}{}{}, j \geq 0$ is Markovian with the same transition rule as $\C{j}{}{}, j \geq 0$. 
For an integer $K > 0$, on the event $\s_\ell \leq \ell + K$ and for any initial configuration $\mathcal X$ as described at the beginning of the proof, we have
$|F(\C{\ell}{}{}) -  F(\C{\s_{\ell} \not\downarrow i}{}{})| \leq (2d+1) K  f_{\max}$. Hence, for all $n$ sufficiently large (depending on $\eps, \eta$ and $K$), 
\begin{linenomath}
\begin{align*} 
\E{\mathbb P_{\C{i \m i}{}{}} \left(\bigcap_{j=i+1}^n \mathcal G_\eps(j-i;j)\}\right)}
&\geq \E{ \mathbb P_{\K_i}\left( \bigcap_{j=i+1}^n \{|F( \C{\s_{j-i} \not\downarrow i}{}{}) - \lambda j| < \eps \lambda j \} \right) \cdot \mathbf 1_{|\s_{n-i} - (n-i)| \leq K}} \\ 
& \geq \E{ \mathbb P_{\K_i}\left(\bigcap_{j=i+1}^{n+K} \mathcal G_{\eps/2}(j-i;j)\right)} - 
\E{ \mathbb P_{\K_i} (|\s_{n-i} - (n-i)| > K)} \\
& \geq 
\E{\mathbb{P}_{\K_i}\left(\bigcap_{j=i+1}^{\infty}  \mathcal G_{\eps/2}(j)\right)} - \E{ \mathbb P_{\K_i} (|\s_{n-i} - (n-i)| > K)}.
\end{align*}
\end{linenomath}
By Proposition~\ref{prop:partition}, for all $n$ sufficiently large, the first term in the last display is at least $1-\delta/2$ for all $\eta n < i \leq n$. Further, we can choose 
 $K$ large enough, such that the absolute value of the second term is bounded from above by
$\delta/2$ for all $\eta n < i \leq n$ and all $n$ sufficiently large. To see this, note that $\mathbb P_{x} (|\s_n - n| \geq K)$ is the probability that the number of faces with ancestral face in $\St{i}{x}$ chosen to be subdivided up to time $n$ exceeds $K$. Let $1 \leq \tau_1 < \tau_2 < \ldots$ be the instances, when such faces are chosen. Then, the sought after quantity equals $\mathbb P_x(\tau_K \leq n)$. 
Note that $\tau_K$ can be bounded from below stochastically by $X_1 + \cdots + X_K$ for independent summands, where $X_\ell$ follows the geometric distribution with success parameter $\min((d+1)\ell f_{\max} / F(x),1)$, which implies that $\E{X_1+\cdots +X_K}\geq F(x) \frac{\log K}{(d+1) f_{\max}}$.
Thus, if $F(x) \geq \lambda \eta n  / 2$, then, for a given $\eps' > 0$, for any $K$ large enough (depending on $\eta$), and all $n$ sufficiently large (depending on $\eps', \eta$ and $K$) we have $\mathbb P_x(\tau_K \leq n) \leq \eps'$ for all $n \geq 1$. This follows from a straightforward application of Chebychev's inequality, whose details we omit.  The fact that $F(\mathcal K_i) \geq \lambda \eta n  / 2$ (since $i \geq \eta n$) with high probability for sufficiently large $n$ (depending on $\eta$) concludes the proof of the lemma. 
\end{proof}

\begin{proof}[Proof of Lemma \ref{lem:help2}]
The proof is very similar to the previous. Let $\C{j\downarrow \mathcal X}{}{}$ be the sub-complex of $\mathcal K_j$ of faces whose ancestral face lies in $\mathcal X$. For $j \geq 1$, let $\s^{\mathcal X}_j$ be the $j$th time a face with ancestral face in $\mathcal X$ is subdivided. Set $\s_0^{\mathcal X}=0$. As before, we have $\s^{\mathcal X}_j \geq j$ and $\s^{\mathcal X}_j - j$ is non-decreasing. Define $\C{j\downarrow \mathcal Y}{}{}$ and $\s_j^{\mathcal Y}$ analogously. Thanks to (ii), under $\mathbb P_{\mathcal X}$, the sequence $\C{\s_j^{\mathcal Y}\downarrow \mathcal Y}{}{}, j \geq 0$ is distributed as $\C{\s_j^{\mathcal X}\downarrow \mathcal X}{}{}, j \geq 0$ under $\mathbb P_{\mathcal Y}$. Thus, it is enough to show that, 
under the conditions (i) - (iv), for sufficiently large $n$, we have
\[\mathbb P_{\mathcal Y}\left( \bigcap_{j=u+1}^m \mathcal G_{\eps_2}(j-u,j) \right) - \eps_3/2 \leq  \mathbb P_{ \mathcal Y} \left( \bigcap_{j=u+1}^m \{| F(\C{\s^{\mathcal X}_{j-u} \downarrow \mathcal X}{}{}) - \lambda j| < 3 \eps_2 j / 2 \} \right)\]
and 
\[\mathbb P_{\mathcal X} \left( \bigcap_{j=u+1}^m \{| F(\C{\s^{\mathcal Y}_{j-u} \downarrow \mathcal Y}{}{}) - \lambda j| < 3 \eps_2 j / 2\} \right)
\leq  \mathbb P_{\mathcal X}\left( \bigcap_{j=u+1}^m  \mathcal G_{2 \eps_2}(j-u,j) \right) + \eps_3 / 2.
\]
We only show the second statement, as the first can be proved by similar arguments. Note that, for any natural number $K$,  we have
\begin{linenomath}
\begin{align*}
& \mathbb P_{\mathcal X}\left( \bigcap_{j=u+1}^m \{|F(\C{\s^{\mathcal Y}_{j-u} \downarrow \mathcal Y}{}{}) - \lambda j| < 3 \eps_2 \lambda j / 2\} \right)  \\
& \leq \sum_{p=0}^K \mathbb P_{\mathcal X} \left( \bigcap_{j=u+1}^m \{|F(\C{\s^{\mathcal Y}_{j-u} \downarrow \mathcal Y}{}{}) - \lambda j| < 3 \eps_2 \lambda j / 2, \s^{\mathcal Y}_{n-u} = n -u +p\}  \right) + \mathbb P_{\mathcal X} (| \s^{\mathcal Y}_{n-u} - (n -u)| \geq K).
\end{align*}
\end{linenomath}
On $\s^{\mathcal Y}_{n-u} = n -u +p$, $0 \leq p \leq K$, 
we have, using (i) and (iii),
\[|F(\C{\s^{\mathcal Y}_{j-u} \downarrow \mathcal Y}{}{}) - F(\C{j-u}{}{})| \leq K (d+1)  f_{\max} + F\left(\mathcal X^{(d-1)} \triangle \mathcal Y^{(d-1)}\right) \leq K (d+1)  f_{\max} + C_1 C_2.\]
(Here, $F\left(\mathcal X^{(d-1)} \triangle \mathcal Y^{(d-1)}\right)$ denotes the sum of all fitneses of faces in $\mathcal X^{(d-1)} \triangle \mathcal Y^{(d-1)}$.) 
Thus, for all $n$ sufficiently large (depending on $\eta, \eps_2$ and $K$), we can bound the right hand side of the last display from above by
\begin{linenomath}
\begin{align*}
& \sum_{p=0}^K \mathbb P_{\mathcal X} \left( \bigcap_{j=u+1}^{m+p} \mathcal G_{2 \eps_2}(j-u, j) \cap \{\s^{\mathcal Y}_{n-u} = n - u +p\}  \right)
+ \mathbb P_{\mathcal X}( |\s^{\mathcal Y}_{n-u} - (n -i)| \geq K) \\
& \leq  \mathbb P_{\mathcal X} \left( \bigcap_{j=u+1}^{m} \mathcal G_{2 \eps_2}(j-u, j) \right)
+ \mathbb P_{\mathcal X}( |\s^{\mathcal Y}_{n-u} - (n -u)| \geq K).
\end{align*}
\end{linenomath}
Now, the same arguments relying on a stochastic bound involving sums of independent geometric random variables used in the previous proof show that the second summand can be made smaller than $\eps_3/2$ for sufficiently large (but fixed) $K$ and all $n$ sufficiently large (depending on $\eta$, $\eps_1,\eps_3, C_1$ and $C_2$). Here, one uses (iv) and the fact that
$F(\mathcal X^{(d-1)} \triangle \mathcal Y^{(d-1)}) \leq C_1 C_2$ to bound the success probabilities of the geometric random variables suitably.
\end{proof}

\section{Appendix}

\subsection{Proof of Lemma \ref{lem:couplingXN}}
\label{ap:couplingb2}
For brevity, we omit the superscript $\eps$ when referring to the process $N^\eps$, and in the notation of other parameters depending on $\eps$. 

Let $\varepsilon > 0$ be small enough such that $\varpi > \vartheta$ (this is possible because $\mu$ does not contain an atom at $1$). Then, $i \varpi + (d-i)\vartheta \leq 1$ for $i \in \{1, \ldots, d\}$. Let $\theta_i = 1 - i \varpi - (d-i)\vartheta, i \in \{0, \ldots, d\}$. The Markov chain $N$ has the following dynamics:
jump times are exponentially distributed with unit mean while the skeleton process performs a random walk on $\{0, \ldots, d\}$ according to the following rules: the process is absorbed
at $0$ and, given that its current state is $i \in \{1, \ldots, d\}$, it moves to $i+1$ with probability $(d-i)\vartheta$ and to $i-1$ with probability $i\varpi$, while it remains at $i$ with probability $\theta_i$. 

We construct the process $N$ from a realisation  of $X$. First, we use the jump times $\sigma_n, n \geq 1$ of the
$X$-process for the jump times of $N$. We define $N_{\sigma_n}$ by induction, starting with $N_{\sigma_0}=\mathscr C (X_{\sigma_0})$, where $\sigma_0 :=0$.
Let $n \geq 1$ and suppose $X_{\sigma_{n-1}} = \textbf{x}$ and $\mathscr C(X_{\sigma_{n-1}}) = \textbf{j}$ (recalling that $\mathscr C(\varnothing) = 0$).
 If $0 \leq \textbf{j} < N_{\sigma_{n-1}}$, then choose $N_{\sigma_n}$ arbitrarily obeying the dynamics of the random walk (for example \, by using additional external randomness).
 If $N_{\sigma_{n-1}} = 0$, set $N_{\sigma_n} = 0$. Finally, assume that $N_{\sigma_{n-1}} = \textbf{j} >0$. 
 Let \[s^\uparrow = \sum_{i=0}^{d-1-\textbf{j}} \frac{\E{f(\textbf{x}_{i \lf W})1_{W > 1 - \varepsilon}}}{M} \leq (d-\textbf{j})\vartheta, \quad s_\downarrow = \sum_{i=d-\textbf{j}}^{d-1} \frac{\E{f(\textbf{x}_{i \lf W})1_{W \leq 1 - \varepsilon}}}{M} \geq \textbf{j} \varpi.\]
 Let $A$ be an event that has probability
 $\textbf{j}\varpi /s_\downarrow \in [0,1]$ which is independent of the past of the process given $X_{\sigma_{n-1}}$.\footnote{For example $A = \left\{U \in [0, \textbf{j}\varpi /s_\downarrow ]\right\}$ for an independent uniformly distributed random variable $U$.}
  Let \[E = \{X_{\sigma_n} = \varnothing\} \cup (\{\mathscr C(X_{\sigma_{n}}) = \mathscr C(X_{\sigma_{n-1}}) - 1\}  \cap A^c) \cup \{\mathscr C(X_{\sigma_{n}}) = \mathscr C(X_{\sigma_{n-1}})\}.\]
 We first define $N(\sigma_n)$ on $E^c$ as follows: we set 
 \[
N_{\sigma_n}  = \begin{cases}  N_{\sigma_{n-1}}  + 1 & \text{on } \{\mathscr C(X_{\sigma_{n}}) = \mathscr C(X_{\sigma_{n-1}}) + 1\}, \\
N_{\sigma_{n-1}} - 1 & \text{on }  \{\mathscr C(X_{\sigma_{n}}) = \mathscr C(X_{\sigma_{n-1}}) - 1\} \cap \{X_{\sigma_n} \neq \varnothing\} \cap A.
 \end{cases}
 \]
Provided that $N_{\sigma_n} \in \{N_{\sigma_{n-1}}, N_{\sigma_{n-1}} + 1\}$ on $E$, this
guarantees that $\mathscr C(X_{\sigma_n}) \leq N_{\sigma_n}$. Finally, we ensure that the coupling respects the dynamics of the process $N$ by using additional randomness where required.
For example, we can proceed as follows: let $B$ be an event that has probability $((d-\textbf{j})\vartheta - s^\uparrow) / (1-s^\uparrow - \textbf{j}\varpi)$ which is independent of the past of the process given $X_{\sigma_{n-1}}$ (note that the denominator in the last expression is the probability of the event $E$ given  $X_{\sigma_{n-1}} = \textbf{x}$). Then, set $N_{\sigma_n} = N_{\sigma_{n-1}}+1$ on 
$B \cap E$ and $N_{\sigma_n} = N_{\sigma_{n-1}}$ on $B^c \cap E$.
 By construction, we have $\mathscr C(X_t) \leq N_t$ for all $t \leq \tau_L \wedge \tau_\varnothing$.

\subsection{Proof of Lemma \ref{lem:couplingb3}}
\label{ap:couplingb3}
First note that since both $X^{\sss (x)}$ and $X^{\sss (y)}$ jump at rate one, we can couple them so that they jump at the same times, which we denote by $(\sigma_i)_{i \in \mathbb{N}}$. 
At the first jump, for any measurable set $A \subseteq \cCd$ we should have
 \[\mathbb P(X^{\sss (x)}_{\sigma_1}\in A)
 = \frac1M \sum_{i=0}^{d-1} \mathbb E\big[f(x_{i\leftarrow W})\mathbf{1}_{A} (x_{i\leftarrow W})\big]; \; \;
 \mathbb P(X^{\sss (y)}_{\sigma_1}\in A)
 = \frac1M \sum_{i=0}^{d-1} \mathbb E\big[f(y_{i\leftarrow W})\mathbf{1}_{A} (y_{i\leftarrow W})\big],\]
 and both processes jump to $\{\varnothing\}$ with probability equal to the remaining mass.
We can interpret these measures as sums of $d+1$ measures given by
$\left(\frac1M \mathbb E\big[f(x_{i\leftarrow W})\delta_{x_{i\leftarrow W}}(\cdot)\big]\right)_{0\leq i\leq d-1}$
and $c(x)\delta_{\varnothing}(\cdot)$, where 
$c(x) := 1-\sum_{i=0}^{d-1}  
\mathbb E\big[f(x_{i\leftarrow W})\big]/M$,
for $X^{\sss (x)}$; similarly for $X^{\sss (y)}$.
On Figure~\ref{fig:coupling}, we draw the unit interval vertically and divide it in sub-intervals of respective lengths $\mathbb E\big[f(y_{i\leftarrow W})\big]/M$. On each of these intervals, we draw, from bottom to top as $i$ increases from 0 to $d-1$,
\[F_i^{\sss (x)} \colon u\mapsto b_{i} + \int_{[0,u]} f(x_{i\leftarrow v}) \mathrm d\mu(v)/M 
\quad\left(\text{ resp.\ }
F_i^{\sss (y)}\colon u\mapsto b_{i} + \int_{[0,u]} f(y_{i\leftarrow v}) \mathrm d\mu(v)/M\right)
\] in orange (resp.\ purple), where for $i\in\{0, \ldots, d-1\}, \, b_i = \sum_{j=0}^{i-1} \mathbb E\big[f(y_{j\leftarrow W})\big]/M.$
Note that, by monotonicity of $f$, both $F_i^{\sss (x)}$ and $F_i^{\sss (y)}$ are non-decreasing, and since $x\preccurlyeq y$, $F_i^{\sss (x)}\leq F_i^{\sss (y)}$ pointwise.

\begin{figure}
\begin{center}
\includegraphics[width=12cm]{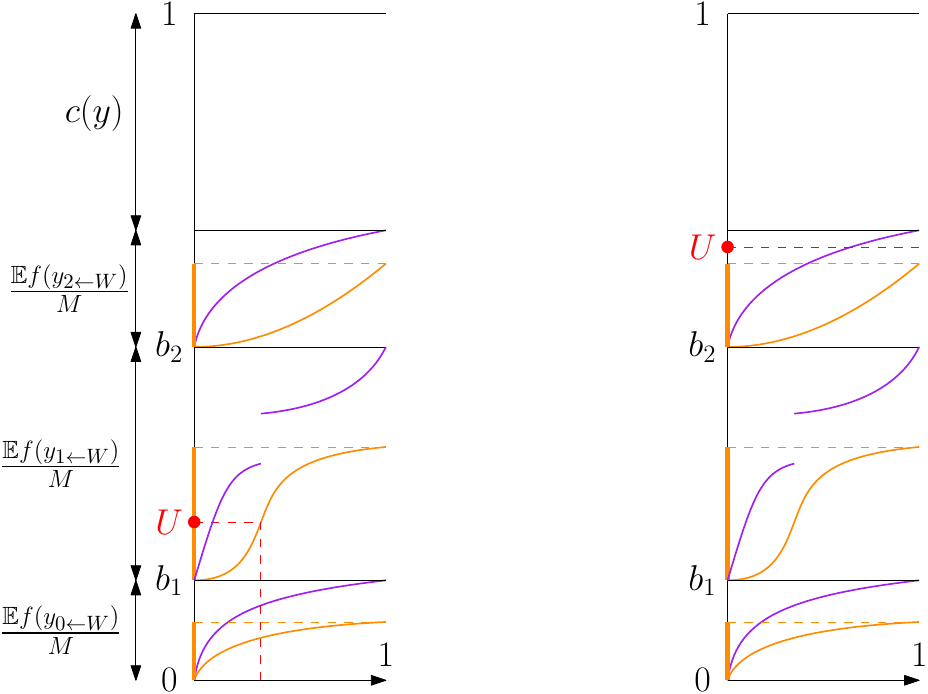}
\end{center}
\caption{Visual aid for the proof of Lemma~\ref{lem:couplingb3}. For the sake of presentation, we have chosen $d=3$.}
\label{fig:coupling}
\end{figure}

Now, consider a uniformly distributed random variable 
$U$ on $[0,1]$. 
If $U$ lands in the top-most interval (that is, if $U \geq \sum_{i=0}^{d-1} \E{f(y_{i \lf W})}$),
then we set $X^{\sss (x)}_{\sigma_1} = X^{\sss (y)}_{\sigma_1} = \varnothing$. If $U$ lands in the $i$-th interval (numbered from the bottom of the picture), we 
consider two cases:
\begin{itemize}
\item If $U$ lands into the orange part of the $i$-th interval (see left-hand-side of Figure~\ref{fig:coupling}), we set $X^{\sss (x)}_{\sigma_1} =x_{i\leftarrow(F_i^{\sss (x)})^{-1}(U)}$ 
and $X^{\sss (y)}_{\sigma_1}= y_{i\leftarrow (F_i^{\sss (x)})^{-1}(U)}$ (if $F_i^{\sss (x)}$ is not strictly increasing, we choose the left-continuous version of the inverse $(F_i^{\sss (x)})^{-1}(w)
:= \inf\{y\in[0,1]\colon F_i^{\sss (x)}(y)\geq w\}$).
\item If $U$ lands in the rest of the $i$-th interval (right-hand-side example on Figure~\ref{fig:coupling}), we set $X^{\sss (x)}_{\sigma_1}=\varnothing$. Set $G_i = F_i^{\sss (y)}-F_i^{\sss (x)}$ and note that this function is non-negative on $[0,1]$ and non-decreasing. Indeed, for all $u<v$, we have
\[G_i(v)-G_i(u) = \int_{(u,v]} \big(f(y_{i\leftarrow w}) - f(x_{i\leftarrow w})\big) \mathrm d\mu(w)/M \geq 0.\]
We can thus define the left-continuous inverse 
$G_i^{-1}(w) := \inf\{y\in[0,1]\colon G_i^{\sss (x)}(y)\geq w\}$, 
and set $X^{\sss (y)}_{\sigma_1} = y_{i\leftarrow G_i^{-1}(U - F_i^{\sss (x)}(1))}$.
\end{itemize}

Let us prove that, with these definition, $X_{\sigma_1}^{\sss (x)}$ and $X_{\sigma_1}^{\sss (y)}$ have the correct distributions and that $X_{\sigma_1}^{(x)} \preccurlyeq X_{\sigma_1}^{(y)}$. 
First note that, if $X_{\sigma_1}^{\sss (y)}=\varnothing$, then $U$ fell into the topmost interval and thus $X_{\sigma_1}^{\sss (x)}=\varnothing$, hence $X_{\sigma_1}^{(x)} \preccurlyeq X_{\sigma_1}^{(y)}$. If $X_{\sigma_1}^{\sss (x)}\neq \varnothing$, then $U$ fell in the orange part of an interval and thus $X_{\sigma_1}^{\sss (x)}=x_{i\leftarrow V}\preccurlyeq y_{i\leftarrow V}=X_{\sigma_1}^{\sss (y)}$ (where $V = (F_i^{\sss (x)})^{-1}(U)$), since $x \preccurlyeq y$.\\
Let us now check that $X_{\sigma_1}^{\sss (x)}$ defined in the coupling above has the right distribution. It is equal to $\varnothing$ if and only if $U$ landed in the topmost interval, or it did not land in an orange sub-interval, and thus 
\begin{linenomath}
\begin{align*}
&\mathbb P(X_{\sigma_1}^{\sss (x)} = \varnothing)
= c(y) + \sum_{i=0}^{d-1} \big(F_i^{\sss (y)}(1)-F_i^{\sss (x)}(1)\big)\\
& = 1- \frac1M \sum_{i=0}^{d-1} \mathbb E [f(y_{i\leftarrow W})]
+ \frac1M \sum_{i=0}^{d-1} \int_{[0,1]} f(y_{i\leftarrow v})\mathrm d\mu(v)
- \frac1M \sum_{i=0}^{d-1} \int_{[0,1]} f(x_{i\leftarrow v})\mathrm d\mu(v)\\
&= 1- \frac1M \sum_{i=0}^{d-1} \mathbb E [f(x_{i\leftarrow W})] = c(x).
\end{align*}
\end{linenomath}
For all Borel sets $A\subseteq \cCd$, we have
\begin{linenomath}
\begin{align*}
\mathbb P(X_{\sigma_1}^{\sss (x)}\in A)
&=\sum_{i=0}^{d-1} \mathbb P(X_{\sigma_1}^{\sss (x)}\in A \text{ and } F_i^{\sss (x)}(0) \leq U \leq F_i^{\sss (x)}(1))\\
&=\sum_{i=0}^{d-1} \int_{F_i^{\sss (x)}(0)}^{F_i^{\sss (x)}(1)} \bs 1_{A}\left(x_{i\leftarrow (F_i^{\sss (x)})^{-1}(u)}\right) \,\mathrm du\\
&= \sum_{i=0}^{d-1} \int_{[0,1]} \bs 1_{A}(x_{i\leftarrow v}) f(x_{i\leftarrow v}) \mathrm d\mu(v)/M,
\end{align*}
\end{linenomath}
by definition of $F_i^{\sss (x)}$ and by the change of variable $u=F_i^{\sss (x)}(v)$. This proves the claim.

Let us now check that $X^{\sss (y)}_{\sigma_1}$ also has the right distribution under the coupling. First note that $\mathbb P(X_{\sigma_1}^{\sss (y)} = \varnothing)$ is equal to the probability that $U$ lands in the topmost interval, which is of length $c(y)$, and thus $\mathbb P(X_{\sigma_1}^{\sss (y)} = \varnothing) = c(y)$. 

For all Borel sets $A\subseteq \cCd$, we have
\begin{linenomath}
\begin{align*}
\mathbb P(X_{\sigma_1}^{\sss (y)}\in A)
=& \sum_{i=0}^{d-1} \mathbb P(X_{\sigma_1}^{\sss (y)}\in A \text{ and } F_i^{\sss (x)}(0) \leq U \leq F_i^{\sss (x)}(1))\\
& + \sum_{i=0}^{d-1} \mathbb P(X_{\sigma_1}^{\sss (y)}\in A \text{ and }F_i^{\sss (x)}(1)<U\leq F_i^{\sss (y)}(1)).
\end{align*}
\end{linenomath}
The first sum is similar to the calculation above when checking the distribution of $X_{\sigma_1}^{\sss (x)}$:
\[\sum_{i=0}^{d-1} \mathbb P(X_{\sigma_1}^{\sss (y)}\in A \text{ and } F_i^{\sss (x)}(0) \leq U \leq F_i^{\sss (x)}(1))
=\frac1M\sum_{i=0}^{d-1}\mathbb E[f(x_{i\leftarrow W})\bs 1_{A} (y_{i\leftarrow W})].\]
For the second sum, we have
\begin{linenomath}
\begin{align*}
&\sum_{i=0}^{d-1} \mathbb P(X_{\sigma_1}^{\sss (y)}\in A \text{ and }F_i^{\sss (x)}(1)<U\leq F_i^{\sss (y)}(1))\\
&\hspace{1cm}= \sum_{i=0}^{d-1} \mathbb P(y_{i\leftarrow G_i^{-1}(U-F_i^{\sss (x)}(1))}\in A \text{ and }F_i^{\sss (x)}(1)<U\leq F_i^{\sss (y)}(1))\\
&\hspace{1cm}= \sum_{i=0}^{d-1} \int_{F_i^{\sss (x)}(1)}^{F_i^{\sss (y)}(1)} \bs 1_{A} \left(y_{i\leftarrow G_i^{-1}(u-F_i^{\sss (x)}(1))}\right)\,\mathrm du\\
&\hspace{1cm}=\sum_{i=0}^{d-1} \int_{[0,1]} \bs 1_{ A} \left(y_{i\leftarrow v}\right)\,(f(y_{i\leftarrow v})-f(x_{i\leftarrow v}))\mathrm d\mu(v)/M,
\end{align*}
\end{linenomath}
by definition of $G_i$ and by the change of variable $u=G_i(v)+F_i^{\sss (x)}(1)$.
We thus conclude that, in total, 
\[\mathbb P(X_{\sigma_1}^{\sss (y)}\in A)
=\frac1M\sum_{i=0}^{d-1}\mathbb E[f(y_{i\leftarrow W})\bs 1_{A}(y_{i\leftarrow W})],\]
as claimed. We can now iterate this coupling at each jump-time until $X^{\sss (x)}$ becomes absorbed. After $X^{\sss (x)}$ reaches $\varnothing$, we let $X^{\sss (y)}$ evolve independently according to its dynamics. This concludes the proof.

\subsection{Proof of Proposition \ref{help1}}
Let $\cC'_f \subseteq \cC'$ be the set of elements of the form $(z, \sum_{i=1}^{m} \delta_{y_i})$ for $z \geq 0, m \geq 1$ and $y_1, y_2, \ldots, y_m \in \cC_{d-2}$. Here, we view $\mathcal M(\cCdd)$ as a metric space under the Prokhorov metric, and view $\cC' = [0, \infty) \times \mathcal M(\cCdd)$ as a product metric space with $\infty$ product metric (where the distance is the maximum co-ordinate wise distance). 
First of all, we prove that there exists
a function $h : \cC'_f \times [0, 1] \times [0, \infty) \to \cC'_f $ such that, for independent and identically distributed random variables $(U_1, W_1), (W_2, U_2) \ldots$, where $U_i, W_i$ are independent, $U_i$ has the uniform distribution on $[0,1]$ and $W_i$ follows the distribution $\mu$  (as before),  we obtain a realisation of the Markov chain
starting at $x' \in \cC'_f$ by setting $S_0 = x'$ and, recursively, $S_{n+1} = h(S_n, U_{n+1}, W_{n+1})$ for $n \geq 0$. We then couple the two Markov chains started at $\varphi(w, x_n)$ and $\varphi(w, x)$ using the same sequence $(U_1,W_1),  (U_2, W_2), \ldots$, and write 
$S^{(n)}_0, S^{(n)}_1, \ldots$ and $S_0, S_1, \ldots$ for these chains. The construction of $h$ is straightforward. Let $x' = (z, \nu) \in \cC'_f$ with $\nu = \sum_{i=1}^{m} \delta_{y_i} \in \cC'_f$ and $u \in [0,1]$, $w' \geq 0$.  Order $y_1, \ldots, y_m$ lexicographically and define 
\begin{equation} \label{def:s} s_0 = 0 \text{ and } s_i = \sum_{j=1}^i f(y_j \cup z), 1 \leq i \leq m. \end{equation} Then, let $1 \leq p \leq m$ be such that
$s_{p-1} < u s_m \leq s_{p}$. We now set 
\[h((z, \nu), u, w') 
= \begin{cases} (z, \nu+ \sum_{i=0}^{d-2} \delta_{(y_p)_{i \lf w'}}), & \text{in Model }\textbf{A}, \\  
(z, \nu+ \sum_{i=0}^{d-2} \delta_{(y_p)_{i \lf w'}} - 
\delta_{y_p}), & \text{in Model }\textbf{B}.\end{cases}\]
It follows immediately from the dynamics of the Markov chain, that the function $h$ has the desired properties. 
Next, we show that, for the coupled Markov chains:
\begin{equation} \label{as:c}
\text{for any } k \geq 0, \text{ we have } S_k^{(n)} \to S_k \text{ almost surely.}
\end{equation}
By continuity of $f$, this implies that $F(S_k^{(n)}) \to F(S_k)$ almost surely, which concludes the proof. To prove \eqref{as:c}, we proceed by induction. The case $k=0$ is trivial as the function $\varphi$ is continuous. Assume that we have already proved the statement for all $j \in \{0, \ldots, k-1\}$. Recall that $S_k = h(S_{k-1}, U_k, W_k)$ and $S_k^{(n)} = h(S_{k-1}^{(n)}, U_k, W_k)$.  Conditioning on $S_{k-1}, S_{k-1}^{(0)}, S_{k-1}^{(1)}, \ldots$ shows that 
 \begin{linenomath}\begin{align*}
 &\Prob{S_k^{(n)} \nrightarrow S_k} \leq \mathbb E [\text{Leb}(\{u \in [0,1]: \text{ there exist } v_1, v_2, \ldots \in \cC'_f \text{ and } w' \geq 0 \\ 
 & \hspace{4cm} \text{ such that } \lim_{\ell\to\infty} v_\ell = S_{k-1} \text{ but } h(v_\ell, u, z) \nrightarrow h(S_{k-1}, u,z) \})] \end{align*}\end{linenomath}
We conclude the proof by showing that, almost surely, the set $u \in [0,1]$ for which  $v_\ell, \ell \geq 1$ and $w' \geq 0$ exist satisfying $v_\ell \to S_{k-1}$ as $\ell\to\infty$ and $h(v_\ell, u, w') \nrightarrow h(S_{k-1}, u, w')$ is a Lebesgue null set. To this end, we prove the following stronger statement: for $x' = (z, \sum_{i=1}^{m} \delta_{y_i}) \in \cC_f'$, we have that, for all $u \notin \{s_1/s_m, \ldots, 1\}$, where $s_1, \ldots, s_m$ are as in \eqref{def:s} for this particular $x'$, it holds that, for any sequence $x'_\ell \to x'$ and $w' \geq 0$, we have $h(x'_\ell, u, w') \to h(x', u, w')$. To see this, let $x_\ell' = (z_\ell, \sum_{i=1}^{m_\ell} \delta_{y_i^{(\ell)}})$ be a sequence with $x'_\ell \to x'$. This implies that $m_n = m$ for all sufficiently large $n$ and that $y_i^{(\ell)} \to y_i$ for all $1 \leq i \leq m$ as $\ell \to \infty$. By continuity of $f$, for the values $s_i^{(\ell)}$ defined in \eqref{def:s} for $x_\ell'$, we have $s_i^{(\ell)} \to s_i$ for all $1 \leq i \leq m$. Hence, if $u \notin \{s_1/s_m, \ldots, 1\}$, again using continuity, we have that $p^{(\ell)} = p$ for all $\ell$ sufficiently large and the desired result follows.

\subsection{Proof of Lemma \ref{lem:prob-sum}}\label{sub:proof_lem_prob_sum}
To prepare the proof of the lemma, we rewrite the relevant sums using probabilistic language. 
Let $U_0, \ldots, U_{k}$ be $k+1$ independent random variables uniformly distributed on $[0,1]$. We write $U_{(0)} \leq \ldots \leq U_{(k)}$ for their order statistics. Let
$I_j = \lceil U_{(j)} n \rceil, j\in \{0, \ldots, k\}$. Then, $I_n= (I_0, \ldots, I_k)$ is the vector of order statistics of $k+1$ independent random variables with uniform distribution on $\{1, \ldots, n\}$. Let $A_n$ be the event that these random variables are distinct. Then, for $\alpha_0, \ldots, \alpha_k \geq 0, 0 < \eta \leq 1/2$, we have 
\[
\Gamma_n(\alpha_0, \ldots, \alpha_k, \eta) 
= \frac{1}{(k+1)!} \cdot \E{ \prod_{j=0}^{k-1}\left( \left(\frac{I_{j}}{I_{j+1}} \right)^{\alpha_j} \cdot \frac{n}{I_{j+1} -1} \right) \left(\frac{I_{k}}{n} \right)^{\alpha_k}\mathbf 1_{A_n}\mathbf 1_{I_0 > \eta n}}. \]
Note that, given $U_{(i)}, U_{(i+1)}, \ldots, U_{(k)},$ the random variables $U_{(0)}, \ldots, U_{(i-1)}$ are distributed like the order statistics of $i$ independent random variables with the uniform distribution on $[0, U_{(i)}]$. Further, $U_{(k)}$ is distributed like $U^{1/(k+1)}$, where $U$ follows the uniform distribution on $[0,1]$. Thus,  setting $V_{i} = U_i^{1/(i+1)} U_{i+1}^{1/(i+2)}  \cdots U_k^{1/(k+1)},$ for $i \in \{0, \ldots, k\}$, the random vectors $(U_{(0)}, \ldots, U_{(k)})$ and  $(V_{0}, \ldots, V_{k})$ are equal in distribution. Therefore, by applying the dominated convergence theorem, for $\eta = 0$ we have
\[ \lim_{n \to \infty} \Gamma_n(\alpha_0, \ldots, \alpha_k, 0) = \frac{1}{(k+1)!} \cdot \E{ \prod_{j=0}^{k-1}\left( \left(\frac{U_{(j)}}{U_{(j+1)}} \right)^{\alpha_j} \cdot \frac{1}{U_{(j+1)}} \right) U_{(k)} ^{\alpha_k}}.\]
The last term is equal to
\begin{linenomath}\begin{align*}  
\frac{1}{(k+1)!} \cdot \E{\prod_{j=0}^{k-1} \left(\frac{V_{j}}{V_{j+1}} \right)^{\alpha_j} \cdot V_{k} ^{\alpha_k} \prod_{j=0}^{k-1}  \frac{1}{V_{j+1}}} & = \frac{1}{(k+1)!} \cdot \E{\prod_{j=0}^{k} U_j^{\alpha_{j}/(j+1)} \prod_{j=0}^k U_j^{-j/(j+1)} } \\ & =  \prod_{j=0}^{k} \frac{1}{\alpha_j + 1}.
 \end{align*}
\end{linenomath}
\begin{proof}[Proof of Lemma \ref{lem:prob-sum}]
We start with the term involving $\eta$. Note that 
$ \prod_{j=0}^{k-1} \frac{n}{I_{j+1} -1} \mathbf 1_{A_n} \leq 2  \prod_{j=0}^{k-1}  U_{(j+1)}^{-1}$, since on the event $A_n$, we have $I_1 \geq 2$. 
Thus, 
\begin{linenomath}
\begin{align*} & \E{ \prod_{j=0}^{k-1}\left( \left(\frac{I_{j}}{I_{j+1}} \right)^{\alpha_j} \cdot \frac{n}{I_{j+1} -1} \right) \left(\frac{I_{k}}{n} \right)^{\alpha_k}\mathbf 1_{A_n}\mathbf 1_{I_0 \leq \eta n}} \\
& \leq 2  \E{ \prod_{j=0}^{k-1} U_{(j+1)}^{-1}   \mathbf{1}_{I_{0} \leq \eta n}} \leq 2  \E{ \prod_{j=0}^{k-1} U_{(j+1)}^{-(k+2)/(k+1)}}^{(k+1)/(k+2)}   \Prob{I_{0} \leq \eta n}^{1/(k+2)} \\
& \leq 2 \left(k+1\right)^{(1 + k(k+1))/(k+2)} \eta^{1/(k+2)}.
\end{align*}
\end{linenomath}
Here, in the last step, we have used $\Prob{I_{0} \leq \eta n} \leq \Prob{U_{(0)} \leq \eta} = 1 - (1-\eta)^{k+1} \leq (k+1)\eta$.
Next, let  $\Delta_{j+1} = \frac{n}{I_{j+1} -1} - \frac{1}{U_{(j+1)}}$. 
In the computation of 
\[\E{ \prod_{j=0}^{k-1}\left( \left(\frac{I_{j}}{I_{j+1}} \right)^{\alpha_j} \cdot \frac{n}{I_{j+1} -1} \right) \left(\frac{I_{k}}{n} \right)^{\alpha_k}\mathbf 1_{A_n}},\]
we can now successively replace $\frac{n}{I_{j+1} -1}$ by $\frac{1}{U_{(j+1)}} + \Delta_{j+1}$ for $j \in \{0, \ldots, k-1\}$. As $\Delta_{j+1} \to 0$ almost surely, it follows from the dominated convergence theorem, that
\begin{linenomath}
\begin{align*} & \E{ \prod_{j=0}^{k-1}\left( \left(\frac{I_{j}}{I_{j+1}} \right)^{\alpha_j} \cdot \left( \frac{1}{U_{(j+1)}} + \Delta_{j+1}\right) \right) \left(\frac{I_{k}}{n}  \right)^{\alpha_k}\mathbf 1_{A_n}} \\ & =   \E{ \prod_{j=0}^{k-1}\left( \left(\frac{I_{j}}{I_{j+1}} \right)^{\alpha_j} \cdot \left( \frac{1}{U_{(j+1)}} \right) \right) \left(\frac{I_{k}}{n}  \right)^{\alpha_k}\mathbf 1_{A_n}} + o(1). \end{align*}
\end{linenomath}
As $\E{|\Delta_{j+1}| \mathbf 1_{U_{(0)} > 1/n}} = O(\log n / n)$, it follows easily that the convergence rate in the last display is $O(\log n / n)$.
Next, let $\Delta'_j = \frac{I_{j}}{I_{j+1}} -\frac{U_{(j)}}{U_{(j+1)}}$.
Note that, for any positive real numbers $x,y$, we have
\[
\frac{-y}{(x+1)x} \leq \frac{\lceil y \rceil}{\lceil x \rceil} - \frac{y}{x} \leq \frac{1}{x},
\]
and thus, on $A_{n}$
\[\Delta'_{j} \in [-(nU_{(j+1)})^{-1}, (nU_{(j+1)})^{-1}].\] 
Hence, by the mean value theorem, if $\alpha \geq 1$, for $j \in \{0, \ldots, k-1\}$, $\left| \left(\frac{I_{j}}{I_{j+1}}\right)^\alpha - \left(\frac{U_{(j)}}{U_{(j+1)}}\right)^\alpha \right|\leq \alpha / (n U_{(j+1)})$. 
 In the case that $\alpha < 1$, observe that \[\min\left( \frac{I_j}{I_{j+1}}, \frac{U_{(j)}}{U_{(j+1)}}\right) \geq \frac{n U_{(j)}}{nU_{(j+1)} + 1} \geq \frac{ U_{(j)}}{2U_{(j+1)}},\]
since $I_1 > 1$, and thus,
\[\max\left( \left(\frac{I_j}{I_{j+1}}\right)^{\alpha-1}, \left(\frac{U_{(j)}}{U_{(j+1)}}\right)^{\alpha-1}\right) \leq \left(\frac{ U_{(j)}}{2U_{(j+1)}}\right)^{\alpha -1} \leq  \frac{2U_{(j+1)}}{U_{(j)}}.\]
Thus, by a similar application of the mean value theorem, if $0 \leq \alpha \leq 1$, then, \[\left|\left(\frac{I_{j}}{I_{j+1}}\right)^\alpha - \left(\frac{U_{(j)}}{U_{(j+1)}}\right)^\alpha\right| \leq  2\alpha / (n U_{(j)}).\]

Now, for  $j \in \{0, \ldots, k\}$, we have 

 \[\E{U_{(j)}^{-1} \prod_{i=0}^{k-1}U_{(i+1)}^{-1}  \mathbf 1_{A_n} \mathbf 1_{I_0 > 1}} \leq 
\E{\prod_{i=0}^{k}U_{i}^{-1}  \mathbf 1_{U_{i} > n^{-i} }} = O(\log^{k+1}(n)).\]
Note that we only need $I_0 > 1$ when $\alpha < 1$, in order to ensure that  $U_{(0)} > 1/n$. 

Thus, successively replacing $\frac{I_{j}}{I_{j+1}}$ by $\frac{U_{(j)}}{U_{(j+1)}} + \Delta'_{j}$ shows
\begin{linenomath}
\begin{align*} & \E{ \prod_{j=0}^{k-1}\left( \left(\frac{I_{j}}{I_{j+1}} \right)^{\alpha_j} \cdot \left( \frac{1}{U_{(j+1)}} \right) \right) \left(\frac{I_{k}}{n}  \right)^{\alpha_k}\mathbf 1_{A_n} \mathbf 1_{I_0 > 1}} \\ & =   \E{ \prod_{j=0}^{k-1}  \left(\frac{U_{(j)}}{U_{(j+1)}} \right)^{\alpha_j} \cdot  \prod_{j=0}^{k-1} \frac{1}{U_{(j+1)}}   \left(\frac{I_{k}}{n}  \right)^{\alpha_k}\mathbf 1_{A_n} \mathbf{1}_{I_0 > 1}} + O\left( \frac{\sum_{j =0}^{k-1} \alpha_j \log^{k+1}(n)}{n}\right).
\end{align*}
\end{linenomath}
Replacing $I_k/n$ by $U_{(k)}$ gives rise to an error term of order at most $\alpha_k \log^{k+1}(n) / n$. As $\Prob{A_n^{c}} = O(1/n)$ and $\Prob{I_0 = 1} = O(1/n)$, an application of H{\"o}lder's inequality shows that we may drop the indicators $\mathbf 1_{A_n}$ and $\mathbf{1}_{I_0 > 1}$ at the cost of an error term of order $n^{-1/(k+2)}$.
\end{proof}

\bibliographystyle{imsart-nameyear}
\bibliography{simplices_large} 
\end{document}